\documentclass[reqno]{amsart}
\usepackage{amssymb}
\usepackage{graphicx}
\usepackage{subcaption}
\usepackage{tikz}
\usepackage{makecell}
\usepackage{empheq}
\usepackage{url}
\usepackage{float}
\usepackage{hyperref}

\newtheorem{theorem}{Theorem}[section]
\newtheorem{lemma}[theorem]{Lemma}
\newtheorem{proposition}[theorem]{Proposition}
\theoremstyle{definition}

\theoremstyle{remark}
\newtheorem{remark}[theorem]{Remark}

\numberwithin{equation}{section}

\newcommand{\R}{\mathbb{R}}

\def\d{\ensuremath{\mathrm{d}}}

\newcommand{\xx}{{\mathbf{x}}}
\newcommand{\yy}{{\mathbf{y}}}
\newcommand{\zz}{{\mathbf{z}}}

\newcommand{\RR}{{\mathbb{R}}}

\newenvironment{equationa*}{\begin{equation*}\begin{aligned}} {\end{aligned}\end{equation*}}

\begin{document}

\title[ND-TNN]
{ND-TNN: Tensor-Neural-Network Approximation for High-Dimensional Nonlocal Diffusion Models}

\author[]{Ziyue Cai}
\address{Ziyue Cai: Qiuzhen College, Tsinghua University,
Beijing, China, 100084.}
\curraddr{}
\email{cai-zy22@mails.tsinghua.edu.cn}

\author[]{Zuoqiang Shi}
\address{Zuoqiang Shi: Yau Mathematical Sciences Center, Tsinghua University,
Beijing, China, 100084. \&
Yanqi Lake Beijing Institute of Mathematical Sciences and Applications,
 Beijing, China, 101408.}
\email{zqshi@tsinghua.edu.cn}
\thanks{This work was supported by the National Natural Science Foundation of China (NSFC) under Grant 92370125.}





\begin{abstract}
We study a numerical method, built on the tensor neural network (TNN)
architecture introduced in \cite{wang2022tensor}, for solving nonlocal
diffusion models in high-dimensional spaces. The tensor-product structure of
the TNN ansatz, combined with the separability of the Gaussian kernel, reduces
the high-dimensional integrals in the nonlocal energy to products of
low-dimensional integrals, which are evaluated by Gauss--Legendre quadrature;
nonseparable source and boundary data are handled by a TNN-based
preconditioning step. For the Dirichlet boundary condition, we establish the
asymptotically compatible $L^2$ error estimate
\[
\|u_{\mathrm{loc}}-u_{\delta,p}\|_{L^2(\Omega)}
\le
C\!\left(\frac{\varepsilon_f}{\sqrt\delta}
+\frac{\varepsilon_g}{\delta}
+\frac{\varepsilon_u}{\sqrt\delta}
+\eta_{\mathrm{opt}}\right)
+C\sqrt\delta,
\]
where $\varepsilon_f$, $\varepsilon_g$ and $\varepsilon_u$ are the data and
trial-class approximation errors and $\eta_{\mathrm{opt}}$ is the
optimization residual. For the Neumann boundary condition, the $L^2$
estimate is improved to
$O(\varepsilon_f+\varepsilon_g/\sqrt\delta+\varepsilon_u
+\eta_{\mathrm{opt}}+\delta)$,
and an $H^1$ gradient estimate is further obtained through a smoothing
post-processing step. Numerical experiments on tensor-product domains up to
$d=20$ support the theoretical results, and additional tests on two- and
three-dimensional $L$-shaped domains demonstrate the practical robustness of
the method beyond the smooth-domain setting covered by the analysis.
\end{abstract}

\maketitle

\section{Introduction}
\label{sec:introduction}

Nonlocal diffusion models have attracted considerable attention in recent
decades as integral-operator counterparts of classical elliptic equations.
By replacing differential operators with integral operators that average
differences over a neighborhood of radius $\delta$, nonlocal models can
capture long-range interactions and singular phenomena that classical
partial differential equations struggle to describe. They have found
applications in diverse fields including fracture mechanics
\cite{littlewood2010simulation,ha2011characteristics,silling2000reformulation},
image processing
\cite{buades2010image,gilboa2007nonlocal,gilboa2009nonlocal,shi2017weighted},
fractional Laplacian
\cite{ainsworth2018towards,bahr2024implementation,bonito2018numerical},
multiscale modeling
\cite{abdulle2015reduced,askari2008peridynamics,du2019multiscale},
and machine learning
\cite{tao2018nonlocal,wang2018non,zhu2003semi}.
A central theoretical property is the asymptotic compatibility
\cite{du2019nonlocal}: under suitable assumptions on the kernel and the
data, the solution of the nonlocal model converges to the solution of the
corresponding local PDE as $\delta\to 0$. In this paper we focus on the
nonlocal diffusion model with a Dirichlet boundary condition
\cite{meng2023maximum}:
\begin{align}
&\frac{1}{\delta^{2}}
\int_\Omega R_\delta(\xx,\yy)
\bigl(u_\delta(\xx)-u_\delta(\yy)\bigr)\,\d\yy
+
\frac{2}{\delta}
\int_{\partial\Omega}
\bar{R}_\delta(\xx,\yy)\,u_\delta(\xx)\,\d S_\yy
\notag\\
&\qquad=
\int_\Omega\bar{R}_\delta(\xx,\yy)f(\yy)\,\d\yy
+
\frac{2}{\delta}
\int_{\partial\Omega}
\bar{R}_\delta(\xx,\yy)g(\yy)\,\d S_\yy,
\qquad\xx\in\Omega,
\label{eq:intro-dirichlet}
\end{align}
and the corresponding Neumann model; here
$R_\delta(\xx,\yy)=\alpha_d\delta^{-d}R(|\xx-\yy|^2/(4\delta^2))$ is the
rescaled kernel, $\bar R_\delta$ is its primitive, and $\delta>0$ is the
nonlocal horizon. As $\delta\to 0$, the solution $u_\delta$ of
\eqref{eq:intro-dirichlet} converges to the solution of the classical
elliptic equation $-\Delta u=f$ in $\Omega$ with $u=g$ on
$\partial\Omega$.

Many numerical methods have been proposed for nonlocal models, including
finite difference methods \cite{tian2013analysis,zhou2010mathematical}, finite element methods
\cite{chen2011continuous,du2013posteriori,meng2026asymptotically}, spectral methods
\cite{du2016asymptotically}, collocation methods
\cite{zhang2016nodal}, and point integral methods \cite{shi2015enforce,li2017point}. A common computational bottleneck is the
evaluation of the double integral over $\Omega\times\Omega$: in $d$ space
dimensions this is a $2d$-dimensional integral, and direct quadrature
rapidly becomes prohibitive as $d$ increases. For the Gaussian kernel
$R(r)=e^{-s^2 r}$, the rescaled kernel factorizes as
\begin{equation}
\label{eq:intro-factorization}
R_\delta(\xx,\yy)
\;\propto\;
\prod_{i=1}^d
\exp\!\left(-\frac{s^2}{4\delta^2}(x_i-y_i)^2\right),
\end{equation}
so that the $2d$-dimensional integral can be decoupled into a product of
$d$ two-dimensional integrals. This key observation was exploited in
\cite{meng2023maximum} to devise a fast implementation of the nonlocal
finite element method on tensor-product domains. The present work uses the same factorization idea together with the TNN
architecture of \cite{wang2022tensor}: the finite element trial space is
replaced by a \emph{tensor neural network} (TNN) trial class, which is
naturally suited to high-dimensional problems and does not require mesh
generation.

The TNN architecture of \cite{wang2022tensor} represents a $d$-variate
function as a finite sum of tensor products of one-dimensional subnetworks:
\begin{equation}
\label{eq:intro-tnn}
u(\xx;\Theta)
= c\sum_{j=1}^{p}\prod_{i=1}^d \phi_{i,j}(x_i;\theta_i),
\end{equation}
where each $\phi_{i,j}$ is a one-dimensional fully connected neural
network. This structure inherits the product structure of
\eqref{eq:intro-factorization} and allows every integral in the nonlocal
energy functional to be evaluated as a product of two-dimensional
integrals, regardless of the spatial dimension $d$. TNN-based methods have
been applied successfully to high-dimensional problems including the
Schr\"{o}dinger equation \cite{zhou2025sum}, eigenvalue problems
\cite{wang2024computing}, and time-fractional PDEs
\cite{lin2026solving}. However, their use for nonlocal models has, to our
knowledge, not been studied.

In this paper, we construct and analyze a variational method for nonlocal
diffusion models with Dirichlet and Neumann boundary conditions by using the
existing TNN ansatz as the trial class. Taking the
Dirichlet problem as an example, the nonlocal equation
\eqref{eq:intro-dirichlet} is equivalent to minimizing, over
$H^1(\Omega)$, the energy functional
\begin{align}
\mathcal{L}_\delta(u)
&= \frac{1}{4\delta^2}
\!\int_\Omega\!\!\int_\Omega\!
R_\delta(\xx,\yy)\bigl(u(\xx)-u(\yy)\bigr)^2\,\d\xx\,\d\yy
+ \frac{1}{\delta}
\!\int_\Omega\!\!\int_{\partial\Omega}\!
\bar R_\delta(\xx,\yy)\,u(\xx)^2\,\d\xx\,\d S_\yy
\notag\\
&\quad
- \!\int_\Omega\!\!\int_\Omega\!
\bar R_\delta(\xx,\yy)\,u(\xx)f(\yy)\,\d\xx\,\d\yy
- \frac{2}{\delta}
\!\int_\Omega\!\!\int_{\partial\Omega}\!
\bar R_\delta(\xx,\yy)\,u(\xx)g(\yy)\,\d\xx\,\d S_\yy .
\label{eq:intro-energy}
\end{align}
For general $f$ and $g$, the integrals in \eqref{eq:intro-energy} cannot be
separated directly. To overcome this, we first approximate $f$ and $g$ by
TNN surrogates $\tilde f$, $\tilde g$ via empirical $L^2$ loss
minimization, and then minimize the modified loss (with $f,g$ replaced by
$\tilde f, \tilde g$) over the TNN class \eqref{eq:intro-tnn}. Combining
\eqref{eq:intro-factorization} with \eqref{eq:intro-tnn}, every term in
this loss decomposes into a product of two-dimensional integrals that can
be evaluated efficiently by composite Gauss--Legendre quadrature.

On the theoretical side, we establish asymptotically compatible error
estimates for both boundary conditions. The total error is decomposed into
three contributions: the data preconditioning errors ($\varepsilon_f$,
$\varepsilon_g$), the best-approximation error of the TNN trial class
($\varepsilon_u$), and the optimization residual ($\eta_{\mathrm{opt}}$).
For the Dirichlet problem, we prove
\begin{equation}
\label{eq:intro-dirichlet-est}
\|u_\delta - u_{\delta,p}\|_{L^2(\Omega)}
\le C\!\left(\frac{\varepsilon_f}{\sqrt\delta}
           +\frac{\varepsilon_g}{\delta}
           +\frac{\varepsilon_u}{\sqrt\delta}
           +\eta_{\mathrm{opt}}\right),
\end{equation}
which, combined with the local-limit estimate
$\|u_{\mathrm{loc}}-u_\delta\|_{H^1(\Omega)}\le C\sqrt\delta$, gives the
asymptotically compatible bound for
$\|u_{\mathrm{loc}}-u_{\delta,p}\|_{L^2(\Omega)}$. For the Neumann
problem, the more favorable coercivity of the Neumann energy yields the
improved estimate
\begin{equation}
\label{eq:intro-neumann-est}
\|u_{\delta,N} - u_{\delta,p,N}\|_{L^2(\Omega)}
\le C\!\left(\varepsilon_f
           +\frac{\varepsilon_g}{\sqrt\delta}
           +\varepsilon_u
           +\eta_{\mathrm{opt}}\right),
\end{equation}
where $\varepsilon_f$ and $\varepsilon_u$ now appear without any negative
power of $\delta$. A gradient ($H^1$) estimate is further obtained via a
smoothing post-processing. In both cases, as $\varepsilon_f$,
$\varepsilon_g$, $\varepsilon_u$, $\eta_{\mathrm{opt}}\to 0$ and
$\delta\to 0$, the TNN output converges to the local PDE solution, so the
method is asymptotically compatible. For completeness, we also provide
detailed proofs, based on standard nonlocal energy estimates, of the
well-posedness and the local-limit estimate for the nonlocal Neumann model.

The theoretical results are supported by extensive numerical experiments.
For tensor-product data on $\Omega=[0,1]^d$ with $d=3,5,10,20$, the
method achieves small residuals in all cases, and both the $L^2$ and $H^1$
errors between the TNN output and the local solution exhibit a numerical
convergence rate of nearly order $1$ in $\delta$, which exceeds the
$O(\sqrt\delta)$ prediction of \eqref{eq:intro-dirichlet-est}; a rigorous
explanation of this superconvergence is left to future work. For
non-tensor-product data, where all three error components $\varepsilon_f$,
$\varepsilon_g$, $\varepsilon_u$ are nonzero, a hyperparameter study shows
that the TNN approximation error of the source term remains at the level of
$10^{-3}$ in relative RMSE for dimensions up to $d=20$, confirming that
the TNN class remains expressive in high dimensions under a fixed
per-dimension subnetwork  architecture. The method is further tested on two- and three-dimensional
$L$-shaped domains under Neumann boundary conditions, demonstrating practical
robustness beyond the smooth tensor-product setting covered by the theory.

The rest of the paper is organized as follows.
Section~\ref{sec:preliminaries} reviews the nonlocal models, the TNN
architecture, and basic notation. Section~\ref{sec:main-result} presents
the variational workflow and states the main error estimates.
Section~\ref{sec:proof} is devoted to the proofs of the main theorems.
Numerical experiments are reported in Section~\ref{sec:numerical-tests}.
Section~\ref{sec:conclusions} concludes the paper. Proofs of the
well-posedness and local-limit results for the Neumann model are given in
Appendices~\ref{appendix:neumann-wellposedness}--\ref{appendix:neumann-locallimit}.

\section{Preliminaries and Notation}
\label{sec:preliminaries}

\subsection{Nonlocal diffusion model with a Dirichlet boundary condition}
\label{subsec:nonlocal-dirichlet}

We first consider the elliptic problem with a Dirichlet boundary condition:
\begin{equation}
  \label{eq:dirichlet-local}
  \left\{
    \begin{aligned}
      -\Delta u(\xx)&=f(\xx), && \xx \in\Omega,\\
      u(\xx)&=g(\xx), && \xx\in \partial\Omega.
    \end{aligned}
  \right.
\end{equation}
Throughout the theoretical sections, $\Omega\subset\RR^d$ is a bounded connected domain with sufficiently smooth boundary . In this subsection we assume
$f \in H^1(\Omega)$ and $g \in H^{5/2}(\partial \Omega)$; the rationale for
these regularity assumptions is explained in
Remark~\ref{rem:dirichlet-regularity} below.

Following \cite{meng2023maximum}, the nonlocal counterpart of
\eqref{eq:dirichlet-local} reads
\begin{align}
&\frac{1}{\delta^{2}}
\int_\Omega R_\delta(\xx,\yy)
\bigl(u_\delta(\xx)-u_\delta(\yy)\bigr)\,\d\yy
+
\frac{2}{\delta}
\int_{\partial \Omega}
\bar{R}_\delta(\xx,\yy)\,u_\delta(\xx)\,\d S_\yy
\notag \\
&\qquad =
\int_\Omega \bar{R}_\delta(\xx,\yy)f(\yy)\,\d\yy
+
\frac{2}{\delta}
\int_{\partial \Omega}
\bar{R}_\delta(\xx,\yy)g(\yy)\,\d S_\yy .
\label{eq:dirichlet-nonlocal}
\end{align}

For the standard analytical statements recalled in this subsection, we impose the following assumptions on the kernel function $R$.
\begin{itemize}
\item[(i)] \textbf{Regularity:} $R\in C^1([0,+\infty))$.
\item[(ii)] \textbf{Positivity and compact support:} $R(r)\ge 0$ for all
$r\ge 0$, and $R(r)=0$ for all $r>1$.
\item[(iii)] \textbf{Non-degeneracy:} there exists $\gamma>0$ such that
$R(r)\ge \gamma$ for all $r\in[0,\tfrac12]$.
\end{itemize}
We define the primitive of $R$ by
\[
  \bar{R}(r)=\int_r^{+\infty}R(s)\,\d s .
\]
Clearly $\bar R$ is nonnegative, compactly supported, and non-degenerate on
$[0,\tfrac12]$.

For $\delta>0$, we define the rescaled kernels
\begin{equation}
  \label{eq:kernel-def}
  R_\delta(\xx,\yy)
  =
  \alpha_d\,\delta^{-d}
  R\!\left(\frac{|\xx-\yy|^2}{4\delta^2}\right),
  \qquad
  \bar{R}_\delta(\xx,\yy)
  =
  \alpha_d\,\delta^{-d}
  \bar{R}\!\left(\frac{|\xx-\yy|^2}{4\delta^2}\right),
\end{equation}
where the constant $\alpha_d$ is chosen so that
\[
    \int_{\RR^d}
    \alpha_d\,\delta^{-d}\,
    \bar{R}\!\left(\frac{|\xx-\yy|^2}{4\delta^2}\right)
    \,\d\yy
    =
    \alpha_d S_d
    \int_0^{2}\bar{R}(r^2/4)\,r^{d-1}\,\d r
    =
    1,
\]
and $S_d$ denotes the surface area of the unit sphere in $\RR^d$.

The well-posedness of \eqref{eq:dirichlet-nonlocal} and its local limit were
established in \cite{meng2023maximum}. We recall the relevant statements
below.

\begin{proposition}[Well-posedness of the nonlocal Dirichlet problem]
\label{prop:dirichlet-wellposedness}
For any $\delta>0$, $f\in L^2(\Omega)$ and $g\in L^2(\partial\Omega)$, the
problem \eqref{eq:dirichlet-nonlocal} admits a unique solution
$u_\delta\in H^1(\Omega)$. Moreover, there exists a constant $C>0$,
independent of $\delta$, such that
\begin{equation}
\label{eq:dirichlet-wellposedness-bound}
\|u_\delta\|_{H^1(\Omega)}
\le
C\!\left(\|f\|_{L^2(\Omega)}
+\frac{1}{\sqrt{\delta}}\,\|g\|_{L^2(\partial\Omega)}\right).
\end{equation}
\end{proposition}

\begin{proposition}[Local limit of the nonlocal Dirichlet problem]
\label{prop:dirichlet-locallimit}
Let $u_\delta$ be the solution of \eqref{eq:dirichlet-nonlocal}, and assume
that the local Dirichlet problem \eqref{eq:dirichlet-local} admits a solution
$u_{\mathrm{loc}}\in H^3(\Omega)$. Then, for all sufficiently small
$\delta>0$,
\begin{equation}
\label{eq:dirichlet-locallimit-bound}
\|u_{\mathrm{loc}}-u_\delta\|_{H^1(\Omega)}
\le
C\sqrt{\delta}\,\|u_{\mathrm{loc}}\|_{H^3(\Omega)},
\end{equation}
where $C>0$ is independent of $\delta$.
\end{proposition}

\begin{remark}
\label{rem:dirichlet-regularity}
The regularity assumptions on $f$ and $g$ are motivated by the standard
elliptic regularity theory; see, for instance, \cite{evans2022partial}. In
particular, if $f\in H^1(\Omega)$ and $g\in H^{5/2}(\partial\Omega)$, then
the solution of \eqref{eq:dirichlet-local} satisfies
\[
\|u_{\mathrm{loc}}\|_{H^3(\Omega)}
\le
C\!\left(\|f\|_{H^1(\Omega)}
+\|g\|_{H^{5/2}(\partial\Omega)}\right).
\]
In this case, Proposition~\ref{prop:dirichlet-locallimit} gives the
$H^1$-norm local-limit estimate at rate $O(\sqrt\delta)$, and consequently the
$L^2$ local-limit error satisfies
$\|u_{\mathrm{loc}}-u_\delta\|_{L^2(\Omega)}\le\|u_{\mathrm{loc}}-u_\delta\|_{H^1(\Omega)}
\le C\sqrt\delta\,\|u_{\mathrm{loc}}\|_{H^3(\Omega)}$,
which is the rate appearing in \eqref{eq:dirichlet-estimate-2}.

Under the stronger assumption $u_{\mathrm{loc}}\in C^4(\overline{\Omega})$,
the maximum principle further yields the pointwise estimate
\[
|u_{\mathrm{loc}}(\xx)-u_\delta(\xx)|
\le C_\Omega\,\delta\,\|u_{\mathrm{loc}}\|_{C^4(\overline{\Omega})},
\qquad \forall\,\xx\in\Omega,
\]
where the constant $C_\Omega>0$ depends only on $\Omega$ and the kernel $R$.
In this stronger setting the $L^2$ local-limit error satisfies
$\|u_{\mathrm{loc}}-u_\delta\|_{L^2(\Omega)}
\le C_\Omega\,|\Omega|^{1/2}\,\delta\,\|u_{\mathrm{loc}}\|_{C^4(\overline{\Omega})}$,
and \eqref{eq:dirichlet-estimate-2} improves to
\[
\|u_{\mathrm{loc}}-u_{\delta,p}\|_{L^2(\Omega)}
\le
C\!\left(\frac{\varepsilon_f}{\sqrt\delta}
+\frac{\varepsilon_g}{\delta}
+\frac{\varepsilon_u}{\sqrt\delta}
+\eta_{\mathrm{opt}}\right)
+C_\Omega\,\delta\,\|u_{\mathrm{loc}}\|_{C^4(\overline{\Omega})},
\]
with a first-order local-limit term; the multiplicative constant depends only
on $\Omega$ and the kernel $R$. The assumption $u_{\mathrm{loc}}\in C^4(\overline{\Omega})$
requires data regularity beyond $f\in H^1(\Omega)$ and
$g\in H^{5/2}(\partial\Omega)$; since the $H^3$ assumption on
$u_{\mathrm{loc}}$ is sufficient for the proof of
\eqref{eq:dirichlet-estimate-2}, we do not impose it in the theorem
statement.
\end{remark}

\subsection{Nonlocal diffusion model with a Neumann boundary condition}
\label{subsec:nonlocal-neumann}

We also consider the elliptic problem with a Neumann boundary condition:
\begin{equation}
\label{eq:neumann-local}
  \left\{
    \begin{aligned}
      -\Delta u(\xx)+u(\xx)&=f(\xx), && \xx \in\Omega,\\
      \frac{\partial u}{\partial \mathbf{n}}(\xx)&=g(\xx),
      && \xx\in \partial\Omega.
    \end{aligned}
  \right.
\end{equation}
The additional zeroth-order term ensures the well-posedness of the Neumann
problem. The corresponding nonlocal formulation is taken from \cite{meng2023maximum}:
\begin{align}
&\frac{1}{\delta^{2}}
\int_\Omega R_\delta(\xx,\yy)
\bigl(u_{\delta,N}(\xx)-u_{\delta,N}(\yy)\bigr)\,\d\yy
+
\int_\Omega
\bar{R}_\delta(\xx,\yy)\,u_{\delta,N}(\yy)\,\d\yy
\notag \\
&\qquad =
\int_\Omega
\bar{R}_\delta(\xx,\yy)f(\yy)\,\d\yy
+
2\int_{\partial\Omega}
\bar{R}_\delta(\xx,\yy)g(\yy)\,\d S_\yy .
\label{eq:neumann-nonlocal}
\end{align}

The following well-posedness and local-limit results will be used in the
subsequent analysis. We provide the proof details needed for the present
analysis in
Appendices~\ref{appendix:neumann-wellposedness}--\ref{appendix:neumann-locallimit}.

\begin{proposition}[Well-posedness of the nonlocal Neumann problem]
\label{prop:neumann-wellposedness}
For any $\delta>0$, $f\in L^2(\Omega)$ and $g\in L^2(\partial\Omega)$, the
problem \eqref{eq:neumann-nonlocal} admits a unique solution
$u_{\delta,N}\in H^1(\Omega)$. Moreover, there exists a constant $C>0$,
independent of $\delta$, such that
\begin{equation}
\label{eq:neumann-wellposedness-bound}
\|u_{\delta,N}\|_{H^1(\Omega)}
\le
C\!\left(\|f\|_{L^2(\Omega)}
+\frac{1}{\sqrt{\delta}}\,\|g\|_{L^2(\partial\Omega)}\right).
\end{equation}
\end{proposition}

\begin{proposition}[Local limit of the nonlocal Neumann problem]
\label{prop:neumann-locallimit}
Let $u_{\delta,N}$ be the solution of \eqref{eq:neumann-nonlocal}, and assume
that the local Neumann problem \eqref{eq:neumann-local} admits a solution
$u_{\mathrm{loc},N}\in H^3(\Omega)$. Then, for all sufficiently small
$\delta>0$,
\begin{equation}
\label{eq:neumann-locallimit-bound}
\|u_{\mathrm{loc},N}-u_{\delta,N}\|_{H^1(\Omega)}
\le
C\delta\,\|u_{\mathrm{loc},N}\|_{H^3(\Omega)},
\end{equation}
where $C>0$ is independent of $\delta$.
\end{proposition}

\begin{remark}
\label{rem:neumann-regularity}
By standard elliptic regularity, if $f\in H^1(\Omega)$ and
$g\in H^{3/2}(\partial\Omega)$, then $u_{\mathrm{loc},N}\in H^3(\Omega)$
and
\[
\|u_{\mathrm{loc},N}\|_{H^3(\Omega)}
\le
C\!\left(\|f\|_{H^1(\Omega)}
+\|g\|_{H^{3/2}(\partial\Omega)}\right).
\]
Note that the Neumann rate $O(\delta)$ is one half-order better than the
Dirichlet rate $O(\sqrt\delta)$.
\end{remark}

\subsection{Tensor neural networks}
\label{subsec:tnn}

We briefly review the tensor neural network (TNN) architecture introduced in
\cite{wang2022tensor}. A TNN represents a multivariate function as a sum
of tensor products of one-dimensional subnetworks, so that high-dimensional
integrals reduce to products of one-dimensional integrals. Throughout this
paper, TNN refers to this previously introduced architecture; our contribution
is to use it as a building block in a variational solver for nonlocal diffusion
models and to analyze the resulting error. The ansatz takes
the form
\begin{equation}
\label{eq:tnn-structure}
\Psi(\xx;\Theta)
=
c \sum_{j=1}^p
\phi_{1,j}(x_1;\theta_1)
\phi_{2,j}(x_2;\theta_2)
\cdots
\phi_{d,j}(x_d;\theta_d)
=
c \sum_{j=1}^p\prod_{i=1}^d \phi_{i,j}(x_i;\theta_i).
\end{equation}
Here $c$ is a trainable scaling parameter. For each $i=1,\ldots,d$,
\[
\Phi_i(x_i;\theta_i)
:=
\bigl(\phi_{i,1}(x_i;\theta_i),\,\ldots,\,\phi_{i,p}(x_i;\theta_i)\bigr)
\]
is a fully connected neural network mapping $\R$ into $\R^{p}$ ,where the sine function is used as the activation function in the hidden layers . The parameter $\theta_i$ collects its trainable weights and biases. The integer $p$ is
called the \emph{separation rank} of the TNN, and
$\theta_1,\ldots,\theta_d$ denote distinct trainable parameter sets for the
one-dimensional subnetworks.
Figure~\ref{fig:tnn-structure} illustrates the overall architecture.

\begin{figure}[htbp]
\centering
\includegraphics[scale=0.5]{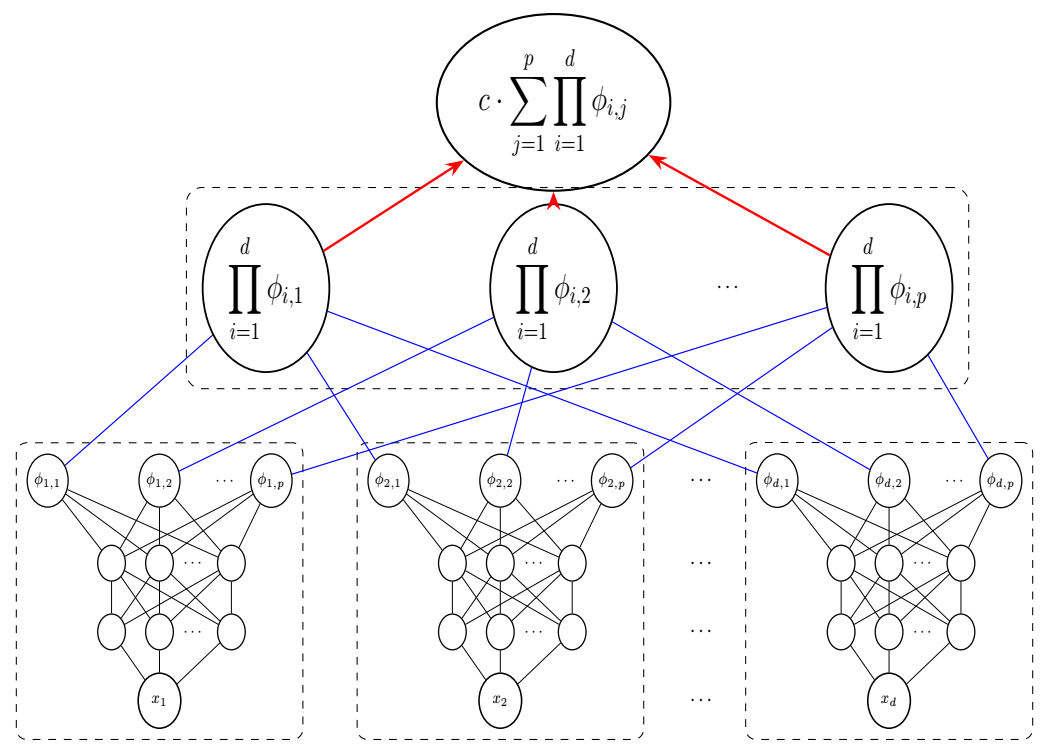}
\caption{Illustration of the tensor neural network architecture.}
\label{fig:tnn-structure}
\end{figure}

For simplicity, we restrict to the tensor-product domain
\[
\Omega = \Omega_1\times\cdots\times\Omega_d,
\qquad
\Omega_i=[a_i,b_i],\quad i=1,\ldots,d.
\]
Denote by $\mathcal{V}_p^d$ the class of functions representable in the form
\eqref{eq:tnn-structure} with separation rank at most $p$; lower-rank
representations are embedded by padding zero components. The space
$\mathcal{V}_p^d$ may be regarded as a finite-rank approximation in
$H^m(\Omega_1)\otimes\cdots\otimes H^m(\Omega_d)$,
or more generally in mixed Sobolev spaces.
The following approximation result was established in \cite{wang2022tensor}.

\begin{proposition}[Approximation property of TNNs]
\label{prop:tnn-approximation}
Let $m\ge 0$ and $f\in H^m(\Omega)$. Then for any $\varepsilon>0$, there exist
a positive integer $p$ and a TNN output $\Psi(\xx;\Theta)\in\mathcal{V}_p^d$
of the form \eqref{eq:tnn-structure} such that
\begin{equation}
\label{eq:tnn-approx}
\|f-\Psi(\cdot;\Theta)\|_{H^m(\Omega)}
<
\varepsilon .
\end{equation}
\end{proposition}

\begin{remark}
Proposition~\ref{prop:tnn-approximation} states that
$\bigcup_{p\ge 1}\mathcal{V}_p^d$ is dense in $H^m(\Omega)$, but does not
provide an explicit relationship between the rank $p$ and the tolerance
$\varepsilon$. Quantitative rates are available in mixed Sobolev spaces
$H^{t,l}_{\mathrm{mix}}(\Omega)$; see \cite{griebel2000optimized}. In the
present work the density statement suffices.
\end{remark}

With the TNN structure \eqref{eq:tnn-structure}, integration over $\Omega$
reduces to products of one-dimensional integrals:
\begin{align}
\int_\Omega \Psi(\xx;\Theta)\,\d\xx
&=
c \sum_{j=1}^p
\prod_{i=1}^d
\int_{\Omega_i}
\phi_{i,j}(x_i;\theta_i)\,\d x_i .
\label{eq:tnn-integral}
\end{align}
More generally, the pointwise product of two TNN functions has a separable
tensor-product representation with rank equal to the product of the individual
ranks.
Given two TNN functions
\[
\Psi(\xx;\Theta)
=
c_\Psi\sum_{j=1}^{p_\Psi}\prod_{i=1}^d \phi_{i,j}(x_i;\theta_i),
\qquad
V(\xx;\Xi)
=
c_V\sum_{k=1}^{p_V}\prod_{i=1}^d \psi_{i,k}(x_i;\xi_i),
\]
their integral over $\Omega$ satisfies
\begin{equation}
\int_\Omega
\Psi(\xx;\Theta)\,V(\xx;\Xi)\,\d\xx
=
c_\Psi c_V
\sum_{j=1}^{p_\Psi}
\sum_{k=1}^{p_V}
\prod_{i=1}^d
\int_{\Omega_i}
\phi_{i,j}(x_i;\theta_i)\,
\psi_{i,k}(x_i;\xi_i)\,\d x_i .
\label{eq:tnn-l2}
\end{equation}
In particular, \eqref{eq:tnn-l2} shows that any integral of a TNN function
against a tensor-product weight reduces to a collection of one-dimensional
integrals, which is the key property exploited in the error analysis below.

\subsection{Notation and basic facts}
\label{subsec:notation}

We collect here some notation and basic estimates used throughout the paper.
For $u\in L^2(\Omega)$, define the smoothing operator
\begin{equation}
\label{eq:mollifier}
S_\delta u(\xx)
=
\frac{1}{w_\delta(\xx)}
\int_\Omega R_\delta(\xx,\yy)u(\yy)\,\d\yy,
\qquad
w_\delta(\xx):=\int_\Omega R_\delta(\xx,\yy)\,\d\yy .
\end{equation}
For kernels satisfying the compact-support assumptions of
Section~\ref{subsec:nonlocal-dirichlet}, and also for the Gaussian kernel used
in Section~\ref{sec:main-result}, the following estimates hold for all
$\xx\in\overline{\Omega}$ and all sufficiently small $\delta>0$:
\begin{align}
&C_1 \le \int_\Omega R_\delta(\xx,\yy)\,\d\yy \le C_2,
\label{eq:kernel-est-1}\\
&C_3 \le \int_\Omega \bar{R}_\delta(\xx,\yy)\,\d\yy \le C_4,
\label{eq:kernel-est-2}\\
&\int_{\partial\Omega} \bar{R}_\delta(\xx,\yy)\,\d S_\yy
\le \frac{C_5}{\delta},
\label{eq:kernel-est-3}\\
&\int_{\partial\Omega} R_\delta(\xx,\yy)\,\d S_\yy
\le \frac{C_6}{\delta}.
\label{eq:kernel-est-4}
\end{align}
The constants $C_1,\dots,C_6$ depend only on $\Omega$ and the kernel
parameters, and are independent of $\delta$. For the Gaussian kernel adopted
later, these estimates follow from exponential decay and finite moments. As an
immediate consequence of \eqref{eq:kernel-est-1} and the scaled moment bounds
of the kernel, the weight $w_\delta$ satisfies the gradient bound
\begin{equation}
\label{eq:wdelta-grad}
|\nabla w_\delta(\xx)|
\le \frac{C}{\delta},
\qquad \xx\in\overline{\Omega} ,
\end{equation}
which follows by differentiating $w_\delta(\xx)=\int_\Omega
R_\delta(\xx,\yy)\,\d\yy$ under the integral sign and using
$\int_\Omega |\nabla_{\!\xx}R_\delta(\xx,\yy)|\,\d\yy\le C/\delta$. In
particular, the lower bound in \eqref{eq:kernel-est-1} gives
$1/w_\delta(\xx)\le C$ uniformly in $\xx\in\overline{\Omega}$ and small
$\delta>0$; in subsequent estimates we absorb such factors of
$1/w_\delta$ into the generic constant $C$ without further comment.

\section{Workflow and the main result}
\label{sec:main-result}

\subsection{Dirichlet boundary condition}
\label{subsec:dirichlet-workflow}

We use the TNN ansatz reviewed in Section~\ref{subsec:tnn} as the trial class
for the nonlocal problems \eqref{eq:dirichlet-nonlocal} and
\eqref{eq:neumann-nonlocal}. To exploit the separation-of-variables structure
of TNNs, we adopt the following two assumptions throughout this and the next
section:
\begin{itemize}
\item[(a)] (\emph{Separable Gaussian kernel}) $R(r)=e^{-s^2 r}$ for
$r\in[0,+\infty)$, where $s>0$ is a fixed parameter.
\item[(b)] (\emph{Rectangularly partitionable domain}) $\Omega = \bigcup_\alpha
T_\alpha$, where the union is finite, the interiors of the rectangles are
disjoint, and each
\[
T_\alpha = [a_1^\alpha, b_1^\alpha]\times\cdots\times[a_d^\alpha, b_d^\alpha]
\]
is a $d$-dimensional tensor-product rectangle.
\end{itemize}
This Gaussian kernel is not compactly supported, but the analysis below only
uses the standard mass, moment, boundary and coercivity/consistency estimates
summarized in Section~\ref{sec:preliminaries}; these estimates remain valid in
the Gaussian case. We choose the Gaussian kernel mainly because its separability
is exact. For this choice, the normalization constant $\alpha_d$ is understood
to be chosen by the same condition as in \eqref{eq:kernel-def}, with the radial
integral taken over $[0,+\infty)$. Under (a), the rescaled kernels become
\begin{align}
R_\delta(\xx,\yy)
&= \alpha_d\,\delta^{-d}
\exp\!\left(-\frac{s^2}{4\delta^2}\sum_{i=1}^d (x_i-y_i)^2\right),
\notag \\
\bar R_\delta(\xx,\yy)
&= \frac{\alpha_d}{s^2}\,\delta^{-d}
\exp\!\left(-\frac{s^2}{4\delta^2}\sum_{i=1}^d (x_i-y_i)^2\right) ,
\label{eq:gaussian-kernel}
\end{align}
so that $\bar R_\delta(\xx,\yy) = s^{-2}R_\delta(\xx,\yy)$. The
decisive feature of the Gaussian kernel is the factorization
\begin{equation}
\label{eq:kernel-factorization}
\exp\!\left(-\frac{s^2}{4\delta^2}\sum_{i=1}^d (x_i-y_i)^2\right)
=
\prod_{i=1}^d
\exp\!\left(-\frac{s^2}{4\delta^2}(x_i-y_i)^2\right) ,
\end{equation}
which, when combined with the TNN ansatz, allows every loss functional
encountered below to be evaluated as a finite sum of products of
low-dimensional integrals; the bulk terms reduce to products of
two-dimensional integrals.

We illustrate the workflow on the Dirichlet problem; the Neumann case is
analogous. For brevity we further assume that $\Omega$ is a single
rectangle,
\[
\Omega = [a_1,b_1]\times\cdots\times[a_d,b_d] ;
\]
the extension to a finite union of such rectangles is straightforward.

A direct calculation shows that, for $u\in H^1(\Omega)$, the energy
\begin{align}
\mathcal{L}_\delta(u)
&:= \frac{1}{4\delta^2}
\!\int_\Omega\!\!\int_\Omega\!
R_\delta(\xx,\yy)\bigl(u(\xx)-u(\yy)\bigr)^2\,\d\xx\,\d\yy
+ \frac{1}{\delta}
\!\int_\Omega\!\!\int_{\partial\Omega}\!
\bar R_\delta(\xx,\yy)\,u(\xx)^2\,\d\xx\,\d S_\yy
\notag \\
&\quad
- \!\int_\Omega\!\!\int_\Omega\!
\bar R_\delta(\xx,\yy)\,u(\xx)f(\yy)\,\d\xx\,\d\yy
- \frac{2}{\delta}
\!\int_\Omega\!\!\int_{\partial\Omega}\!
\bar R_\delta(\xx,\yy)\,u(\xx)g(\yy)\,\d\xx\,\d S_\yy
\label{eq:dirichlet-loss}
\end{align}
is well defined (the boundary term making sense by the standard trace
theorem on $H^1(\Omega)$), is strictly convex in $u$, and admits a unique
minimizer in $H^1(\Omega)$. The first variation of
\eqref{eq:dirichlet-loss} recovers \eqref{eq:dirichlet-nonlocal} after
symmetrization in $\xx,\yy$, so minimizing $\mathcal{L}_\delta$ over
$H^1(\Omega)$ is equivalent to solving the nonlocal Dirichlet problem
\eqref{eq:dirichlet-nonlocal}. The numerical scheme considered here replaces the trial space
$H^1(\Omega)$ by the TNN class $\mathcal{V}_p^d$ and minimizes $\mathcal{L}_\delta$ over
$\mathcal{V}_p^d$.

For general $f$ and $g$, however, the integrals in
\eqref{eq:dirichlet-loss} cannot be separated into products of
low-dimensional integrals in the way needed here, because the data are not
separable in their arguments. To overcome this, we
follow \cite{li2024tensor} and adopt a preconditioning step: first
approximate $f$ and $g$ in $L^2$ by TNN surrogates $\tilde f$ and
$\tilde g$, and then minimize a modified loss in which $f$, $g$ are
replaced by $\tilde f$, $\tilde g$. The construction of $\tilde f$ is
straightforward, but a remark on $\tilde g$ is in order. The boundary
$\partial\Omega$ of the rectangle $\Omega = [a_1,b_1]\times\cdots\times[a_d,b_d]$
decomposes into $2d$ faces,
\[
\partial\Omega
= \bigcup_{i=1}^{d}\bigcup_{s\in\{a_i,b_i\}} F_{i,s},
\qquad
F_{i,s} := \{\xx\in\overline{\Omega} : x_i = s\} ,
\]
each of which is itself a $(d-1)$-dimensional tensor-product rectangle.
Accordingly, $\tilde g$ is constructed face-by-face: on each face
$F_{i,s}$ we use a $(d-1)$-variable TNN; the global surrogate $\tilde g$
is the union of these $2d$ pieces. For brevity we suppress the face index
in the formulas below and write $\tilde g$ as a single TNN with $d-1$
subnetworks, with the understanding that this representation is to be
applied face by face.

We obtain $\tilde f$ and $\tilde g$ by minimizing the empirical $L^2$
losses
\begin{align}
\mathcal{L}_f(\tilde\theta_f)
&:= \frac{1}{N_f}\sum_{i=1}^{N_f}
\bigl|f(\xx_i)-\tilde f(\xx_i;\tilde\theta_f)\bigr|^2 ,
\label{eq:dirichlet-loss-f}
\\
\mathcal{L}_g(\tilde\theta_g)
&:= \frac{1}{N_g}\sum_{j=1}^{N_g}
\bigl|g(\yy_j)-\tilde g(\yy_j;\tilde\theta_g)\bigr|^2 ,
\label{eq:dirichlet-loss-g}
\end{align}
where $\{\xx_i\}_{i=1}^{N_f}\subset\Omega$ and
$\{\yy_j\}_{j=1}^{N_g}\subset\partial\Omega$ are i.i.d.\ Monte Carlo
samples. Upon convergence, we obtain the TNN representations
\begin{align*}
\tilde f(\xx;\tilde\theta_f)
&= c_f\sum_{j=1}^{p_f}\prod_{i=1}^d
\phi^{f}_{i,j}(x_i;\tilde\theta^{f}_i),
\\
\tilde g(\yy;\tilde\theta_g)
&= c_g\sum_{j=1}^{p_g}\prod_{i=1}^{d-1}
\phi^{g}_{i,j}(y_i;\tilde\theta^{g}_i) ,
\end{align*}
satisfying the prescribed tolerances
\begin{equation}
\label{eq:dirichlet-precondition}
\|f-\tilde f\|_{L^2(\Omega)}<\varepsilon_f,
\qquad
\|g-\tilde g\|_{L^2(\partial\Omega)}<\varepsilon_g .
\end{equation}
The achievability of \eqref{eq:dirichlet-precondition} for any prescribed
$\varepsilon_f,\varepsilon_g>0$ follows from
Proposition~\ref{prop:tnn-approximation}.

With $\tilde f$ and $\tilde g$ at hand, the loss function actually
minimized is
\begin{align}
\tilde{\mathcal{L}}_\delta(u)
&:= \frac{1}{4\delta^2}
\!\int_\Omega\!\!\int_\Omega\!
R_\delta(\xx,\yy)\bigl(u(\xx)-u(\yy)\bigr)^2\,\d\xx\,\d\yy
+ \frac{1}{\delta}
\!\int_\Omega\!\!\int_{\partial\Omega}\!
\bar R_\delta(\xx,\yy)\,u(\xx)^2\,\d\xx\,\d S_\yy
\notag \\
&\quad
-\!\int_\Omega\!\!\int_\Omega\!
\bar R_\delta(\xx,\yy)\,u(\xx)\tilde f(\yy)\,\d\xx\,\d\yy
- \frac{2}{\delta}
\!\int_\Omega\!\!\int_{\partial\Omega}\!
\bar R_\delta(\xx,\yy)\,u(\xx)\tilde g(\yy)\,\d\xx\,\d S_\yy .
\label{eq:dirichlet-actual-loss}
\end{align}
Expanding $(u(\xx)-u(\yy))^2$ and using the symmetry
$R_\delta(\xx,\yy)=R_\delta(\yy,\xx)$, the quadratic part splits as
\begin{align*}
&\frac{1}{4\delta^2}
\!\int_\Omega\!\!\int_\Omega\!
R_\delta(\xx,\yy)\bigl(u(\xx)-u(\yy)\bigr)^2\,\d\xx\,\d\yy
\\
&\qquad =
\frac{1}{2\delta^2}\!\int_\Omega\!\!\int_\Omega\!
R_\delta(\xx,\yy)\,u(\xx)^2\,\d\xx\,\d\yy
-
\frac{1}{2\delta^2}\!\int_\Omega\!\!\int_\Omega\!
R_\delta(\xx,\yy)\,u(\xx)u(\yy)\,\d\xx\,\d\yy .
\end{align*}
and consequently $\tilde{\mathcal{L}}_\delta$ admits the decomposition
\begin{equation}
\label{eq:dirichlet-loss-decomp}
\tilde{\mathcal{L}}_\delta(u)
= \tilde{\mathcal{L}}_{\delta,1}(u)
+ \tilde{\mathcal{L}}_{\delta,2}(u)
+ \tilde{\mathcal{L}}_{\delta,3}(u)
+ \tilde{\mathcal{L}}_{\delta,4}(u)
+ \tilde{\mathcal{L}}_{\delta,5}(u),
\end{equation}
where
\begin{align*}
\tilde{\mathcal{L}}_{\delta,1}(u)
&= \frac{1}{2\delta^2}
\!\int_\Omega\!\!\int_\Omega\!
R_\delta(\xx,\yy)\,u(\xx)^2\,\d\xx\,\d\yy,
\\
\tilde{\mathcal{L}}_{\delta,2}(u)
&= -\frac{1}{2\delta^2}
\!\int_\Omega\!\!\int_\Omega\!
R_\delta(\xx,\yy)\,u(\xx)u(\yy)\,\d\xx\,\d\yy,
\\
\tilde{\mathcal{L}}_{\delta,3}(u)
&= \frac{1}{\delta}
\!\int_\Omega\!\!\int_{\partial\Omega}\!
\bar R_\delta(\xx,\yy)\,u(\xx)^2\,\d\xx\,\d S_\yy,
\\
\tilde{\mathcal{L}}_{\delta,4}(u)
&= -\!\int_\Omega\!\!\int_\Omega\!
\bar R_\delta(\xx,\yy)\,u(\xx)\tilde f(\yy)\,\d\xx\,\d\yy,
\\
\tilde{\mathcal{L}}_{\delta,5}(u)
&= -\frac{2}{\delta}
\!\int_\Omega\!\!\int_{\partial\Omega}\!
\bar R_\delta(\xx,\yy)\,u(\xx)\tilde g(\yy)\,\d\xx\,\d S_\yy.
\end{align*}

We minimize $\tilde{\mathcal{L}}_\delta$ over $\mathcal{V}_p^d$. Suppose
that the current trial function takes the TNN form
\begin{equation}
\label{eq:tnn-trial}
u(\xx;\Theta)
= c \sum_{j=1}^{p_u}\prod_{i=1}^d
\phi_{i,j}(x_i;\theta_i) .
\end{equation}
Combining \eqref{eq:gaussian-kernel} and
\eqref{eq:kernel-factorization} with \eqref{eq:tnn-trial}, every term in
\eqref{eq:dirichlet-loss-decomp} reduces to a finite sum of two-dimensional
integrals. We illustrate this on $\tilde{\mathcal{L}}_{\delta,1}$.
Squaring \eqref{eq:tnn-trial} gives
\[
u(\xx)^2
= c^2\sum_{j_1,j_2=1}^{p_u}\prod_{i=1}^d
\phi_{i,j_1}(x_i;\theta_i)\,\phi_{i,j_2}(x_i;\theta_i),
\]
and substituting into $\tilde{\mathcal{L}}_{\delta,1}$ yields
\begin{align}
\tilde{\mathcal{L}}_{\delta,1}(u)
&=
\frac{c^2 \alpha_d}{2\,\delta^{d+2}}
\sum_{j_1,j_2=1}^{p_u}
\prod_{i=1}^d
\!\int_{a_i}^{b_i}\!\!\int_{a_i}^{b_i}\!
\exp\!\left(-\frac{s^2}{4\delta^2}(x_i-y_i)^2\right)
\notag \\
&\hspace{4cm}\times\,
\phi_{i,j_1}(x_i;\theta_i)\,\phi_{i,j_2}(x_i;\theta_i)
\,\d x_i\,\d y_i .
\label{eq:L1-factorized}
\end{align}
The original $2d$-dimensional integral has thus been reduced to a finite
sum of products of two-dimensional integrals, which we evaluate by composite
Gauss--Legendre quadrature. The remaining terms
$\tilde{\mathcal{L}}_{\delta,k}$, $k=2,3,4,5$, admit analogous
factorizations. The number of subintervals and quadrature points used in
the two-dimensional rules grows as $\delta\to 0$, since the Gaussian factor
$\exp(-s^2(x_i-y_i)^2/(4\delta^2))$ becomes increasingly localized near the
diagonal $x_i=y_i$; the precise choice of these quadrature parameters is
specified in Section~\ref{sec:numerical-tests}.

Assume that the exact minimizer
$\tilde u_{\delta,p}\in\mathcal{V}_p^d$ of
$\tilde{\mathcal{L}}_\delta$ over $\mathcal{V}_p^d$ exists, and let
$u_{\delta,p}\in\mathcal{V}_p^d$ be the actual TNN output produced by the
optimization algorithm. We measure the gap between the two by the
\emph{optimization error}
\begin{equation}
\label{eq:opt-error-dirichlet}
\eta_{\mathrm{opt}}^{\,2}
:= \tilde{\mathcal{L}}_\delta(u_{\delta,p})
- \tilde{\mathcal{L}}_\delta(\tilde u_{\delta,p}) \ge 0 .
\end{equation}
Our goal is to estimate $\|u_\delta-u_{\delta,p}\|_{L^2(\Omega)}$ and
$\|u_{\mathrm{loc}}-u_{\delta,p}\|_{L^2(\Omega)}$, where $u_\delta$ and
$u_{\mathrm{loc}}$ are the solutions of \eqref{eq:dirichlet-nonlocal} and
\eqref{eq:dirichlet-local}, respectively. The main difficulty is to relate
the energy minimization to the approximation property of TNNs in
$H^1(\Omega)$. Our first main result is the following.

\begin{theorem}[Error estimate, Dirichlet case]
\label{thm:dirichlet-error}
Let $0<\delta\le\delta_0\le 1$, where $\delta_0$ is fixed, and let
$\varepsilon_f,\varepsilon_g,\varepsilon_u>0$ be prescribed tolerances.
Assume that the Dirichlet data satisfy the regularity condition stated in
Section~\ref{subsec:nonlocal-dirichlet}, namely
\begin{equation}
\label{eq:dirichlet-data-regularity-main}
f\in H^1(\Omega),
\qquad
g\in H^{5/2}(\partial\Omega).
\end{equation}
Then, by the elliptic regularity statement recalled in
Remark~\ref{rem:dirichlet-regularity}, the local solution satisfies
$u_{\mathrm{loc}}\in H^3(\Omega)$ and
\[
\|u_{\mathrm{loc}}\|_{H^3(\Omega)}
\le
C\!\left(\|f\|_{H^1(\Omega)}
+\|g\|_{H^{5/2}(\partial\Omega)}\right).
\]
Assume that there exist $\tilde f\in\mathcal{V}_{p_f}^d$,
$\tilde g\in\mathcal{V}_{p_g}^{d-1}$ and
$\bar u_\delta\in\mathcal{V}_{p_u}^d$ for some
$p_f,p_g,p_u\in\mathbb{N}^{*}$ such that
\begin{align}
\|f-\tilde f\|_{L^2(\Omega)}
&\le \varepsilon_f,
&
\|g-\tilde g\|_{L^2(\partial\Omega)}
&\le \varepsilon_g,
&
\|u_\delta-\bar u_\delta\|_{H^1(\Omega)}
&\le \varepsilon_u .
\label{eq:dirichlet-tolerances}
\end{align}
Let $p\ge p_u$. Assume that $\tilde{\mathcal L}_\delta$ admits a minimizer
$\tilde u_{\delta,p}$ over $\mathcal{V}_p^d$, and let
$u_{\delta,p}\in\mathcal{V}_p^d$ be the TNN output, with optimization error
$\eta_{\mathrm{opt}}$ defined by \eqref{eq:opt-error-dirichlet}. Then
there exists a constant $C>0$, independent of $\delta$, such that
\begin{equation}
\label{eq:dirichlet-estimate-1}
\|u_\delta-u_{\delta,p}\|_{L^2(\Omega)}
\le
C\!\left(
\frac{\varepsilon_f}{\sqrt{\delta}}
+\frac{\varepsilon_g}{\delta}
+\frac{\varepsilon_u}{\sqrt{\delta}}
+\eta_{\mathrm{opt}}\right) .
\end{equation}
Moreover,
\begin{equation}
\label{eq:dirichlet-estimate-2}
\|u_{\mathrm{loc}}-u_{\delta,p}\|_{L^2(\Omega)}
\le
C\!\left(
\frac{\varepsilon_f}{\sqrt{\delta}}
+\frac{\varepsilon_g}{\delta}
+\frac{\varepsilon_u}{\sqrt{\delta}}
+\eta_{\mathrm{opt}}\right)
+
C\sqrt{\delta}\!\left(\|f\|_{H^1(\Omega)}
+\|g\|_{H^{5/2}(\partial\Omega)}\right) .
\end{equation}
\end{theorem}

\subsection{Neumann boundary condition}
\label{subsec:neumann-workflow}

The Neumann case is treated analogously, but yields sharper estimates. As
the variational counterpart of the nonlocal Neumann problem
\eqref{eq:neumann-nonlocal}, we minimize over $H^1(\Omega)$ the energy
\begin{align}
\mathcal{L}_{\delta,N}(u)
&:= \frac{1}{4\delta^2}
\!\int_\Omega\!\!\int_\Omega\!
R_\delta(\xx,\yy)\bigl(u(\xx)-u(\yy)\bigr)^2\,\d\xx\,\d\yy
+ \frac{1}{2}
\!\int_\Omega\!\!\int_\Omega\!
\bar R_\delta(\xx,\yy)\,u(\xx)u(\yy)\,\d\xx\,\d\yy
\notag \\
&\quad
-\!\int_\Omega\!\!\int_\Omega\!
\bar R_\delta(\xx,\yy)\,u(\xx)f(\yy)\,\d\xx\,\d\yy
- 2\!\int_\Omega\!\!\int_{\partial\Omega}\!
\bar R_\delta(\xx,\yy)\,u(\xx)g(\yy)\,\d\xx\,\d S_\yy .
\label{eq:neumann-loss}
\end{align}
The first variation of \eqref{eq:neumann-loss} recovers
\eqref{eq:neumann-nonlocal} after symmetrization in $\xx,\yy$.

As in the Dirichlet case, we precondition the data by TNN surrogates
$\tilde f$, $\tilde g$ satisfying
\begin{equation}
\label{eq:neumann-precondition}
\|f-\tilde f\|_{L^2(\Omega)}<\varepsilon_f,
\qquad
\|g-\tilde g\|_{L^2(\partial\Omega)}<\varepsilon_g ,
\end{equation}
and minimize over $\mathcal{V}_p^d$ the modified loss
\begin{align}
\tilde{\mathcal{L}}_{\delta,N}(u)
&:= \frac{1}{4\delta^2}
\!\int_\Omega\!\!\int_\Omega\!
R_\delta(\xx,\yy)\bigl(u(\xx)-u(\yy)\bigr)^2\,\d\xx\,\d\yy
+ \frac{1}{2}
\!\int_\Omega\!\!\int_\Omega\!
\bar R_\delta(\xx,\yy)\,u(\xx)u(\yy)\,\d\xx\,\d\yy
\notag \\
&\quad
-\!\int_\Omega\!\!\int_\Omega\!
\bar R_\delta(\xx,\yy)\,u(\xx)\tilde f(\yy)\,\d\xx\,\d\yy
- 2\!\int_\Omega\!\!\int_{\partial\Omega}\!
\bar R_\delta(\xx,\yy)\,u(\xx)\tilde g(\yy)\,\d\xx\,\d S_\yy .
\label{eq:neumann-actual-loss}
\end{align}
The factorization of \eqref{eq:neumann-actual-loss} into products of
low-dimensional integrals proceeds exactly as in the Dirichlet case, so we
omit the details. Assume that the exact minimizer
$\tilde u_{\delta,p,N}\in\mathcal{V}_p^d$ of
$\tilde{\mathcal{L}}_{\delta,N}$ over $\mathcal{V}_p^d$ exists, and let
$u_{\delta,p,N}\in\mathcal{V}_p^d$ be the corresponding TNN output,
with optimization error
\begin{equation}
\label{eq:opt-error-neumann}
\eta_{\mathrm{opt}}^{\,2}
:= \tilde{\mathcal{L}}_{\delta,N}(u_{\delta,p,N})
- \tilde{\mathcal{L}}_{\delta,N}(\tilde u_{\delta,p,N}) \ge 0 .
\end{equation}

The Neumann counterpart of Theorem~\ref{thm:dirichlet-error} reads as
follows; the improvement over the Dirichlet case lies in the better
$\delta$-dependence of the data terms.

\begin{theorem}[$L^2$ error estimate, Neumann case]
\label{thm:neumann-error-l2}
Let $0<\delta\le\delta_0\le 1$, where $\delta_0$ is fixed, and let
$\varepsilon_f,\varepsilon_g,\varepsilon_u>0$ be prescribed tolerances.
Assume that the Neumann data satisfy
\begin{equation}
\label{eq:neumann-data-regularity-main}
f\in H^1(\Omega),
\qquad
g\in H^{3/2}(\partial\Omega),
\end{equation}
so that, by the elliptic regularity statement in
Remark~\ref{rem:neumann-regularity}, the local solution satisfies
$u_{\mathrm{loc},N}\in H^3(\Omega)$ and
\[
\|u_{\mathrm{loc},N}\|_{H^3(\Omega)}
\le
C\!\left(\|f\|_{H^1(\Omega)}
+\|g\|_{H^{3/2}(\partial\Omega)}\right).
\]
Assume that there exist $\tilde f\in\mathcal{V}_{p_f}^d$,
$\tilde g\in\mathcal{V}_{p_g}^{d-1}$ and
$\bar u_{\delta,N}\in\mathcal{V}_{p_u}^d$ for some
$p_f,p_g,p_u\in\mathbb{N}^{*}$ satisfying
\begin{align*}
\|f-\tilde f\|_{L^2(\Omega)}
&\le \varepsilon_f,
&
\|g-\tilde g\|_{L^2(\partial\Omega)}
&\le \varepsilon_g,
&
\|u_{\delta,N}-\bar u_{\delta,N}\|_{H^1(\Omega)}
&\le \varepsilon_u .
\end{align*}
Let $p\ge p_u$. Assume that $\tilde{\mathcal L}_{\delta,N}$ admits a
minimizer $\tilde u_{\delta,p,N}$ over $\mathcal{V}_p^d$, and let
$u_{\delta,p,N}\in\mathcal{V}_p^d$ be the TNN output, with optimization error
$\eta_{\mathrm{opt}}$ defined by \eqref{eq:opt-error-neumann}. Then there
exists a constant $C>0$, independent of $\delta$, such that
\begin{equation}
\label{eq:neumann-estimate-1}
\|u_{\delta,N}-u_{\delta,p,N}\|_{L^2(\Omega)}
\le
C\!\left(
\varepsilon_f
+\frac{\varepsilon_g}{\sqrt{\delta}}
+\varepsilon_u
+\eta_{\mathrm{opt}}\right) .
\end{equation}
Moreover,
\begin{equation}
\label{eq:neumann-estimate-2}
\|u_{\mathrm{loc},N}-u_{\delta,p,N}\|_{L^2(\Omega)}
\le
C\!\left(
\varepsilon_f
+\frac{\varepsilon_g}{\sqrt{\delta}}
+\varepsilon_u
+\eta_{\mathrm{opt}}\right)
+
C\delta\!\left(\|f\|_{H^1(\Omega)}
+\|g\|_{H^{3/2}(\partial\Omega)}\right) .
\end{equation}
\end{theorem}

In addition to the $L^2$ estimates above, the Neumann setting allows for a
gradient estimate, which is not available in the Dirichlet case.

\begin{theorem}[Gradient estimate, Neumann case]
\label{thm:neumann-error-h1}
Under the assumptions of Theorem~\ref{thm:neumann-error-l2}, the smoothed
TNN output $S_\delta u_{\delta,p,N}$ satisfies
\begin{equation}
\label{eq:neumann-H1-estimate-1}
\begin{aligned}
&\|\nabla u_{\delta,N}-\nabla S_\delta u_{\delta,p,N}\|_{L^2(\Omega)}
\\
&\qquad\le
C\!\left(
\varepsilon_f
+\frac{\varepsilon_g}{\sqrt{\delta}}
+\varepsilon_u
+\eta_{\mathrm{opt}}\right)
+
C\!\left(\delta\,\|f\|_{L^2(\Omega)}
+\sqrt{\delta}\,\|g\|_{L^2(\partial\Omega)}\right) .
\end{aligned}
\end{equation}
Moreover,
\begin{equation}
\label{eq:neumann-H1-estimate-2}
\begin{aligned}
&\|\nabla u_{\mathrm{loc},N}-\nabla S_\delta u_{\delta,p,N}\|_{L^2(\Omega)}
\\
&\qquad\le
C\!\left(
\varepsilon_f
+\frac{\varepsilon_g}{\sqrt{\delta}}
+\varepsilon_u
+\eta_{\mathrm{opt}}\right)
+
C\!\left(\delta\,\|f\|_{H^1(\Omega)}
+\sqrt{\delta}\,\|g\|_{H^{3/2}(\partial\Omega)}\right) .
\end{aligned}
\end{equation}
Here $S_\delta$ is the smoothing operator defined in \eqref{eq:mollifier}.
\end{theorem}

\begin{remark}
Theorem~\ref{thm:dirichlet-error} provides only an $L^2$ estimate in the
Dirichlet case. Numerical experiments (see
Section~\ref{sec:numerical-tests}) suggest that the corresponding
$H^1$ convergence is also achieved, although a rigorous proof remains
open.
\end{remark}

The proofs of
Theorems~\ref{thm:dirichlet-error}--\ref{thm:neumann-error-h1} are given
in Section~\ref{sec:proof}.

\section{Proof of the main result}
\label{sec:proof}

\subsection{Proof of Theorem~\ref{thm:dirichlet-error}}
\label{subsec:proof-l2}

Define the nonlocal energy functional for the Dirichlet problem:
\begin{align}
E_\delta(u)
&:=
\frac{1}{2\delta^2}
\!\int_\Omega\!\!\int_\Omega\!
R_\delta(\xx,\yy)\bigl(u(\xx)-u(\yy)\bigr)^{2}
\,\d\xx\,\d\yy
+
\frac{2}{\delta}
\!\int_\Omega\!\!\int_{\partial\Omega}\!
\bar R_\delta(\xx,\yy)\,u(\xx)^{2}
\,\d\xx\,\d S_\yy
\notag \\
&=: E_{\delta,1}(u)+E_{\delta,2}(u).
\label{eq:dirichlet-energy}
\end{align}
The modified loss can be written as
\begin{equation}
\label{eq:dirichlet-loss-energy-linear}
\tilde{\mathcal{L}}_\delta(u)
=
\frac{1}{2}E_\delta(u)-\Phi_\delta^{\tilde f,\tilde g}(u),
\end{equation}
where the data-dependent linear functional is
\begin{align}
\Phi_\delta^{\tilde f,\tilde g}(u)
&:=
\!\int_\Omega\!\!\int_\Omega\!
\bar R_\delta(\xx,\yy)\,u(\xx)\tilde f(\yy)\,\d\xx\,\d\yy
+
\frac{2}{\delta}
\!\int_\Omega\!\!\int_{\partial\Omega}\!
\bar R_\delta(\xx,\yy)\,u(\xx)\tilde g(\yy)\,\d\xx\,\d S_\yy .
\label{eq:dirichlet-linear-functional}
\end{align}
Thus \eqref{eq:dirichlet-loss-energy-linear} separates the quadratic energy
from the scalar linear functional generated by the data through their
preconditioned surrogates $\tilde f$ and $\tilde g$. We will use the
following two lemmas to bound $E_\delta$.

\begin{lemma}
\label{lem:dirichlet-energy-bulk}
There exists a constant $C>0$, independent of $\delta$, such that
\begin{equation}
\label{eq:dirichlet-energy-bound-1}
E_{\delta,1}(u) \le C\,\|u\|_{H^1(\Omega)}^{2}
\end{equation}
for all $u\in H^1(\Omega)$.
\end{lemma}

The proof of Lemma~\ref{lem:dirichlet-energy-bulk} uses the following
standard shift estimate in Sobolev spaces.

\begin{lemma}[{\cite{bourgain2001another}}]
\label{lem:sobolev-shift}
There exists a constant $C=C(d,p)$ such that
\begin{equation}
\label{eq:sobolev-shift}
\left(\int_{\RR^d}|f(\xx+\mathbf{h})-f(\xx)|^p\,\d\xx\right)^{1/p}
\le
C\,|\mathbf{h}|\,\|f\|_{W^{1,p}(\Omega)}
\end{equation}
for all $f\in W^{1,p}(\Omega)$ (extended to $W^{1,p}(\RR^d)$ via the Sobolev
extension theorem) and all $\mathbf{h}\in\RR^d$.
\end{lemma}

\begin{proof}[Proof of Lemma~\ref{lem:dirichlet-energy-bulk}]
By the Sobolev extension theorem, we may assume $u\in H^1(\RR^d)$ with
$\|u\|_{H^1(\RR^d)}\le C\|u\|_{H^1(\Omega)}$. After the change of variables
$\yy=\xx+\delta\zz$,
\begin{align*}
E_{\delta,1}(u)
&\le
\frac{1}{2\delta^2}
\int_{\RR^d}\int_{\RR^d}
R_\delta(\xx,\yy)\bigl(u(\xx)-u(\yy)\bigr)^2
\,\d\xx\,\d\yy
\\
&=
\frac{\alpha_d}{2\delta^2}
\int_{\RR^d}
\int_{\RR^d}
R\!\left(\tfrac{|\zz|^2}{4}\right)
\bigl(u(\xx)-u(\xx+\delta\zz)\bigr)^2
\,\d\xx\,\d\zz .
\end{align*}
The scaled second moment
$\int_{\RR^d}R(|\zz|^2/4)|\zz|^2\,\d\zz$ is finite under the present kernel
assumptions. Applying Lemma~\ref{lem:sobolev-shift} with $p=2$ and
$\mathbf{h}=\delta\zz$,
\begin{align*}
E_{\delta,1}(u)
&\le
\frac{\alpha_d}{2\delta^2}
\int_{\RR^d}
R\!\left(\tfrac{|\zz|^2}{4}\right)
C^2\delta^2|\zz|^2\,\|u\|_{H^1(\Omega)}^2
\,\d\zz
\le
C\,\|u\|_{H^1(\Omega)}^2,
\end{align*}
which completes the proof.
\end{proof}

\begin{lemma}
\label{lem:dirichlet-energy-boundary}
There exists a constant $C>0$, independent of $\delta$, such that
\begin{equation}
\label{eq:dirichlet-energy-bound-2}
E_{\delta,2}(u) \le \frac{C}{\delta}\,\|u\|_{H^1(\Omega)}^2
\end{equation}
for all $u\in H^1(\Omega)$.
\end{lemma}

\begin{proof}[Proof of Lemma~\ref{lem:dirichlet-energy-boundary}]
We use the standard boundary-layer trace estimate
\begin{equation*}
\int_\Omega
\left(\int_{\partial\Omega}\bar R_\delta(\xx,\yy)\,\d S_\yy\right)
u(\xx)^2\,\d\xx
\le C\|u\|_{H^1(\Omega)}^2 .
\end{equation*}
For compactly supported kernels this is the usual thin-strip estimate
\cite[Lemma~A.1]{meng2026asymptotically}; the Gaussian kernel gives the same
estimate by its exponential decay. Therefore
\begin{align*}
E_{\delta,2}(u)
&=
\frac{2}{\delta}
\int_\Omega\int_{\partial\Omega}
\bar R_\delta(\xx,\yy)\,u(\xx)^2\,\d\xx\,\d S_\yy
\le
\frac{C}{\delta}\,\|u\|_{H^1(\Omega)}^2 .
\end{align*}
\end{proof}

Combining Lemmas~\ref{lem:dirichlet-energy-bulk}
and~\ref{lem:dirichlet-energy-boundary},
\begin{equation}
\label{eq:dirichlet-energy-estimate}
E_\delta(u) \le \frac{C}{\delta}\,\|u\|_{H^1(\Omega)}^{2}
\qquad\text{for all }u\in H^1(\Omega).
\end{equation}

\begin{proof}[Proof of Theorem~\ref{thm:dirichlet-error}]
Let $\tilde u_\delta\in H^1(\Omega)$ be the exact minimizer of
$\tilde{\mathcal{L}}_\delta$ over $H^1(\Omega)$, i.e., the solution of
\eqref{eq:dirichlet-nonlocal} with $f$ and $g$ replaced by $\tilde f$ and
$\tilde g$. Since $u_\delta$ solves \eqref{eq:dirichlet-nonlocal} with data
$f$, $g$, the difference $u_\delta-\tilde u_\delta$ satisfies the same
equation with data $f-\tilde f$, $g-\tilde g$. Applying
Proposition~\ref{prop:dirichlet-wellposedness} and
\eqref{eq:dirichlet-precondition},
\begin{equation}
\label{eq:dirichlet-tilde-h1}
\|u_\delta-\tilde u_\delta\|_{H^1(\Omega)}
\le
C\!\left(\|f-\tilde f\|_{L^2(\Omega)}
+\frac{\|g-\tilde g\|_{L^2(\partial\Omega)}}{\sqrt{\delta}}\right)
\le
C\!\left(\varepsilon_f+\frac{\varepsilon_g}{\sqrt\delta}\right).
\end{equation}
Together with \eqref{eq:dirichlet-tolerances},
\begin{equation}
\label{eq:dirichlet-error-actual-and-real}
\|\bar u_\delta-\tilde u_\delta\|_{H^1(\Omega)}
\le
C\!\left(\varepsilon_f+\frac{\varepsilon_g}{\sqrt\delta}\right)+\varepsilon_u .
\end{equation}

Since $\tilde u_\delta$ is the minimizer of $\tilde{\mathcal{L}}_\delta$ in
$H^1(\Omega)$ and the loss has the decomposition
\eqref{eq:dirichlet-loss-energy-linear}, the second-order expansion at
$\tilde u_\delta$ gives
\[
\tilde{\mathcal{L}}_\delta(v)
= \tilde{\mathcal{L}}_\delta(\tilde u_\delta)
+ \tfrac{1}{2}E_\delta(v-\tilde u_\delta)
\qquad\text{for all }v\in H^1(\Omega).
\]
Applying this to $v=\bar u_\delta$ and using \eqref{eq:dirichlet-energy-estimate}
and \eqref{eq:dirichlet-error-actual-and-real},
\begin{equation}
\label{eq:dirichlet-bar-loss}
\tilde{\mathcal{L}}_\delta(\bar u_\delta)
\le
\tilde{\mathcal{L}}_\delta(\tilde u_\delta)
+
C\!\left(\frac{\varepsilon_f^2}{\delta}
+\frac{\varepsilon_g^2}{\delta^2}
+\frac{\varepsilon_u^2}{\delta}\right).
\end{equation}
Since $p\ge p_u$, we have $\bar u_\delta\in\mathcal{V}_p^d$. Because
$\tilde u_{\delta,p}$ minimizes $\tilde{\mathcal{L}}_\delta$ over
$\mathcal{V}_p^d$, and using the definition of $\eta_{\mathrm{opt}}$ in
\eqref{eq:opt-error-dirichlet},
\begin{equation}
\label{eq:dirichlet-final-1}
\tilde{\mathcal{L}}_\delta(u_{\delta,p})
\le
\tilde{\mathcal{L}}_\delta(\tilde u_\delta)
+
C\!\left(\frac{\varepsilon_f^2}{\delta}
+\frac{\varepsilon_g^2}{\delta^2}
+\frac{\varepsilon_u^2}{\delta}\right)
+\eta_{\mathrm{opt}}^{\,2}.
\end{equation}

On the other hand, by the coercivity of $E_\delta$
\cite[Lemma~3.1]{meng2026asymptotically},
\begin{equation}
\label{eq:dirichlet-coercivity}
E_\delta(u) \ge C\,\|u\|_{L^2(\Omega)}^{2}
\qquad\text{for all }u\in L^2(\Omega).
\end{equation}
Applying the same second-order expansion to $u_{\delta,p}$:
\begin{equation}
\label{eq:dirichlet-final-2}
\tilde{\mathcal{L}}_\delta(u_{\delta,p})
= \tilde{\mathcal{L}}_\delta(\tilde u_\delta)
+ \tfrac{1}{2}E_\delta(u_{\delta,p}-\tilde u_\delta)
\ge
\tilde{\mathcal{L}}_\delta(\tilde u_\delta)
+ C\,\|u_{\delta,p}-\tilde u_\delta\|_{L^2(\Omega)}^{2}.
\end{equation}
Combining \eqref{eq:dirichlet-final-1} and \eqref{eq:dirichlet-final-2},
\[
\|u_{\delta,p}-\tilde u_\delta\|_{L^2(\Omega)}^{2}
\le
C\!\left(\frac{\varepsilon_f^2}{\delta}
+\frac{\varepsilon_g^2}{\delta^2}
+\frac{\varepsilon_u^2}{\delta}\right)
+\eta_{\mathrm{opt}}^{\,2},
\]
hence
\begin{equation}
\label{eq:dirichlet-final-3}
\|u_{\delta,p}-\tilde u_\delta\|_{L^2(\Omega)}
\le
C\!\left(\frac{\varepsilon_f}{\sqrt\delta}
+\frac{\varepsilon_g}{\delta}
+\frac{\varepsilon_u}{\sqrt\delta}
+\eta_{\mathrm{opt}}\right).
\end{equation}
By \eqref{eq:dirichlet-tilde-h1}, the triangle inequality, and $0<\delta\le1$,
\begin{equation}
\label{eq:dirichlet-final-4}
\|u_{\delta,p}-u_\delta\|_{L^2(\Omega)}
\le
C\!\left(\frac{\varepsilon_f}{\sqrt\delta}
+\frac{\varepsilon_g}{\delta}
+\frac{\varepsilon_u}{\sqrt\delta}
+\eta_{\mathrm{opt}}\right),
\end{equation}
which proves \eqref{eq:dirichlet-estimate-1}. For \eqref{eq:dirichlet-estimate-2},
observe that
Proposition~\ref{prop:dirichlet-locallimit} and the continuous embedding
$H^1(\Omega)\hookrightarrow L^2(\Omega)$ give
\begin{align*}
\|u_{\mathrm{loc}}-u_\delta\|_{L^2(\Omega)}
&\le
\|u_{\mathrm{loc}}-u_\delta\|_{H^1(\Omega)}
\le
C\sqrt{\delta}\,\|u_{\mathrm{loc}}\|_{H^3(\Omega)}
\\
&\le
C\sqrt{\delta}\!\left(\|f\|_{H^1(\Omega)}
+\|g\|_{H^{5/2}(\partial\Omega)}\right),
\end{align*}
where the last step uses the elliptic regularity estimate of
Remark~\ref{rem:dirichlet-regularity}. Combining this with
\eqref{eq:dirichlet-final-4} via the triangle inequality yields
\eqref{eq:dirichlet-estimate-2}.
\end{proof}

\subsection{Proof of Theorems~\ref{thm:neumann-error-l2}
and~\ref{thm:neumann-error-h1}}
\label{subsec:proof-h1}

For the Neumann problem, define the energy functional
\begin{align}
E_{\delta,N}(u)
&:=
\frac{1}{2\delta^{2}}
\!\int_\Omega\!\!\int_\Omega\!
R_\delta(\xx,\yy)\bigl(u(\xx)-u(\yy)\bigr)^{2}
\,\d\xx\,\d\yy
+
\!\int_\Omega\!\!\int_\Omega\!
\bar R_\delta(\xx,\yy)\,u(\xx)u(\yy)
\,\d\xx\,\d\yy
\notag \\
&=: E_{\delta,1}(u)+E_{\delta,3}(u),
\label{eq:neumann-energy}
\end{align}
so that
\begin{equation}
\label{eq:neumann-loss-energy-linear}
\tilde{\mathcal{L}}_{\delta,N}(u)
=
\frac{1}{2}E_{\delta,N}(u)-\Phi_{\delta,N}^{\tilde f,\tilde g}(u),
\end{equation}
where the data-dependent linear functional is
\begin{align}
\Phi_{\delta,N}^{\tilde f,\tilde g}(u)
&:=
\!\int_\Omega\!\!\int_\Omega\!
\bar R_\delta(\xx,\yy)\,u(\xx)\tilde f(\yy)\,\d\xx\,\d\yy
+
2\!\int_\Omega\!\!\int_{\partial\Omega}\!
\bar R_\delta(\xx,\yy)\,u(\xx)\tilde g(\yy)\,\d\xx\,\d S_\yy .
\label{eq:neumann-linear-functional}
\end{align}

By Cauchy--Schwarz and \eqref{eq:kernel-est-2},
\begin{align*}
E_{\delta,3}(u)
&\le
\!\left(\int_\Omega\!\!\int_\Omega\!
\bar R_\delta(\xx,\yy)\,u(\xx)^2\,\d\xx\,\d\yy\right)^{\!1/2}
\\
&\quad\times
\!\left(\int_\Omega\!\!\int_\Omega\!
\bar R_\delta(\xx,\yy)\,u(\yy)^2\,\d\xx\,\d\yy\right)^{\!1/2}
\le C\,\|u\|_{L^2(\Omega)}^{2}.
\end{align*}
Combining with Lemma~\ref{lem:dirichlet-energy-bulk},
\begin{equation}
\label{eq:neumann-energy-estimate}
E_{\delta,N}(u) \le C\,\|u\|_{H^1(\Omega)}^{2}
\qquad\text{for all }u\in H^1(\Omega).
\end{equation}
The key improvement over \eqref{eq:dirichlet-energy-estimate} is the
absence of the $\delta^{-1}$ factor.

\begin{proof}[Proof of Theorem~\ref{thm:neumann-error-l2}]
Let $\tilde u_{\delta,N}\in H^1(\Omega)$ be the exact minimizer of
$\tilde{\mathcal{L}}_{\delta,N}$ over $H^1(\Omega)$, i.e., the solution of
\eqref{eq:neumann-nonlocal} with $f$ and $g$ replaced by $\tilde f$ and
$\tilde g$. Applying Proposition~\ref{prop:neumann-wellposedness} and
\eqref{eq:neumann-precondition} to the equation satisfied by
$u_{\delta,N}-\tilde u_{\delta,N}$,
\begin{equation}
\label{eq:neumann-error-actual-and-real}
\|u_{\delta,N}-\tilde u_{\delta,N}\|_{H^1(\Omega)}
\le
C\!\left(\varepsilon_f+\frac{\varepsilon_g}{\sqrt\delta}\right).
\end{equation}
Together with the assumption on $\bar u_{\delta,N}$ in
Theorem~\ref{thm:neumann-error-l2},
\begin{equation}
\label{eq:neumann-bar-h1}
\|\bar u_{\delta,N}-\tilde u_{\delta,N}\|_{H^1(\Omega)}
\le
C\!\left(\varepsilon_f+\frac{\varepsilon_g}{\sqrt\delta}\right)+\varepsilon_u .
\end{equation}

By the second-order expansion at the minimizer $\tilde u_{\delta,N}$,
\[
\tilde{\mathcal{L}}_{\delta,N}(v)
= \tilde{\mathcal{L}}_{\delta,N}(\tilde u_{\delta,N})
+ \tfrac{1}{2}E_{\delta,N}(v-\tilde u_{\delta,N})
\qquad\text{for all }v\in H^1(\Omega).
\]
Applied to $v=\bar u_{\delta,N}$, and using \eqref{eq:neumann-energy-estimate}
and \eqref{eq:neumann-bar-h1},
\begin{equation}
\label{eq:neumann-final-bar}
\tilde{\mathcal{L}}_{\delta,N}(\bar u_{\delta,N})
\le
\tilde{\mathcal{L}}_{\delta,N}(\tilde u_{\delta,N})
+
C\!\left(\varepsilon_f^2+\frac{\varepsilon_g^2}{\delta}+\varepsilon_u^2\right).
\end{equation}
Since $p\ge p_u$, we have $\bar u_{\delta,N}\in\mathcal{V}_p^d$. Using the
minimality of $\tilde u_{\delta,p,N}$ over $\mathcal{V}_p^d$ and
\eqref{eq:opt-error-neumann},
\begin{equation}
\label{eq:neumann-final-1}
\tilde{\mathcal{L}}_{\delta,N}(u_{\delta,p,N})
\le
\tilde{\mathcal{L}}_{\delta,N}(\tilde u_{\delta,N})
+C\!\left(\varepsilon_f^2+\frac{\varepsilon_g^2}{\delta}+\varepsilon_u^2\right)
+\eta_{\mathrm{opt}}^{\,2}.
\end{equation}

By the coercivity of $E_{\delta,N}$
\cite[Lemma~3.1]{meng2026asymptotically},
\begin{equation}
\label{eq:neumann-l2-coercivity}
E_{\delta,N}(u) \ge C\,\|u\|_{L^2(\Omega)}^{2}
\qquad\text{for all }u\in L^2(\Omega).
\end{equation}
The same expansion applied to $u_{\delta,p,N}$ gives
\begin{equation}
\label{eq:neumann-final-2}
\begin{aligned}
\tilde{\mathcal{L}}_{\delta,N}(u_{\delta,p,N})
&= \tilde{\mathcal{L}}_{\delta,N}(\tilde u_{\delta,N})
+ \tfrac{1}{2}E_{\delta,N}(u_{\delta,p,N}-\tilde u_{\delta,N})
\\
&\ge
\tilde{\mathcal{L}}_{\delta,N}(\tilde u_{\delta,N})
+C\,\|u_{\delta,p,N}-\tilde u_{\delta,N}\|_{L^2(\Omega)}^{2}.
\end{aligned}
\end{equation}
Combining \eqref{eq:neumann-final-1} and \eqref{eq:neumann-final-2},
\[
\|u_{\delta,p,N}-\tilde u_{\delta,N}\|_{L^2(\Omega)}
\le
C\!\left(\varepsilon_f+\frac{\varepsilon_g}{\sqrt\delta}+\varepsilon_u+\eta_{\mathrm{opt}}\right).
\]
By \eqref{eq:neumann-error-actual-and-real} and the triangle inequality,
\begin{equation}
\label{eq:neumann-final-3}
\|u_{\delta,p,N}-u_{\delta,N}\|_{L^2(\Omega)}
\le
C\!\left(\varepsilon_f+\frac{\varepsilon_g}{\sqrt\delta}+\varepsilon_u+\eta_{\mathrm{opt}}\right),
\end{equation}
which proves \eqref{eq:neumann-estimate-1}. Estimate \eqref{eq:neumann-estimate-2}
follows from \eqref{eq:neumann-final-3} and
Proposition~\ref{prop:neumann-locallimit}.
\end{proof}

\begin{proof}[Proof of Theorem~\ref{thm:neumann-error-h1}]
By the gradient part of the same coercivity estimate
\cite[Lemma~3.1]{meng2026asymptotically},
\begin{equation}
\label{eq:neumann-sdelta-coercivity}
E_{\delta,N}(u) \ge C\,\|\nabla S_\delta u\|_{L^2(\Omega)}^{2}
\qquad\text{for all }u\in L^2(\Omega).
\end{equation}
From \eqref{eq:neumann-final-2} and \eqref{eq:neumann-sdelta-coercivity},
\begin{equation}
\label{eq:neumann-final-4}
\tilde{\mathcal{L}}_{\delta,N}(u_{\delta,p,N})
\ge
\tilde{\mathcal{L}}_{\delta,N}(\tilde u_{\delta,N})
+C\,\|\nabla S_\delta(u_{\delta,p,N}-\tilde u_{\delta,N})\|_{L^2(\Omega)}^{2}.
\end{equation}
Combining \eqref{eq:neumann-final-1} and \eqref{eq:neumann-final-4},
\begin{equation}
\label{eq:neumann-grad-tnn-to-tilde}
\|\nabla S_\delta(u_{\delta,p,N}-\tilde u_{\delta,N})\|_{L^2(\Omega)}
\le
C\!\left(\varepsilon_f+\frac{\varepsilon_g}{\sqrt\delta}+\varepsilon_u+\eta_{\mathrm{opt}}\right).
\end{equation}

It remains to estimate $\|\nabla\tilde u_{\delta,N}-\nabla S_\delta\tilde u_{\delta,N}\|_{L^2(\Omega)}$.
Since $\tilde u_{\delta,N}$ solves \eqref{eq:neumann-nonlocal} with $\tilde f$, $\tilde g$,
the decomposition \eqref{eq:neumann-decomposition} gives
\begin{align}
\tilde u_{\delta,N}(\xx) - S_\delta\tilde u_{\delta,N}(\xx)
&=
\frac{\delta^{2}}{w_\delta(\xx)}
\!\left(
\int_\Omega\!\bar R_\delta(\xx,\yy)\tilde f(\yy)\,\d\yy
+2\int_{\partial\Omega}\!\bar R_\delta(\xx,\yy)\tilde g(\yy)\,\d S_\yy
\right.
\notag \\
&\hspace{3cm}
\left.
-\int_\Omega\!\bar R_\delta(\xx,\yy)\tilde u_{\delta,N}(\yy)\,\d\yy
\right)
\notag \\
&=: I_1(\xx)+I_2(\xx)+I_3(\xx).
\label{eq:neumann-tilde-decomp}
\end{align}
By \eqref{eq:I1-est}--\eqref{eq:I3-est} in Appendix~\ref{appendix:neumann-wellposedness},
with $\tilde f$, $\tilde g$, $\tilde u_{\delta,N}$ in place of $f$, $g$, $u_{\delta,N}$,
and applying Proposition~\ref{prop:neumann-wellposedness} together with
\eqref{eq:neumann-precondition} to bound $\|\tilde u_{\delta,N}\|_{L^2(\Omega)}$,
\begin{align}
\|\nabla I_1\|_{L^2(\Omega)}
&\le C\delta\,\|\tilde f\|_{L^2(\Omega)}
\le C\delta\bigl(\|f\|_{L^2(\Omega)}+\varepsilon_f\bigr),
\notag \\
\|\nabla I_2\|_{L^2(\Omega)}
&\le C\sqrt\delta\,\|\tilde g\|_{L^2(\partial\Omega)}
\le C\sqrt\delta\bigl(\|g\|_{L^2(\partial\Omega)}+\varepsilon_g\bigr),
\notag \\
\|\nabla I_3\|_{L^2(\Omega)}
&\le C\delta\,\|\tilde u_{\delta,N}\|_{L^2(\Omega)}
\le C\delta\bigl(\|f\|_{L^2(\Omega)}+\varepsilon_f\bigr)
  +C\sqrt\delta\bigl(\|g\|_{L^2(\partial\Omega)}+\varepsilon_g\bigr).
\notag
\end{align}
Summing and absorbing the $\varepsilon_f$, $\varepsilon_g$ terms,
\begin{equation}
\label{eq:neumann-final-5}
\|\nabla\tilde u_{\delta,N}-\nabla S_\delta\tilde u_{\delta,N}\|_{L^2(\Omega)}
\le
C\!\left(\delta\,\|f\|_{L^2(\Omega)}
+\sqrt\delta\,\|g\|_{L^2(\partial\Omega)}\right)
+C\!\left(\delta\,\varepsilon_f+\sqrt\delta\,\varepsilon_g\right).
\end{equation}

By the triangle inequality,
\begin{align*}
&\|\nabla u_{\delta,N}-\nabla S_\delta u_{\delta,p,N}\|_{L^2(\Omega)}
\\
&\qquad\le
\|\nabla(u_{\delta,N}-\tilde u_{\delta,N})\|_{L^2(\Omega)}
+\|\nabla\tilde u_{\delta,N}-\nabla S_\delta\tilde u_{\delta,N}\|_{L^2(\Omega)}
\\
&\qquad\quad
+\|\nabla S_\delta(\tilde u_{\delta,N}-u_{\delta,p,N})\|_{L^2(\Omega)}.
\end{align*}
The first term is bounded by \eqref{eq:neumann-error-actual-and-real}, the
second by \eqref{eq:neumann-final-5}, and the third by
\eqref{eq:neumann-grad-tnn-to-tilde}. Since $\delta\le 1$, we have
$\delta\,\varepsilon_f\le \varepsilon_f$ and
$\sqrt\delta\,\varepsilon_g\le \varepsilon_g/\sqrt\delta$, so the
$\varepsilon$-terms in \eqref{eq:neumann-final-5} are absorbed into those
already present, yielding \eqref{eq:neumann-H1-estimate-1}. Estimate
\eqref{eq:neumann-H1-estimate-2} follows from \eqref{eq:neumann-H1-estimate-1}
and Proposition~\ref{prop:neumann-locallimit}.
\end{proof}

\section{Numerical tests}
\label{sec:numerical-tests}

In this section we present numerical experiments to illustrate the method
described in Section~\ref{sec:main-result}. All experiments were run on a
single NVIDIA RTX 3090 GPU with 24\,GB RAM.

\subsection{Experiments with tensor-product data}
\label{subsec:tensor-product-experiments}

We first test the TNN-based solver on problems whose source term and boundary data
are of tensor-product form. In this setting the data require no
preconditioning, so $\varepsilon_f=\varepsilon_g=0$ in
Theorems~\ref{thm:dirichlet-error}--\ref{thm:neumann-error-h1}. The only
nontrivial approximation error is $\varepsilon_u$, which arises from
approximating the nonlocal solution by a TNN.

\subsubsection{Dirichlet case}
\label{subsubsec:dirichlet-tensor}

We consider \eqref{eq:dirichlet-nonlocal} on $\Omega=[0,1]^d$ for
$d=3,5,10,20$, with exact local solution
\[
u_{\mathrm{loc}}(\xx)
= \sin(\pi x_1)\cdots\sin(\pi x_d) + x_1\cdots x_d .
\]
The corresponding source term and boundary data are
\begin{align*}
f(\xx)
&= -\Delta u_{\mathrm{loc}}
= d\pi^2\sin(\pi x_1)\cdots\sin(\pi x_d),
\qquad \xx\in\Omega,\\
g(\xx)
&= u_{\mathrm{loc}}\big|_{\partial\Omega}
= \prod_{i=1}^d x_i,
\qquad \xx\in\partial\Omega.
\end{align*}
Both $f$ and $g$ are already of tensor-product form, so the
preconditioning step of Section~\ref{subsec:dirichlet-workflow} is
unnecessary and we set $\tilde f=f$ and $\tilde g=g$ directly.

The network architecture, quadrature parameters, and training schedule for
each experiment are summarized in Table~\ref{tab:numerical-parameters}.
Here $p$ denotes the TNN separation rank; \#layers and \#neurons are the
depth and width of each one-dimensional subnetwork; $n_{\mathrm{sub}}$ and
$n_{\mathrm{pts}}$ are the number of subintervals and Gauss--Legendre
quadrature points per subinterval used to evaluate the loss integrals;
the columns ``Adam'' and ``L-BFGS'' record the number of training
iterations of the corresponding optimizers; and the last column reports
the total training time. As $\delta$ decreases, $n_{\mathrm{sub}}$ must be
increased to resolve the increasingly localized Gaussian factor
$\exp(-s^2(x_i-y_i)^2/(4\delta^2))$.

\begin{table}[htbp]
\centering
\caption{Network, quadrature, and training parameters for the
         tensor-product Dirichlet experiments.}
\label{tab:numerical-parameters}
\small
\setlength{\tabcolsep}{3pt}
\begin{tabular}{lcccccccc}
\hline
Experiment & $p$ & \#layers & \#neurons
           & $n_{\mathrm{sub}}$ & $n_{\mathrm{pts}}$
           & Adam & L-BFGS & Time\,(s) \\
\hline
$d=3,\;\delta=0.200$ & 40 & 3 & 150 &  40 & 16 & 3000 & 10 &  67.9 \\
$d=3,\;\delta=0.100$ & 40 & 3 & 150 &  40 & 16 & 3000 & 10 &  67.7 \\
$d=3,\;\delta=0.050$ & 40 & 3 & 150 &  64 & 16 & 5000 & 10 & 115.4 \\
$d=3,\;\delta=0.025$ & 40 & 3 & 150 & 128 & 16 & 5000 & 10 & 122.9 \\
$d=5,\;\delta=0.200$ & 40 & 3 & 200 &  40 & 16 & 3000 & 10 &  89.6 \\
$d=5,\;\delta=0.100$ & 40 & 3 & 200 &  40 & 16 & 3000 & 10 &  89.8 \\
$d=5,\;\delta=0.050$ & 40 & 3 & 200 &  64 & 16 & 5000 & 10 & 146.4 \\
$d=5,\;\delta=0.025$ & 40 & 3 & 200 & 128 & 16 & 5000 & 10 & 159.1 \\
$d=10,\;\delta=0.200$ & 60 & 3 & 200 &  40 & 16 & 3000 & 10 & 129.9 \\
$d=10,\;\delta=0.100$ & 60 & 3 & 200 &  40 & 16 & 3000 & 10 & 129.1 \\
$d=10,\;\delta=0.050$ & 60 & 3 & 200 &  64 & 16 & 6000 & 10 & 269.0 \\
$d=10,\;\delta=0.025$ & 60 & 3 & 200 & 128 & 16 & 6000 & 10 & 289.1 \\
$d=20,\;\delta=0.200$ & 80 & 3 & 250 &  40 & 16 & 3000 & 10 & 275.1 \\
$d=20,\;\delta=0.100$ & 80 & 3 & 250 &  40 & 16 & 3000 & 10 & 274.9 \\
$d=20,\;\delta=0.050$ & 80 & 3 & 250 &  64 & 16 & 6000 & 10 & 524.7 \\
$d=20,\;\delta=0.025$ & 80 & 3 & 250 & 128 & 16 & 6000 & 10 & 538.6 \\
\hline
\end{tabular}
\end{table}

Since the exact nonlocal solution $u_\delta$ of
\eqref{eq:dirichlet-nonlocal} is not available in closed form, we assess
the solution quality through the mean pointwise residual. Substituting
the TNN output $u_{\delta,p}$ into the left- and right-hand sides of
\eqref{eq:dirichlet-nonlocal}, the residual at a point $\xx$ is
\begin{align}
r(\xx;\,u_{\delta,p})
&:=
\frac{1}{\delta^{2}}
\int_\Omega R_\delta(\xx,\yy)
\bigl(u_{\delta,p}(\xx)-u_{\delta,p}(\yy)\bigr)\,\d\yy
+
\frac{2}{\delta}
\int_{\partial\Omega}
\bar{R}_\delta(\xx,\yy)\,u_{\delta,p}(\xx)\,\d S_\yy
\notag\\
&\quad
-\int_\Omega \bar{R}_\delta(\xx,\yy)f(\yy)\,\d\yy
-\frac{2}{\delta}
\int_{\partial\Omega}
\bar{R}_\delta(\xx,\yy)g(\yy)\,\d S_\yy .
\label{eq:dirichlet-residual}
\end{align}
We sample $N=1000$ points $\{\xx_i\}_{i=1}^N$ uniformly in $\Omega$ and
report the mean residual
\[
r_{\mathrm{mean}}
:= \frac{1}{N}\sum_{i=1}^N \bigl|r(\xx_i;\,u_{\delta,p})\bigr|.
\]
Since the local solution $u_{\mathrm{loc}}$ is known explicitly, the
errors $\|u_{\mathrm{loc}}-u_{\delta,p}\|_{L^2(\Omega)}$ and
$\|u_{\mathrm{loc}}-u_{\delta,p}\|_{H^1(\Omega)}$ are computed directly.

Figure~\ref{fig:dirichlet-residual} shows $r_{\mathrm{mean}}$ as a
function of the training iteration for $d=3,5,10,20$ at fixed
$\delta=0.05$. In all cases the residual decreases to a small value,
confirming that $u_{\delta,p}$ closely satisfies the nonlocal equation.
Figure~\ref{fig:dirichlet-errors} reports
$\|u_{\mathrm{loc}}-u_{\delta,p}\|_{L^2(\Omega)}$ and
$\|u_{\mathrm{loc}}-u_{\delta,p}\|_{H^1(\Omega)}$ as functions of $\delta$
for each dimension. The $L^2$ error and empirically also the $H^1$ error exhibit a numerical convergence rate of
order nearly $1$ in $\delta$, which is better than the $O(\sqrt{\delta})$
prediction of \eqref{eq:dirichlet-estimate-2}; a rigorous theoretical
explanation of this observation is left to future work.

\begin{figure}[htbp]
\centering
\begin{subfigure}[b]{0.48\textwidth}
  \centering
  \includegraphics[width=\textwidth]{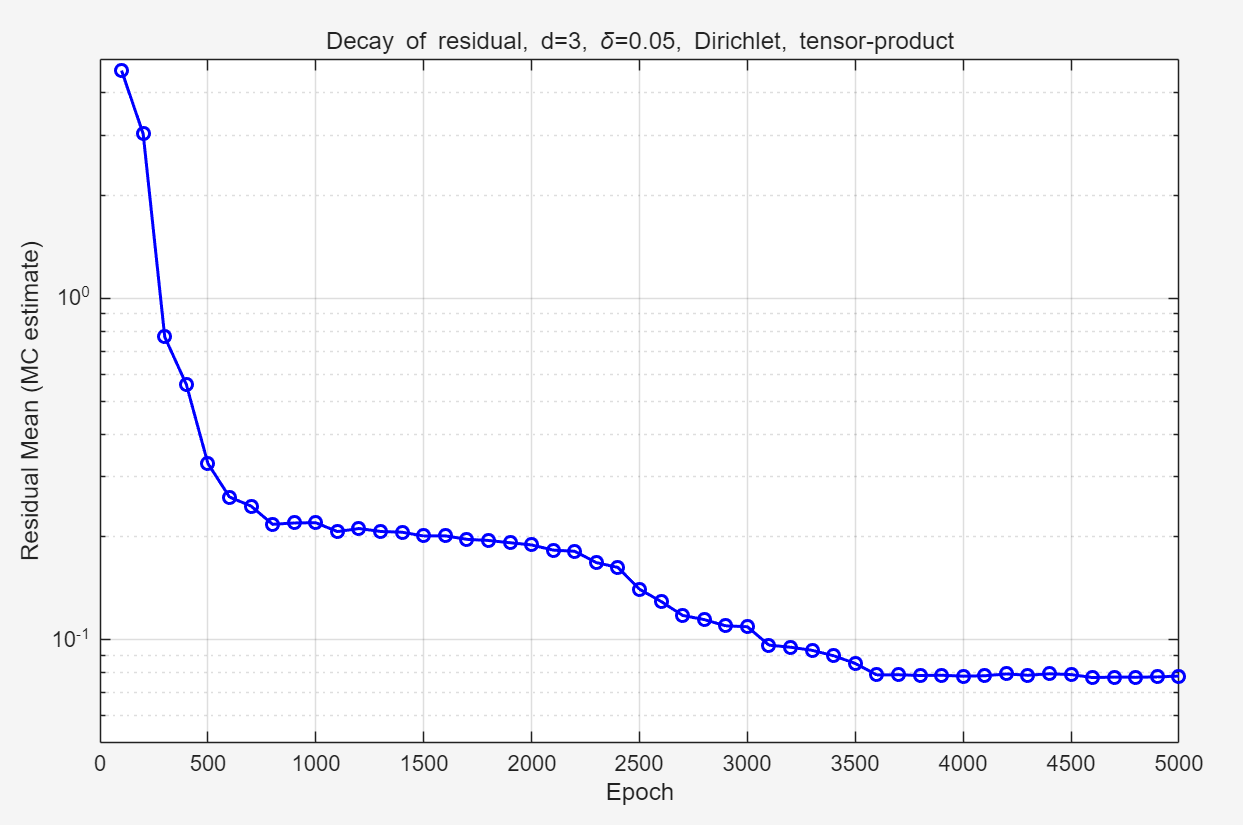}
  \caption{$d=3$}
\end{subfigure}
\hfill
\begin{subfigure}[b]{0.48\textwidth}
  \centering
  \includegraphics[width=\textwidth]{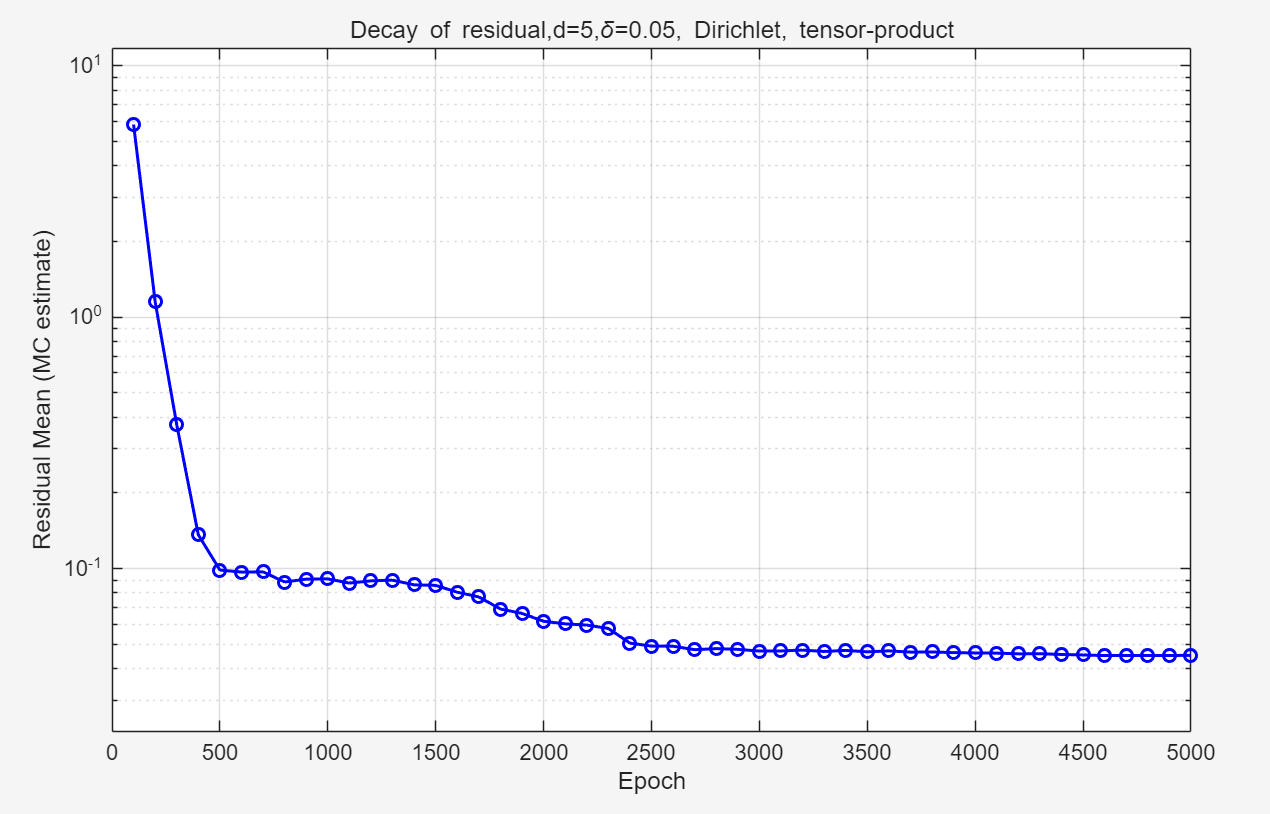}
  \caption{$d=5$}
\end{subfigure}

\vspace{0.5em}

\begin{subfigure}[b]{0.48\textwidth}
  \centering
  \includegraphics[width=\textwidth]{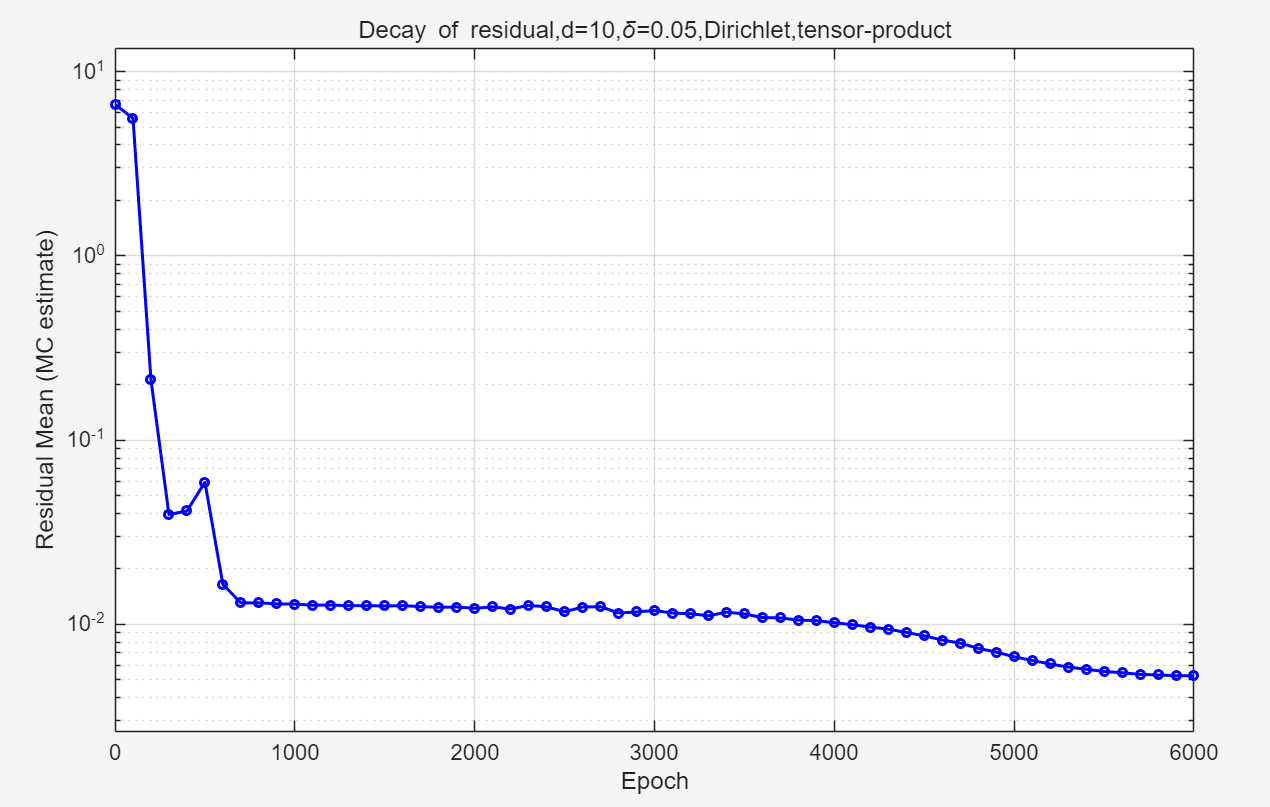}
  \caption{$d=10$}
\end{subfigure}
\hfill
\begin{subfigure}[b]{0.48\textwidth}
  \centering
  \includegraphics[width=\textwidth]{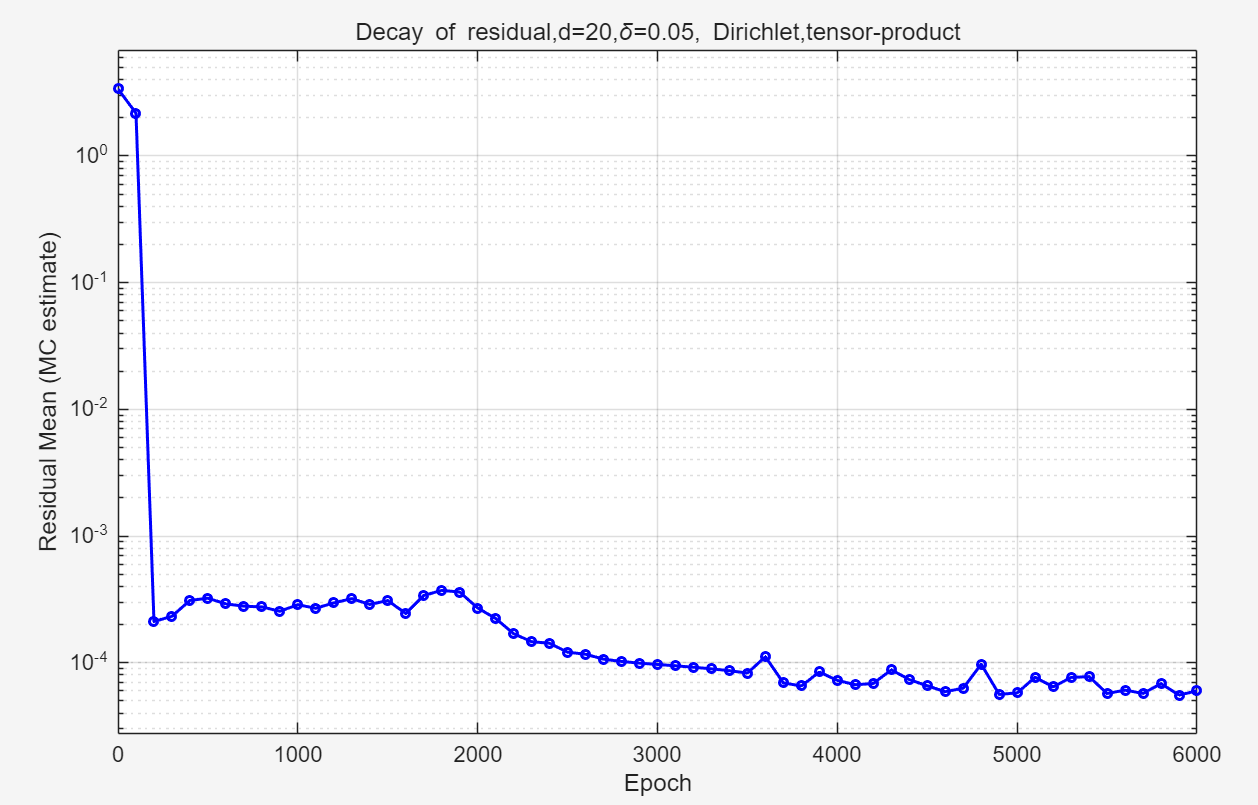}
  \caption{$d=20$}
\end{subfigure}
\caption{Mean pointwise residual $r_{\mathrm{mean}}$ versus training
         iteration for the Dirichlet case with tensor-product data
         ($\delta=0.05$, $d=3,5,10,20$).}
\label{fig:dirichlet-residual}
\end{figure}

\begin{figure}[htbp]
\centering
\begin{subfigure}[b]{0.48\textwidth}
  \centering
  \includegraphics[width=\textwidth]{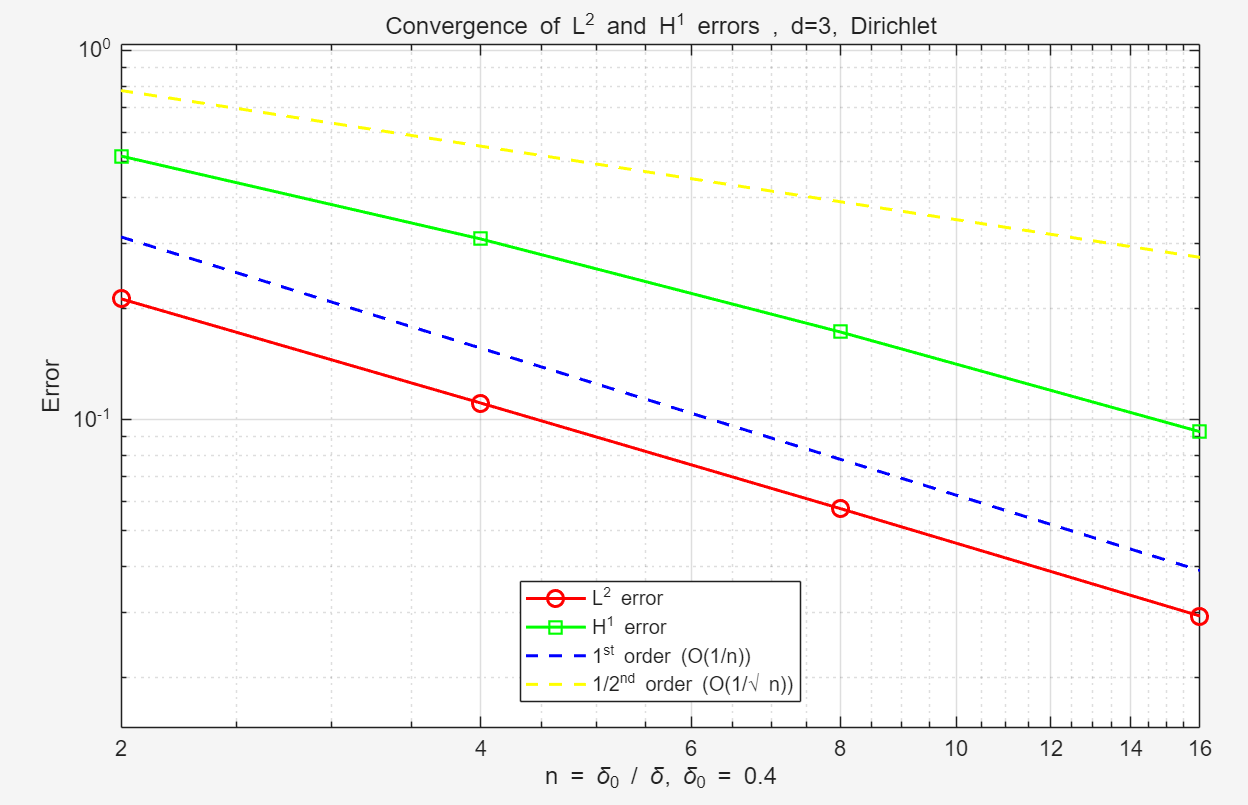}
  \caption{$d=3$}
\end{subfigure}
\hfill
\begin{subfigure}[b]{0.48\textwidth}
  \centering
  \includegraphics[width=\textwidth]{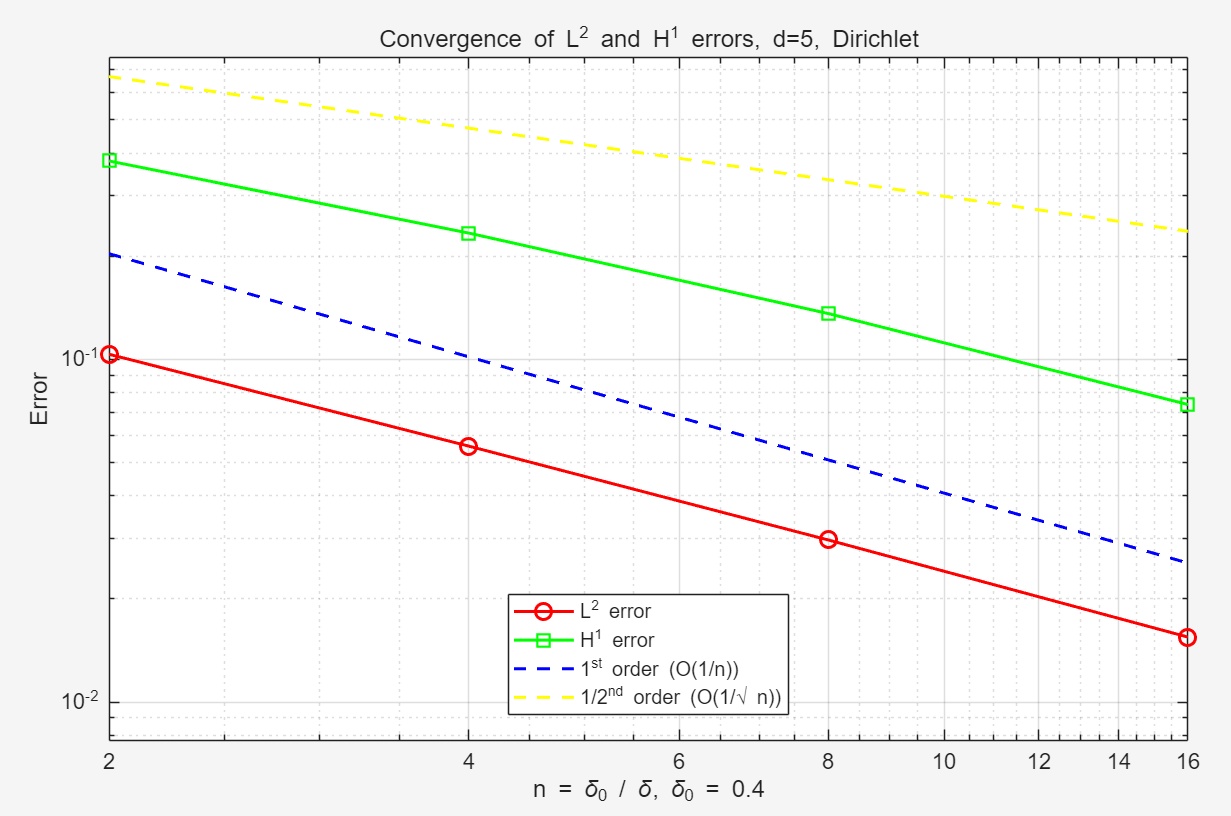}
  \caption{$d=5$}
\end{subfigure}

\vspace{0.5em}

\begin{subfigure}[b]{0.48\textwidth}
  \centering
  \includegraphics[width=\textwidth]{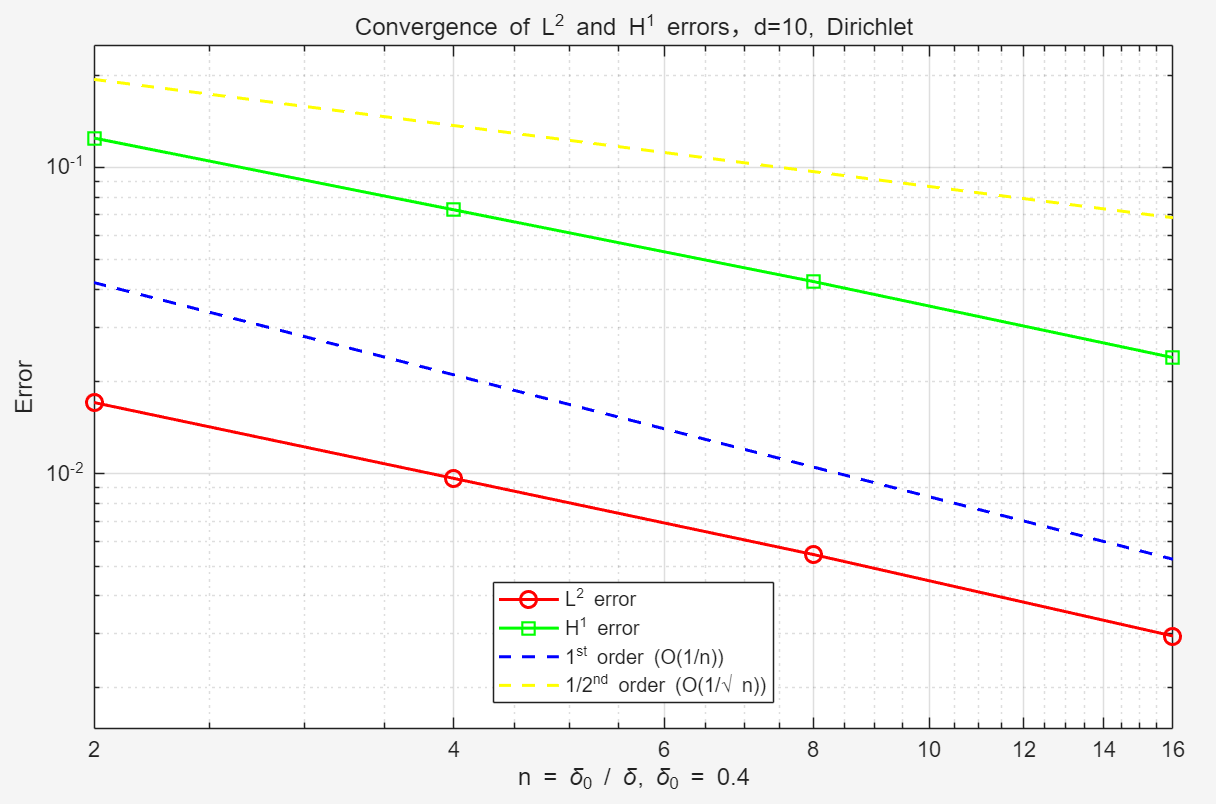}
  \caption{$d=10$}
\end{subfigure}
\hfill
\begin{subfigure}[b]{0.48\textwidth}
  \centering
  \includegraphics[width=\textwidth]{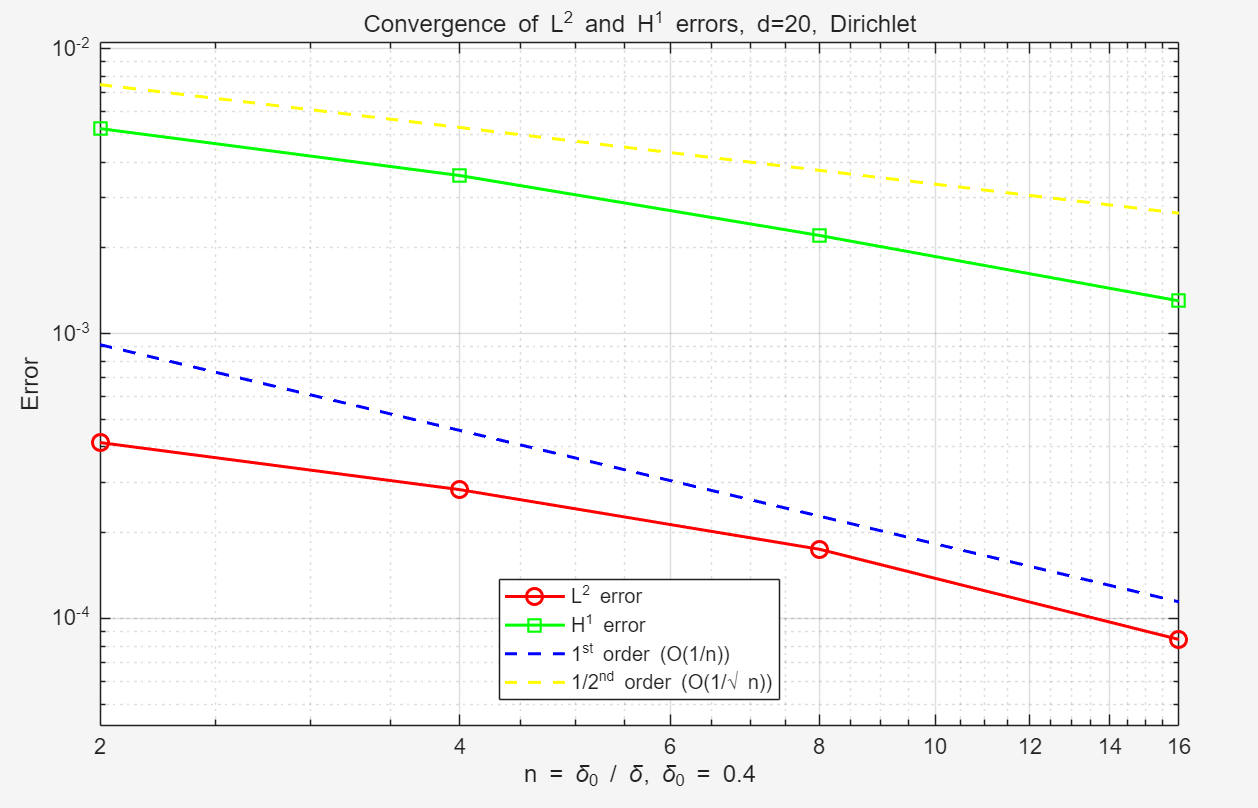}
  \caption{$d=20$}
\end{subfigure}
\caption{$L^2$ and $H^1$ errors between the local solution
         $u_{\mathrm{loc}}$ and the TNN output $u_{\delta,p}$ as
         functions of $\delta$ for the Dirichlet case with
         tensor-product data ($d=3,5,10,20$).}
\label{fig:dirichlet-errors}
\end{figure}

\subsubsection{Neumann case}
\label{subsubsec:neumann-tensor}

We now consider \eqref{eq:neumann-nonlocal} on $\Omega=[0,1]^d$ for
$d=3,5,10,20$ with exact local solution
\[
u_{\mathrm{loc},N}(\xx)
= \cos(\pi x_1)\cdots\cos(\pi x_d) + x_1\cdots x_d .
\]
The corresponding source term and boundary data are
\[
f(\xx)
= -\Delta u_{\mathrm{loc},N} + u_{\mathrm{loc},N}
= (1+d\pi^2)\cos(\pi x_1)\cdots\cos(\pi x_d) + x_1\cdots x_d,
\quad \xx\in\Omega,
\]
and
\[
g(\xx)
= \frac{\partial u_{\mathrm{loc},N}}{\partial\mathbf{n}}
=
\begin{cases}
-\,\displaystyle\prod_{j\ne i}x_j,
& \xx\in F_{i,0}:=\{\xx\in\partial\Omega : x_i=0\},\\[2mm]
\displaystyle\prod_{j\ne i}x_j,
& \xx\in F_{i,1}:=\{\xx\in\partial\Omega : x_i=1\},
\end{cases}
\qquad i=1,\ldots,d.
\]
As in the Dirichlet case, we define the residual of
\eqref{eq:neumann-nonlocal} by
\begin{align}
r_{N}(\xx;\,u_{\delta,p,N})
&:=
\frac{1}{\delta^{2}}
\int_\Omega R_\delta(\xx,\yy)
\bigl(u_{\delta,p,N}(\xx)-u_{\delta,p,N}(\yy)\bigr)\,\d\yy
+
\int_\Omega
\bar{R}_\delta(\xx,\yy)\,u_{\delta,p,N}(\yy)\,\d\yy
\notag\\
&\quad
-\int_\Omega \bar{R}_\delta(\xx,\yy)f(\yy)\,\d\yy
\notag\\
&\quad
-2\int_{\partial\Omega}
\bar{R}_\delta(\xx,\yy)g(\yy)\,\d S_\yy ,
\label{eq:neumann-residual}
\end{align}
and the corresponding mean residual
\[
r_{\mathrm{mean},N}
:= \frac{1}{N}\sum_{i=1}^N \bigl|r_{N}(\xx_i;\,u_{\delta,p,N})\bigr|
\]
with $N=1000$ uniform samples in $\Omega$. The hyperparameters are taken
to be the same as in the Dirichlet experiments
(Table~\ref{tab:numerical-parameters}). The numerical results are
displayed in Figures~\ref{fig:neumann-residual}
and~\ref{fig:neumann-errors}. Both
$\|u_{\mathrm{loc},N}-u_{\delta,p,N}\|_{L^2(\Omega)}$ and
$\|u_{\mathrm{loc},N}-u_{\delta,p,N}\|_{H^1(\Omega)}$ exhibit a numerical
convergence rate of order nearly $1$ in $\delta$, which is again better
than the prediction \eqref{eq:neumann-H1-estimate-2}.

\begin{figure}[htbp]
\centering
\begin{subfigure}[b]{0.48\textwidth}
  \centering
  \includegraphics[width=\textwidth]{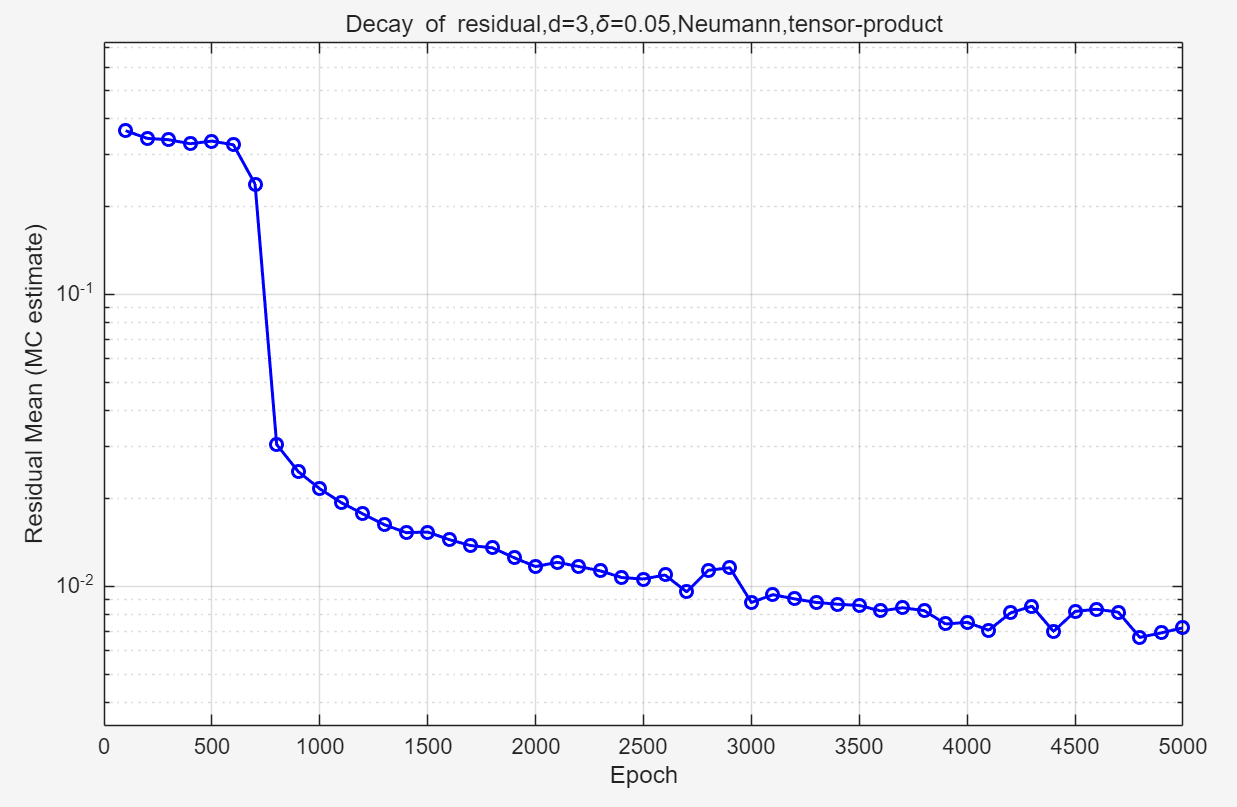}
  \caption{$d=3$}
\end{subfigure}
\hfill
\begin{subfigure}[b]{0.48\textwidth}
  \centering
  \includegraphics[width=\textwidth]{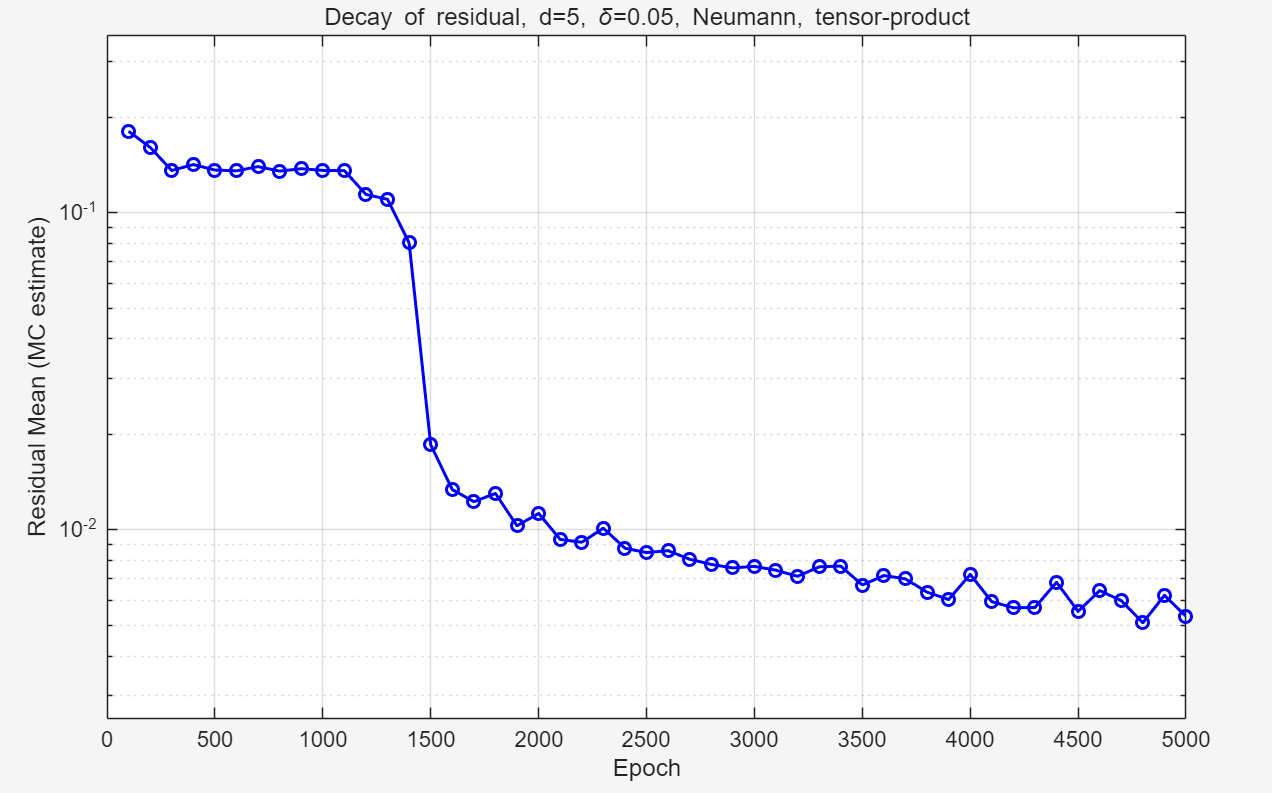}
  \caption{$d=5$}
\end{subfigure}

\vspace{0.5em}

\begin{subfigure}[b]{0.48\textwidth}
  \centering
  \includegraphics[width=\textwidth]{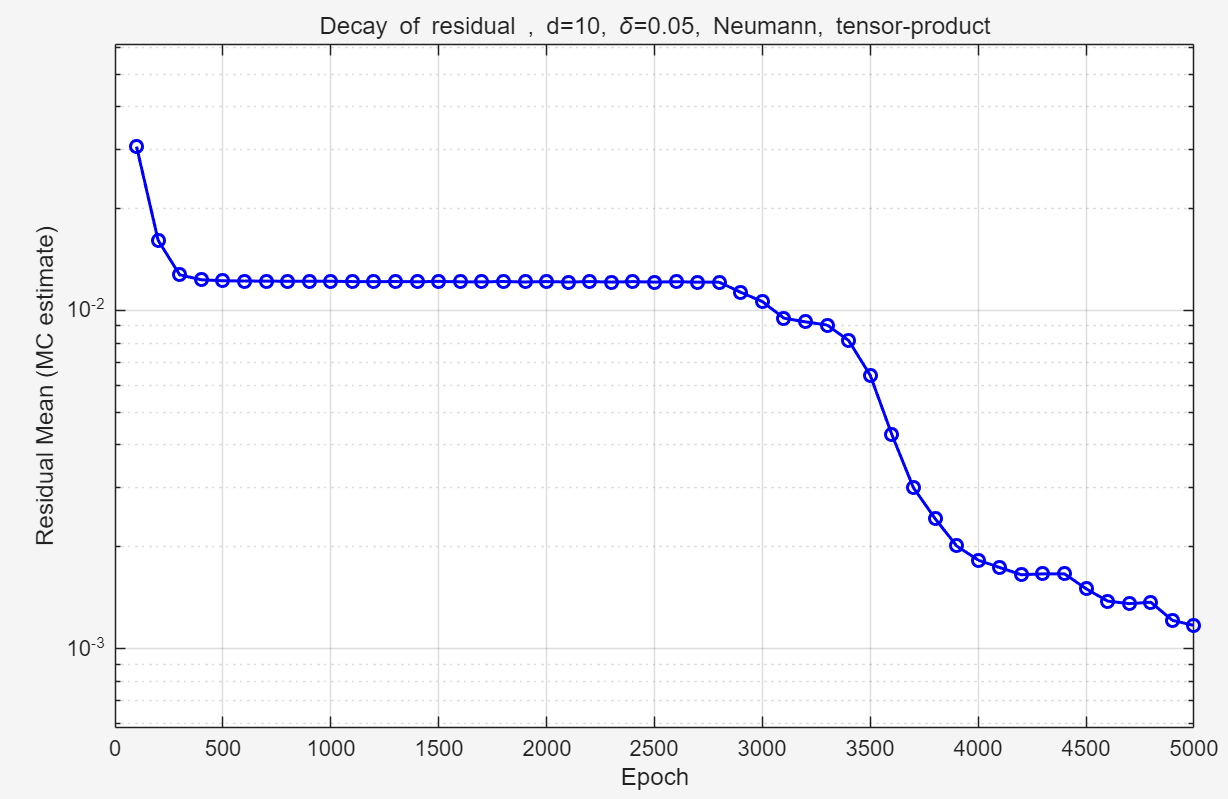}
  \caption{$d=10$}
\end{subfigure}
\hfill
\begin{subfigure}[b]{0.48\textwidth}
  \centering
  \includegraphics[width=\textwidth]{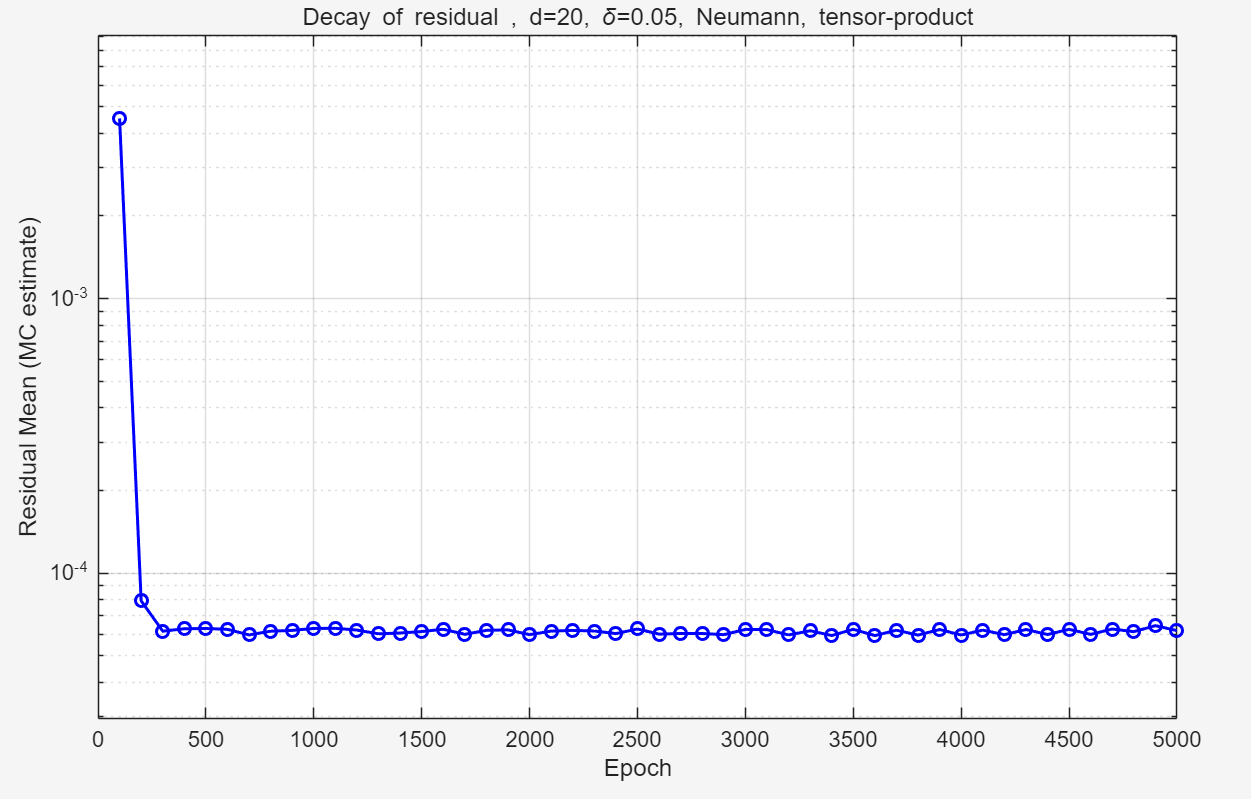}
  \caption{$d=20$}
\end{subfigure}
\caption{Mean pointwise residual $r_{\mathrm{mean},N}$ versus training
         iteration for the Neumann case with tensor-product data
         ($\delta=0.05$, $d=3,5,10,20$).}
\label{fig:neumann-residual}
\end{figure}

\begin{figure}[htbp]
\centering
\begin{subfigure}[b]{0.48\textwidth}
  \centering
  \includegraphics[width=\textwidth]{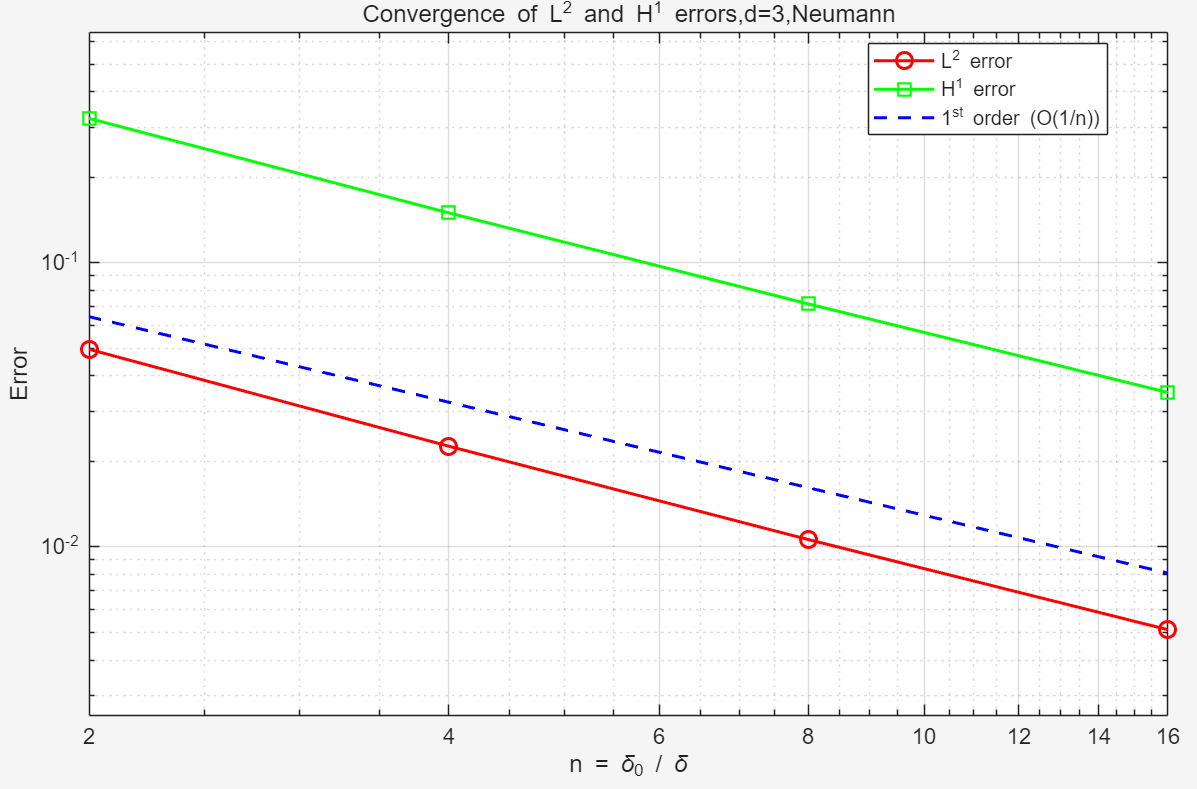}
  \caption{$d=3$}
\end{subfigure}
\hfill
\begin{subfigure}[b]{0.48\textwidth}
  \centering
  \includegraphics[width=\textwidth]{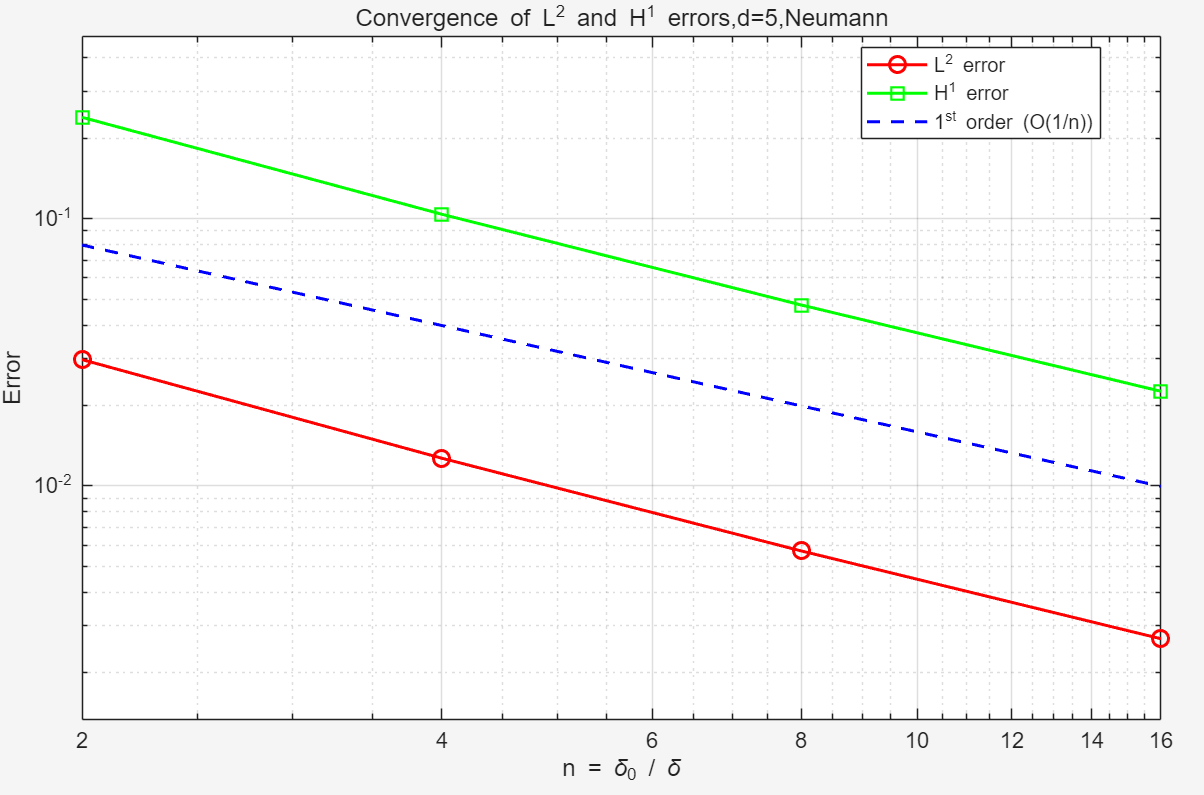}
  \caption{$d=5$}
\end{subfigure}

\vspace{0.5em}

\begin{subfigure}[b]{0.48\textwidth}
  \centering
  \includegraphics[width=\textwidth]{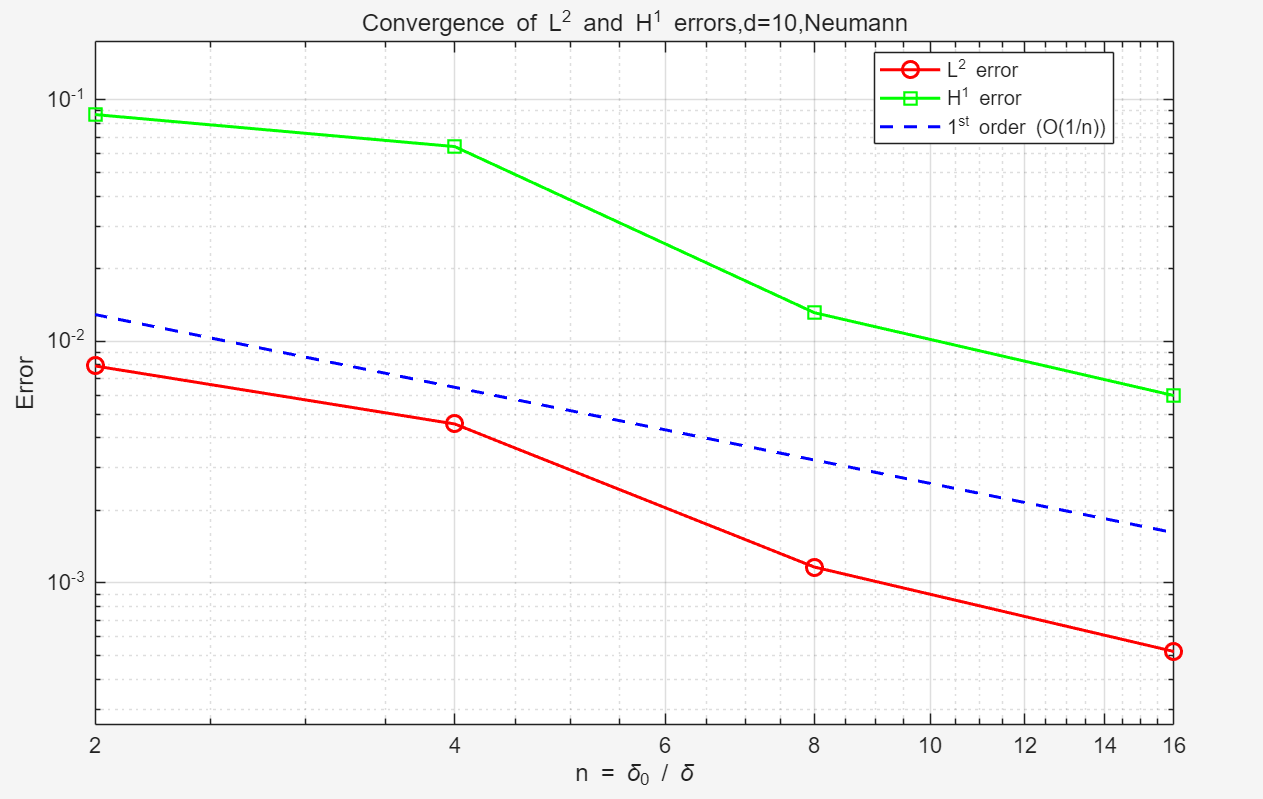}
  \caption{$d=10$}
\end{subfigure}
\hfill
\begin{subfigure}[b]{0.48\textwidth}
  \centering
  \includegraphics[width=\textwidth]{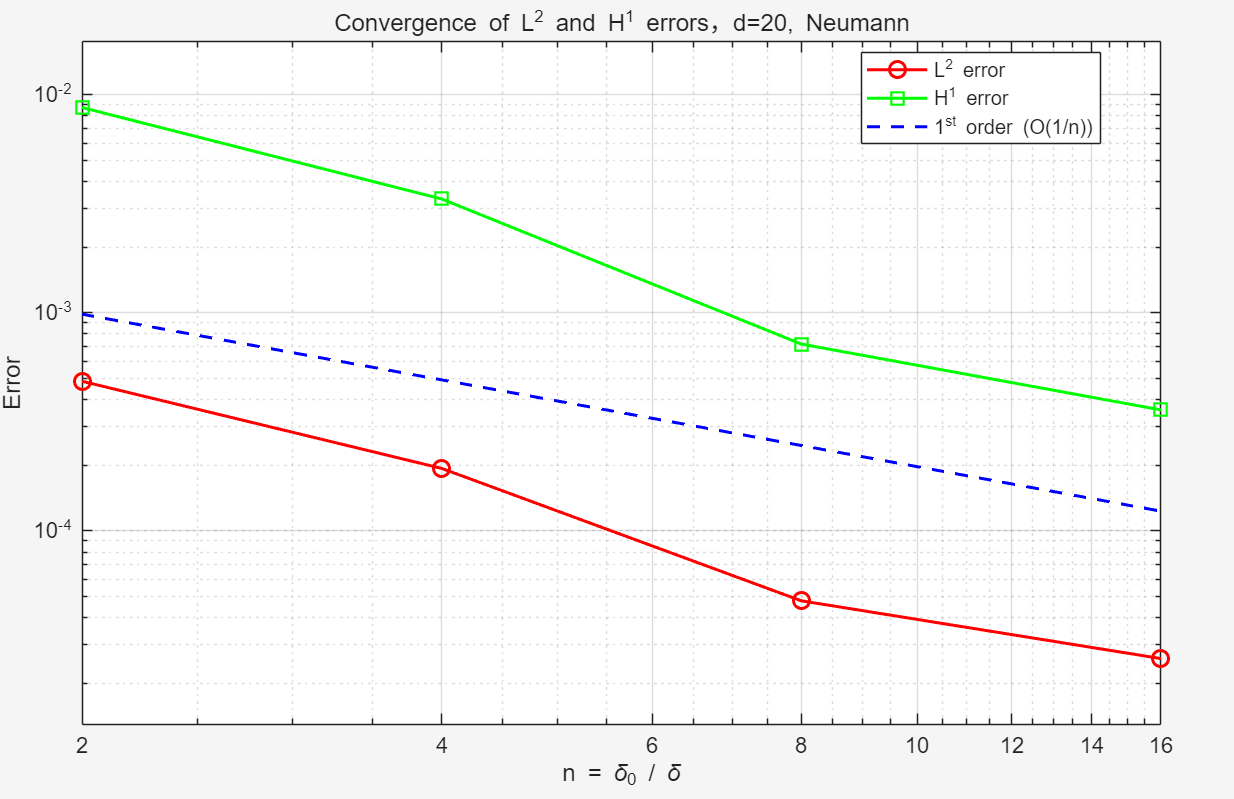}
  \caption{$d=20$}
\end{subfigure}
\caption{$L^2$ and $H^1$ errors between the local solution
         $u_{\mathrm{loc},N}$ and the TNN output $u_{\delta,p,N}$ as
         functions of $\delta$ for the Neumann case with tensor-product
         data ($d=3,5,10,20$).}
\label{fig:neumann-errors}
\end{figure}

\subsection{Experiments with non-tensor-product data}
\label{subsec:nontensor-experiments}

We now consider problems whose exact solution is not of tensor-product
form, in which case all of $\varepsilon_f$, $\varepsilon_g$ and
$\varepsilon_u$ are nonzero in
Theorems~\ref{thm:dirichlet-error}--\ref{thm:neumann-error-h1}. We work on
$\Omega=[0,1]^d$ with the exact solution
\[
u_{\mathrm{loc}}(\xx) = u_{\mathrm{loc},N}(\xx)
= \exp\!\left(\frac{1}{d}\sum_{i=1}^d x_i x_{i+1}\right),
\qquad x_{d+1}:=x_1.
\]
For convenience, we write
$S(\xx):=\frac{1}{d}\sum_{i=1}^d x_i x_{i+1}$, so that
$u_{\mathrm{loc}}=u_{\mathrm{loc},N}=e^{S}$. We consider the Neumann
problem first, in which the source term is given by
\[
f_N(\xx)
= -\Delta u_{\mathrm{loc},N}(\xx) + u_{\mathrm{loc},N}(\xx)
= u_{\mathrm{loc},N}\bigl(1-|\nabla S|^{2}\bigr) ,
\]
where the last equality uses the identity
$\Delta(e^{S}) = e^{S}\bigl(|\nabla S|^{2}+\Delta S\bigr)$ together with
$\Delta S=0$, since each summand $x_i x_{i+1}$ is harmonic. As described
in Section~\ref{sec:main-result}, we first approximate $f_N$ by a TNN
surrogate $\tilde f_N$. We begin by examining the effect of the
hyperparameters on the quality of this approximation.

\subsubsection{Effects of the hyperparameters}
\label{subsubsec:hyperparam}

For two functions $f_N$ and $\tilde f_N$, we randomly sample
$N_{\mathrm{approx}}=20000$ points
$\xx_1,\ldots,\xx_{N_{\mathrm{approx}}}$ in $\Omega$ and use the
following two indicators to measure the approximation quality:
\begin{align*}
\mathrm{RMSE}
&:= \sqrt{\frac{1}{N_{\mathrm{approx}}}
\sum_{j=1}^{N_{\mathrm{approx}}}
\bigl(\tilde f_N(\xx_j)-f_N(\xx_j)\bigr)^{2}},\\
\mathrm{rel.\,RMSE}
&:= \frac{\mathrm{RMSE}}
{\sqrt{\frac{1}{N_{\mathrm{approx}}}
\sum_{j=1}^{N_{\mathrm{approx}}}f_N(\xx_j)^{2}}}.
\end{align*}
We investigate the effects of the four hyperparameters: the separation
rank $p$, the dimension $d$, the depth (\#layers) and the width
(\#neurons) of each one-dimensional subnetwork. The parameter settings of
the four scans are summarized in Table~\ref{tab:e5-hyperparam-settings}.
In each scan only the corresponding hyperparameter is varied, while the
others are fixed.

\begin{table}[htbp]
    \centering
    \caption{Parameter settings for the four hyperparameter scans.}
    \label{tab:e5-hyperparam-settings}
    \small
    \setlength{\tabcolsep}{4pt}
    \begin{tabular}{|c|c|c|c|c|}
        \hline
        scan parameter & rank $p$ & dimension $d$ & \#layers & \#neurons \\
        \hline
        $p$ & $10,20,40,60,80,120$ & $5$ & $2$ & $100$ \\
        \hline
        $d$ & $60$ & $3,5,8,10,15,20$ & $2$ & $100$ \\
        \hline
        \#neurons & $60$ & $5$ & $2$ & $30,60,100,150,200$ \\
        \hline
        \#layers & $60$ & $5$ & $1,2,3,4,5$ & $100$ \\
        \hline
    \end{tabular}
\end{table}

\begin{table}[htbp]
    \centering
    \caption{Scan over the rank $p$, with target $f_N=u_{\mathrm{loc},N}(1-|\nabla S|^2)$.}
    \label{tab:rank-scan}
    \begin{tabular}{|r|r|c|c|c|}
        \hline
        rank $p$ & $n_{\text{params}}$ & RMSE & rel.\ RMSE & time (s) \\
        \hline
         10 &  56{,}551 & $1.919\times 10^{-3}$ & $1.966\times 10^{-3}$ &  78.7 \\
        \hline
         20 &  61{,}601 & $1.519\times 10^{-3}$ & $1.556\times 10^{-3}$ &  97.2 \\
        \hline
         40 &  71{,}701 & $1.489\times 10^{-3}$ & $1.526\times 10^{-3}$ &  92.3 \\
        \hline
         60 &  81{,}801 & $1.449\times 10^{-3}$ & $1.484\times 10^{-3}$ &  95.0 \\
        \hline
         80 &  91{,}901 & $1.578\times 10^{-3}$ & $1.616\times 10^{-3}$ & 103.2 \\
        \hline
        120 & 112{,}101 & $1.769\times 10^{-3}$ & $1.812\times 10^{-3}$ & 121.3 \\
        \hline
    \end{tabular}
\end{table}

\begin{table}[htbp]
    \centering
    \caption{Scan over the dimension $d$, with target $f_N=u_{\mathrm{loc},N}(1-|\nabla S|^2)$.}
    \label{tab:dimension-scan}
    \begin{tabular}{|r|r|c|c|c|}
        \hline
        dimension $d$ & $n_{\text{params}}$ & RMSE & rel.\ RMSE & time (s) \\
        \hline
         3 &  49{,}081 & $4.019\times 10^{-3}$ & $5.218\times 10^{-3}$ &  58.1 \\
        \hline
         5 &  81{,}801 & $1.457\times 10^{-3}$ & $1.493\times 10^{-3}$ &  96.8 \\
        \hline
         8 & 130{,}881 & $1.212\times 10^{-3}$ & $1.104\times 10^{-3}$ & 151.8 \\
        \hline
        10 & 163{,}601 & $1.541\times 10^{-3}$ & $1.355\times 10^{-3}$ & 188.7 \\
        \hline
        15 & 245{,}401 & $1.450\times 10^{-3}$ & $1.221\times 10^{-3}$ & 280.5 \\
        \hline
        20 & 327{,}201 & $1.881\times 10^{-3}$ & $1.551\times 10^{-3}$ & 373.7 \\
        \hline
    \end{tabular}
\end{table}

\begin{table}[htbp]
    \centering
    \caption{Scan over the number of hidden neurons, with target $f_N=u_{\mathrm{loc},N}(1-|\nabla S|^2)$.}
    \label{tab:hidden-neurons-scan}
    \begin{tabular}{|r|r|c|c|c|}
        \hline
        hidden neurons & $n_{\text{params}}$ & RMSE & rel.\ RMSE & time (s) \\
        \hline
         30 &  14{,}251 & $1.831\times 10^{-3}$ & $1.876\times 10^{-3}$ &  35.0 \\
        \hline
         60 &  37{,}201 & $1.566\times 10^{-3}$ & $1.604\times 10^{-3}$ &  46.8 \\
        \hline
        100 &  81{,}801 & $1.545\times 10^{-3}$ & $1.582\times 10^{-3}$ &  94.4 \\
        \hline
        150 & 160{,}051 & $1.398\times 10^{-3}$ & $1.432\times 10^{-3}$ & 185.6 \\
        \hline
        200 & 263{,}301 & $1.429\times 10^{-3}$ & $1.464\times 10^{-3}$ & 237.4 \\
        \hline
    \end{tabular}
\end{table}

\begin{table}[htbp]
    \centering
    \caption{Scan over the number of hidden layers, with target $f_N=u_{\mathrm{loc},N}(1-|\nabla S|^2)$.}
    \label{tab:hidden-layers-scan}
    \begin{tabular}{|r|r|c|c|c|}
        \hline
        hidden layers & $n_{\text{params}}$ & RMSE & rel.\ RMSE & time (s) \\
        \hline
         1 &  31{,}301 & $3.365\times 10^{-3}$ & $3.447\times 10^{-3}$ &  42.0 \\
        \hline
         2 &  81{,}801 & $1.419\times 10^{-3}$ & $1.454\times 10^{-3}$ &  97.3 \\
        \hline
         3 & 132{,}301 & $1.499\times 10^{-3}$ & $1.536\times 10^{-3}$ & 151.9 \\
        \hline
         4 & 182{,}801 & $1.542\times 10^{-3}$ & $1.579\times 10^{-3}$ & 207.4 \\
        \hline
         5 & 233{,}301 & $1.611\times 10^{-3}$ & $1.650\times 10^{-3}$ & 262.6 \\
        \hline
    \end{tabular}
\end{table}

Tables~\ref{tab:rank-scan}--\ref{tab:hidden-layers-scan} show that a
larger model size does not always lead to better approximation. In the
rank scan, the RMSE decreases from $1.919\times10^{-3}$ at $p=10$ to
$1.449\times10^{-3}$ at $p=60$, but increases slightly when $p$ is
further enlarged. In the width scan, the smallest error is obtained at
\#neurons $=150$, with RMSE $1.398\times10^{-3}$, while using $200$
neurons gives no further improvement. In the depth scan, two hidden
layers achieve the smallest error; deeper networks increase both the
number of parameters and the training time without improving the
accuracy. For the dimension scan, only $d$ is varied, while $p=60$,
\#layers $=2$ and \#neurons $=100$ are kept fixed. The largest relative
RMSE appears at $d=3$ ($5.218\times10^{-3}$), and the error remains of
order $10^{-3}$ for $d\ge 5$. The best relative RMSE is obtained at $d=8$,
$1.104\times10^{-3}$. As $d$ increases, the parameter count and training
time grow steadily, from $49{,}081$ parameters and $58.1$ seconds at
$d=3$ to $327{,}201$ parameters and $373.7$ seconds at $d=20$. Thus,
under the fixed baseline architecture, the TNN approximation remains
stable in higher dimensions, at increasing computational cost.

\subsubsection{Neumann case}
\label{subsubsec:neumann-nontensor}

Having examined the effects of $p$, $d$, \#layers, and \#neurons, we now
turn to the Neumann problem with the non-tensor-product data introduced
above. The hyperparameters used in the experiments are summarized in
Table~\ref{tab:numerical-parameters-nontensor}.

\begin{table}[htbp]
\centering
\caption{Network, quadrature, and approximation parameters for the
         non-tensor-product Neumann experiments.}
\label{tab:numerical-parameters-nontensor}
\small
\setlength{\tabcolsep}{3pt}
\begin{tabular}{lccccccc}
\hline
Experiment & $p_u$ & \#layers & \#neurons
           & $n_{\mathrm{sub}}$ & $n_{\mathrm{pts}}$
           & $p_f$ & $p_g$ \\
\hline
$d=3,\;\delta=0.200$ & 30 & 2 & 80 &  40 & 20 & 30 & 20 \\
$d=3,\;\delta=0.100$ & 30 & 2 & 80 &  80 & 20 & 30 & 20 \\
$d=3,\;\delta=0.050$ & 30 & 2 & 80 & 160 & 20 & 30 & 20 \\
$d=3,\;\delta=0.025$ & 40 & 2 & 80 & 320 & 20 & 30 & 20 \\
$d=5,\;\delta=0.200$ & 50 & 2 & 100 &  40 & 20 & 50 & 30 \\
$d=5,\;\delta=0.100$ & 50 & 2 & 100 &  80 & 20 & 50 & 30 \\
$d=5,\;\delta=0.050$ & 50 & 2 & 100 & 160 & 20 & 50 & 30 \\
$d=5,\;\delta=0.025$ & 50 & 2 & 100 & 320 & 20 & 50 & 30 \\
$d=10,20,\;\delta=0.200$ & 80 & 3 & 120 &  40 & 20 & 80 & 60 \\
$d=10,20,\;\delta=0.100$ & 80 & 3 & 120 &  80 & 20 & 80 & 60 \\
$d=10,20,\;\delta=0.050$ & 80 & 3 & 120 & 160 & 20 & 80 & 60 \\
$d=10,20,\;\delta=0.025$ & 80 & 3 & 120 & 320 & 20 & 80 & 60 \\
\hline
\end{tabular}
\end{table}

Here $p_u$, $p_f$ and $p_g$ are the TNN separation ranks of the
approximations of $u$, $f$ and $g$, respectively. The columns \#layers
and \#neurons denote the depth and width of each one-dimensional
subnetwork used in the approximation of $f$, $g$ and $u$. The quantities
$n_{\mathrm{sub}}$ and $n_{\mathrm{pts}}$ have the same meaning as in
Section~\ref{subsec:tensor-product-experiments}. Before turning to the
main numerical tests, we first examine the effect of the quadrature
parameters $n_{\mathrm{sub}}$ and $n_{\mathrm{pts}}$ on the experiment
$d=3$, $\delta=0.05$.

\begin{table}[htbp]
\centering
\caption{Effect of $n_{\mathrm{pts}}$ for $d=3$, $\delta=0.05$.}
\label{tab:e6-integral-points}
\begin{tabular}{cccccc}
\hline
$n_{\mathrm{sub}}$ & $n_{\mathrm{pts}}$
& resid.\ mean & $L^2$ final & $H^1$ final & time (s) \\
\hline
160 &  8 & $5.260\times 10^{-2}$ & $6.127\times 10^{-3}$ & $1.335\times 10^{-2}$ & 171 \\
160 & 12 & $5.255\times 10^{-2}$ & $6.023\times 10^{-3}$ & $1.324\times 10^{-2}$ & 183 \\
160 & 16 & $5.260\times 10^{-2}$ & $6.130\times 10^{-3}$ & $1.343\times 10^{-2}$ & 214 \\
160 & 20 & $5.250\times 10^{-2}$ & $4.717\times 10^{-3}$ & $1.320\times 10^{-2}$ & 372 \\
160 & 24 & $5.259\times 10^{-2}$ & $5.519\times 10^{-3}$ & $1.302\times 10^{-2}$ & 456 \\
160 & 32 & $5.253\times 10^{-2}$ & $6.127\times 10^{-3}$ & $1.297\times 10^{-2}$ & 623 \\
\hline
\end{tabular}
\end{table}

\begin{table}[htbp]
\centering
\caption{Effect of $n_{\mathrm{sub}}$ for $d=3$, $\delta=0.05$.}
\label{tab:e6-integral-subintervals}
\begin{tabular}{cccccc}
\hline
$n_{\mathrm{sub}}$ & $n_{\mathrm{pts}}$
& resid.\ mean & $L^2$ final & $H^1$ final & time (s) \\
\hline
 40 & 20 & $5.254\times 10^{-2}$ & $6.127\times 10^{-3}$ & $1.272\times 10^{-2}$ & 170 \\
 80 & 20 & $5.252\times 10^{-2}$ & $5.455\times 10^{-3}$ & $1.321\times 10^{-2}$ & 171 \\
120 & 20 & $5.250\times 10^{-2}$ & $5.338\times 10^{-3}$ & $1.298\times 10^{-2}$ & 203 \\
160 & 20 & $5.250\times 10^{-2}$ & $4.717\times 10^{-3}$ & $1.320\times 10^{-2}$ & 372 \\
240 & 20 & $5.249\times 10^{-2}$ & $4.967\times 10^{-3}$ & $1.270\times 10^{-2}$ & 597 \\
320 & 20 & $5.244\times 10^{-2}$ & $4.165\times 10^{-3}$ & $1.311\times 10^{-2}$ & 1261 \\
\hline
\end{tabular}
\end{table}

The indicator $\mathrm{resid.\,mean}$ is the mean residual defined as in
\eqref{eq:neumann-residual}, while $L^2$\,final and $H^1$\,final denote
$\|u_{\mathrm{loc},N}-u_{\delta,p,N}\|_{L^2(\Omega)}$ and
$\|u_{\mathrm{loc},N}-u_{\delta,p,N}\|_{H^1(\Omega)}$, respectively.
Since $u_{\mathrm{loc},N}$, $f$ and $g$ are no longer of tensor-product
form, these errors are computed by Monte Carlo sampling to avoid the
curse of dimensionality. Tables~\ref{tab:e6-integral-points}
and~\ref{tab:e6-integral-subintervals} show that the residual mean is
rather stable, staying around $5.25\times10^{-2}$ under different
quadrature settings. The main effect of refining the quadrature is
reflected in the final $L^2$ and $H^1$ errors, together with the
computational cost. Increasing $n_{\mathrm{pts}}$ up to $20$ improves the
$L^2$ error, but further increase yields little benefit and significantly
longer training time. Similarly, larger $n_{\mathrm{sub}}$ further
reduces the $L^2$ error, but the cost grows rapidly. Therefore, in the
following Neumann experiments we fix $n_{\mathrm{pts}}=20$ and choose
$n_{\mathrm{sub}}$ according to the horizon size $\delta$, increasing
$n_{\mathrm{sub}}$ as $\delta$ decreases; this gives a practical balance
between accuracy and efficiency.

The numerical results of the non-tensor-product Neumann experiments are
displayed in Figures~\ref{fig:neumann-residual-nontensor}
and~\ref{fig:neumann-errors-nontensor}. Due to the additional
approximation errors $\varepsilon_f$ and $\varepsilon_g$ from the
preconditioning step, the results are slightly worse than those in
Figures~\ref{fig:neumann-residual} and~\ref{fig:neumann-errors}, but they
remain in a reasonable regime, which supports the effectiveness of the
TNN-based method in the non-tensor-product setting.

\begin{figure}[htbp]
\centering
\begin{subfigure}[b]{0.48\textwidth}
  \centering
  \includegraphics[width=\textwidth]{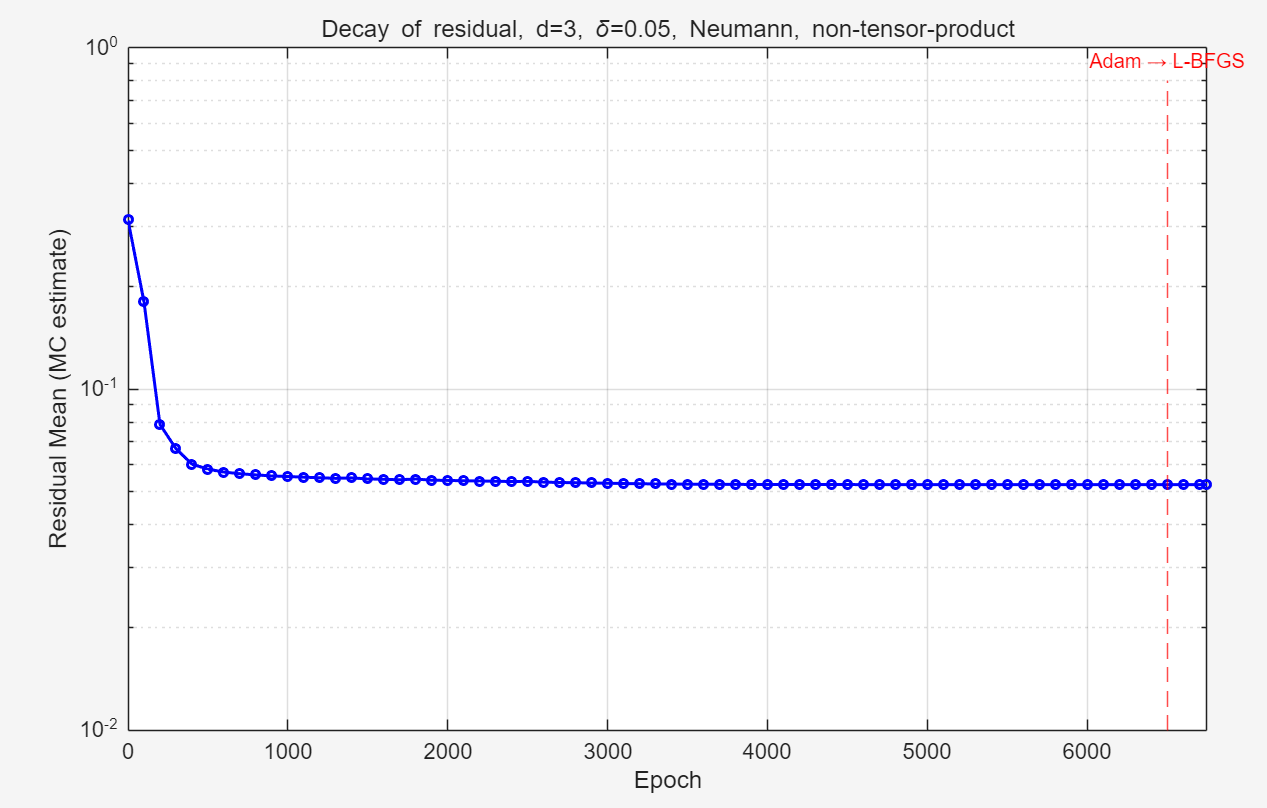}
  \caption{$d=3$}
\end{subfigure}
\hfill
\begin{subfigure}[b]{0.48\textwidth}
  \centering
  \includegraphics[width=\textwidth]{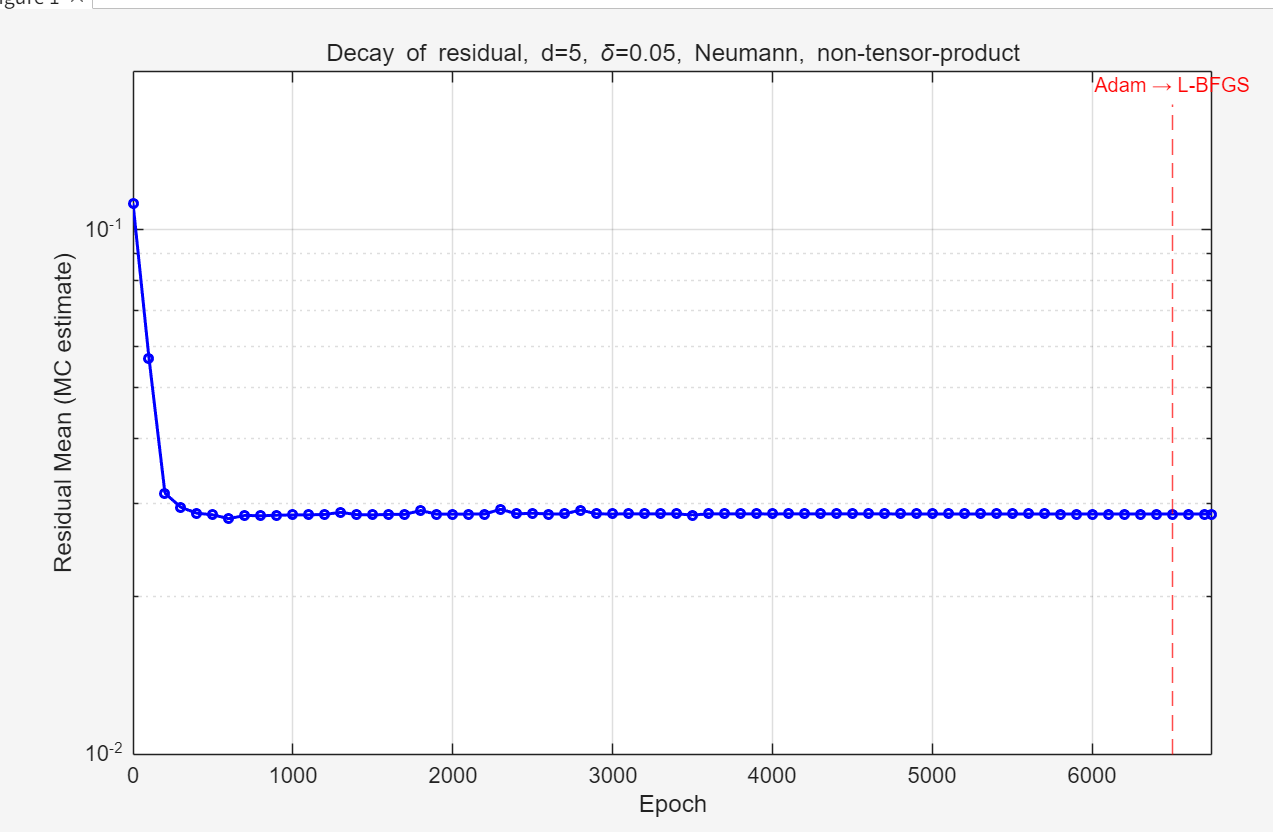}
  \caption{$d=5$}
\end{subfigure}

\vspace{0.5em}

\begin{subfigure}[b]{0.48\textwidth}
  \centering
  \includegraphics[width=\textwidth]{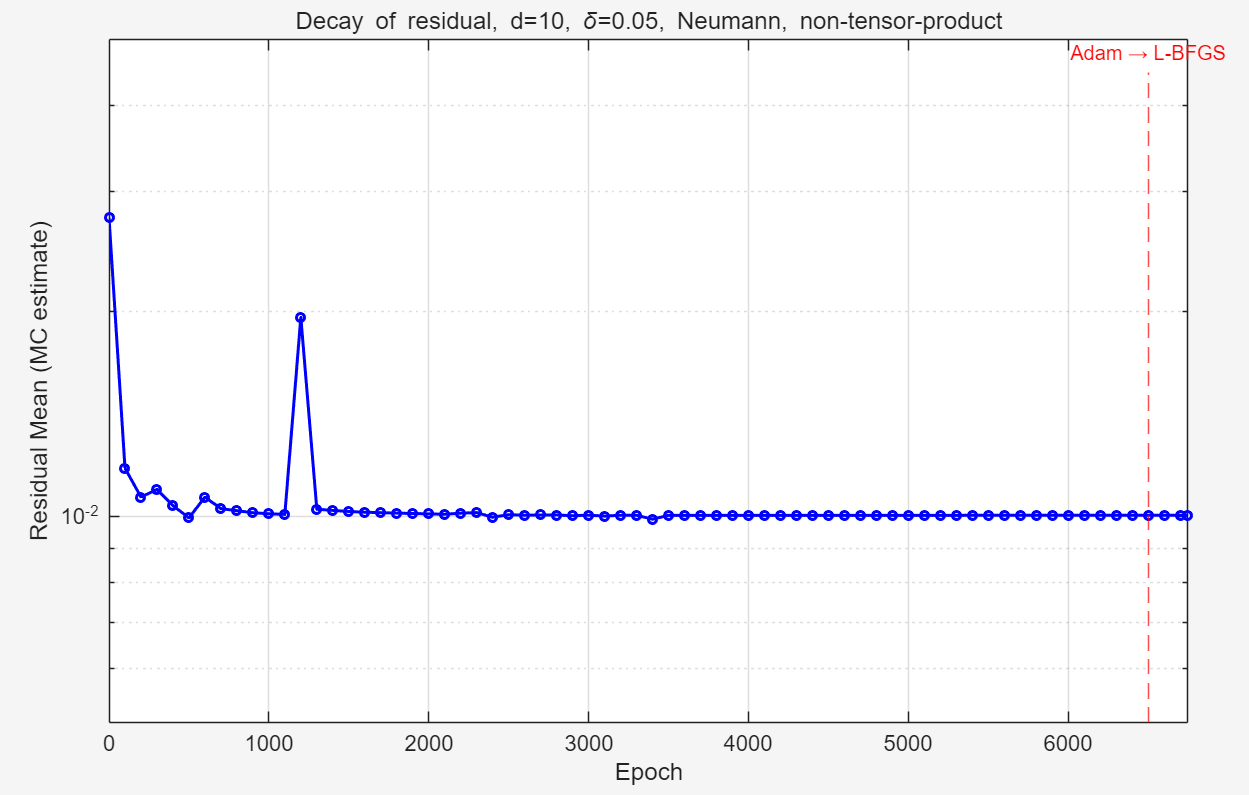}
  \caption{$d=10$}
\end{subfigure}
\hfill
\begin{subfigure}[b]{0.48\textwidth}
  \centering
  \includegraphics[width=\textwidth]{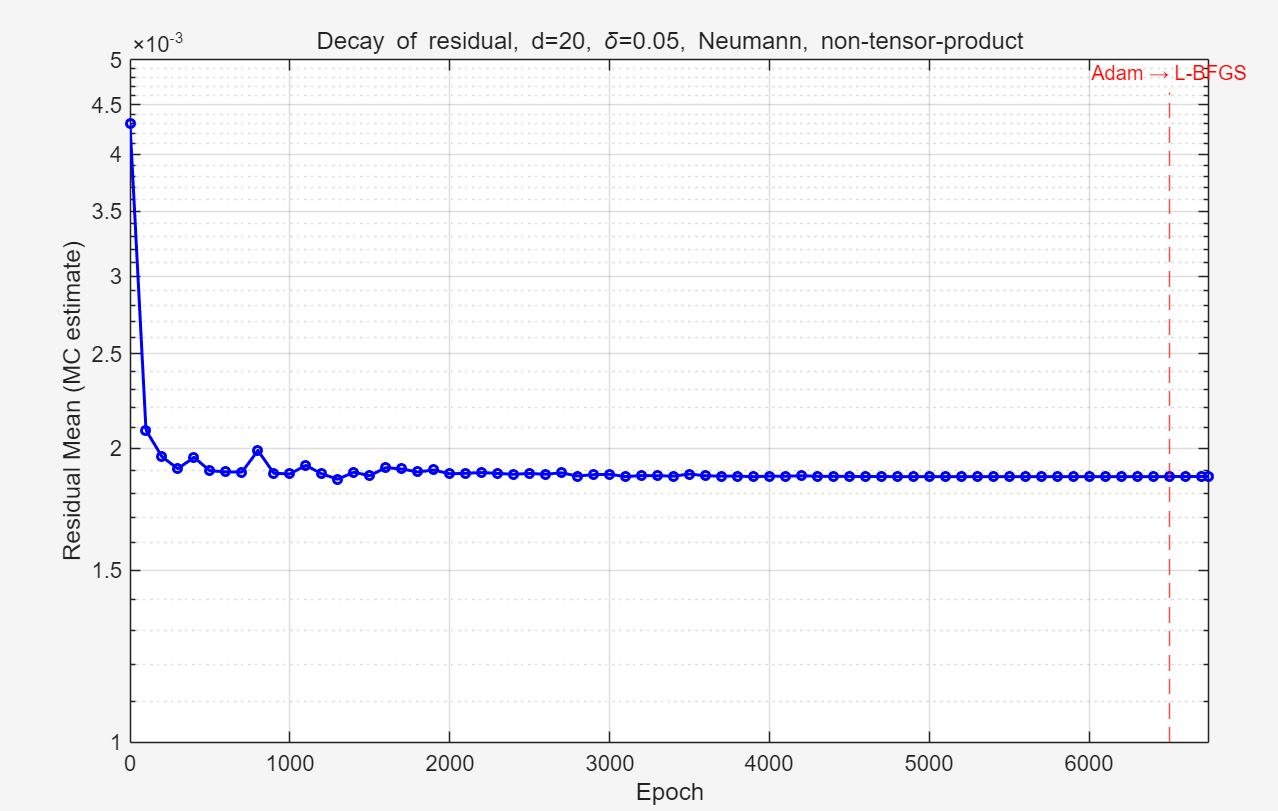}
  \caption{$d=20$}
\end{subfigure}
\caption{Mean pointwise residual $r_{\mathrm{mean},N}$ versus training
         iteration for the Neumann case with non-tensor-product data
         ($d=3,5,10,20$).}
\label{fig:neumann-residual-nontensor}
\end{figure}

\begin{figure}[htbp]
\centering
\begin{subfigure}[b]{0.48\textwidth}
  \centering
  \includegraphics[width=\textwidth]{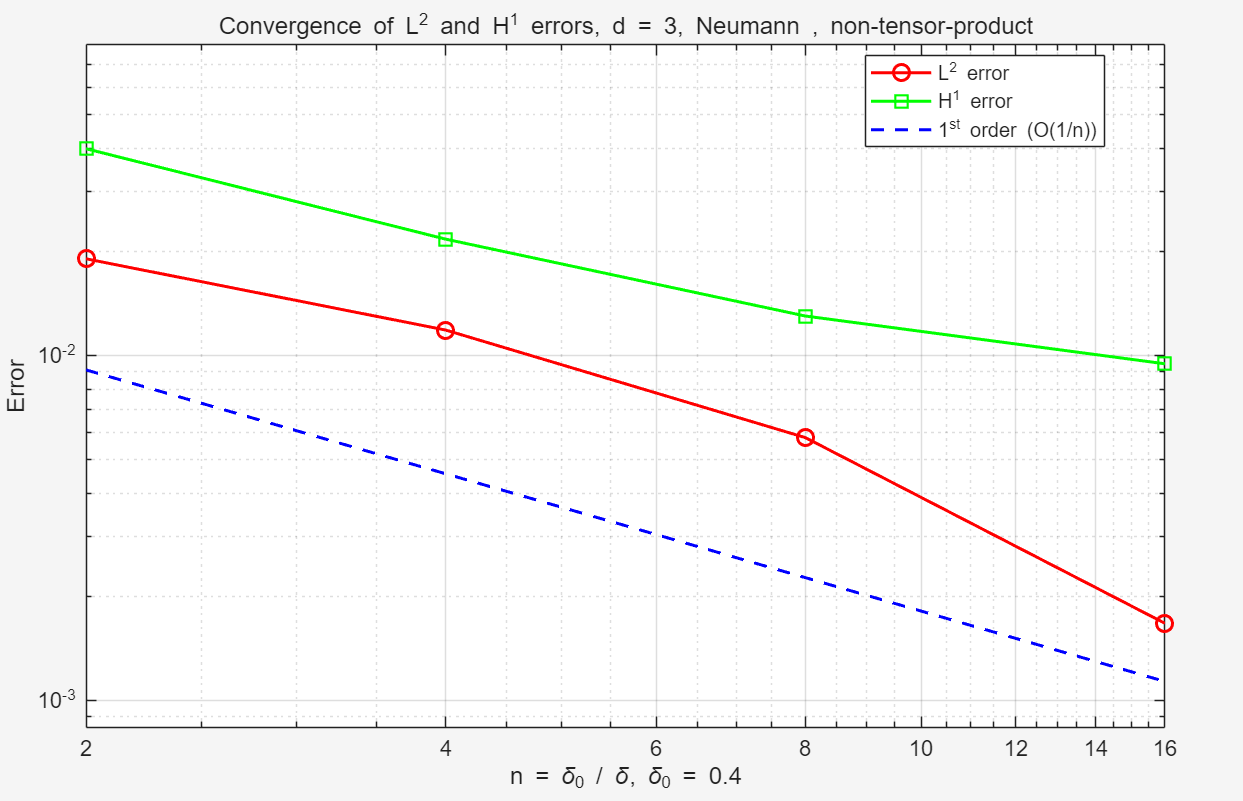}
  \caption{$d=3$}
\end{subfigure}
\hfill
\begin{subfigure}[b]{0.48\textwidth}
  \centering
  \includegraphics[width=\textwidth]{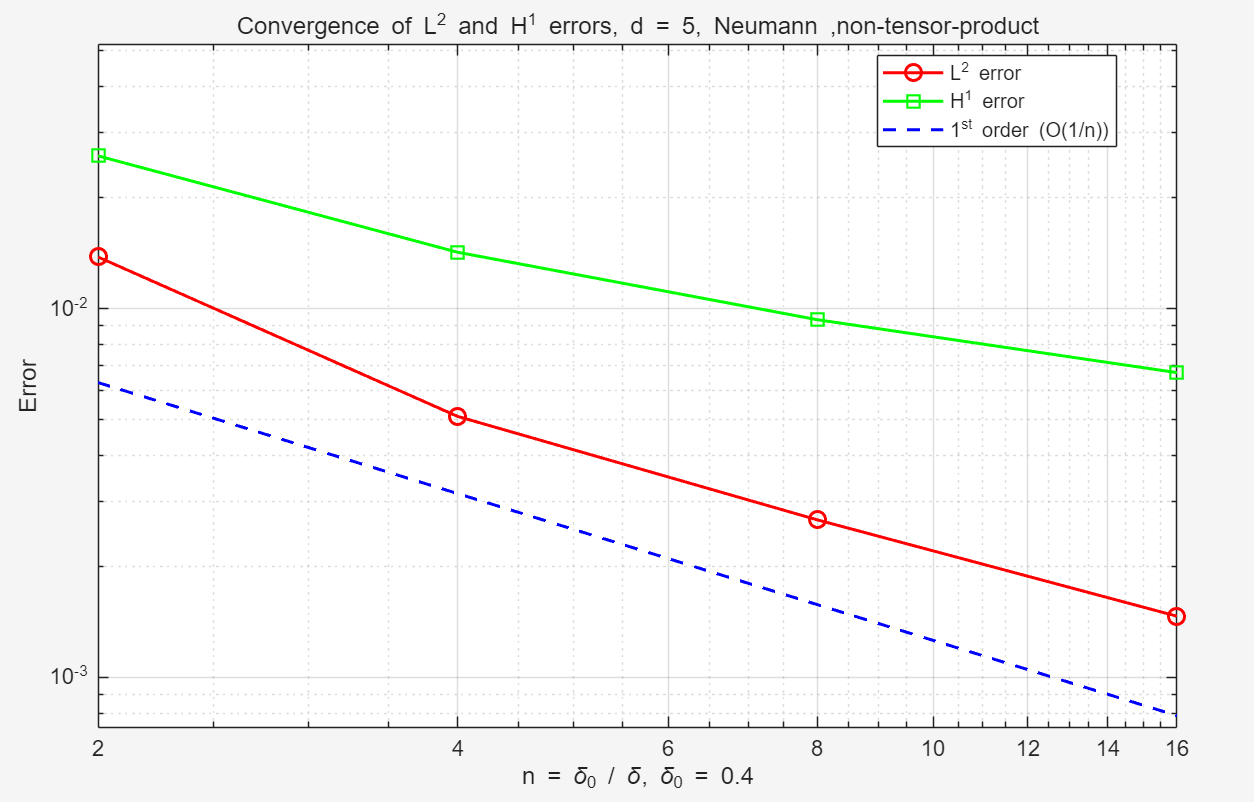}
  \caption{$d=5$}
\end{subfigure}

\vspace{0.5em}

\begin{subfigure}[b]{0.48\textwidth}
  \centering
  \includegraphics[width=\textwidth]{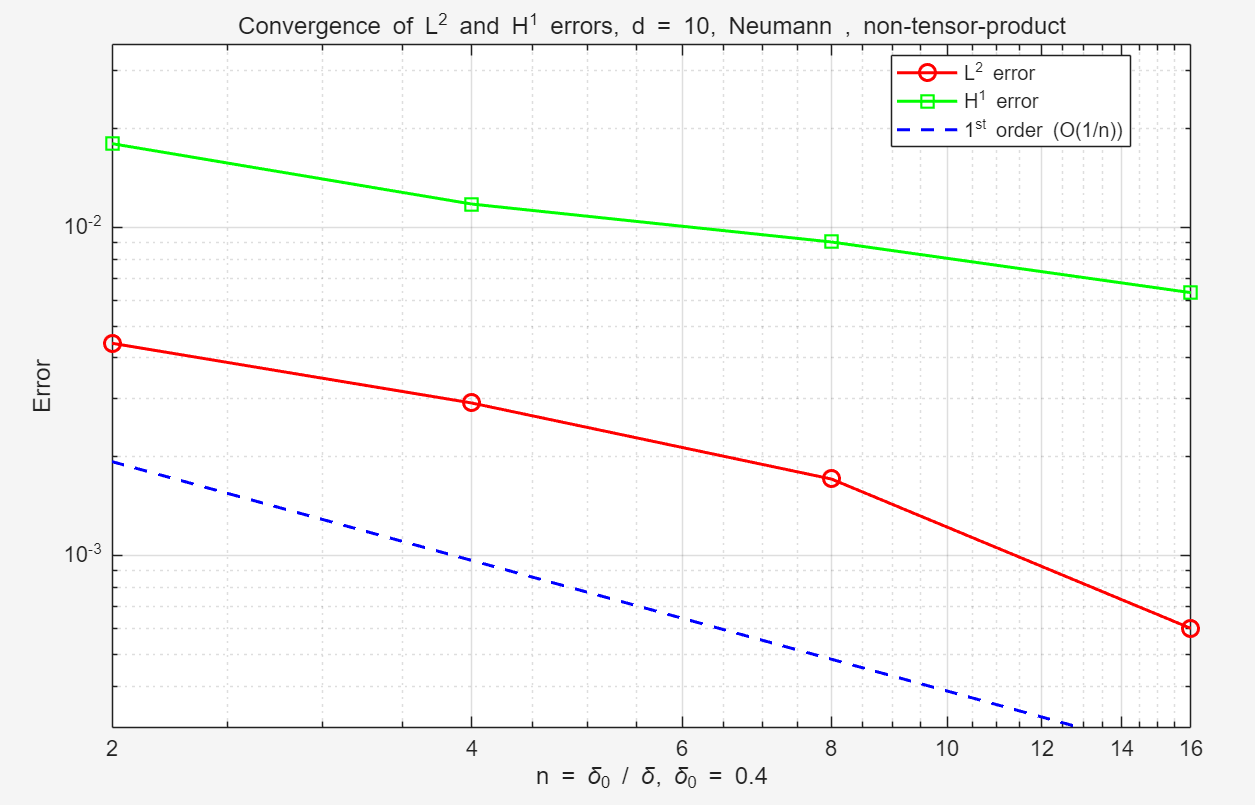}
  \caption{$d=10$}
\end{subfigure}
\hfill
\begin{subfigure}[b]{0.48\textwidth}
  \centering
  \includegraphics[width=\textwidth]{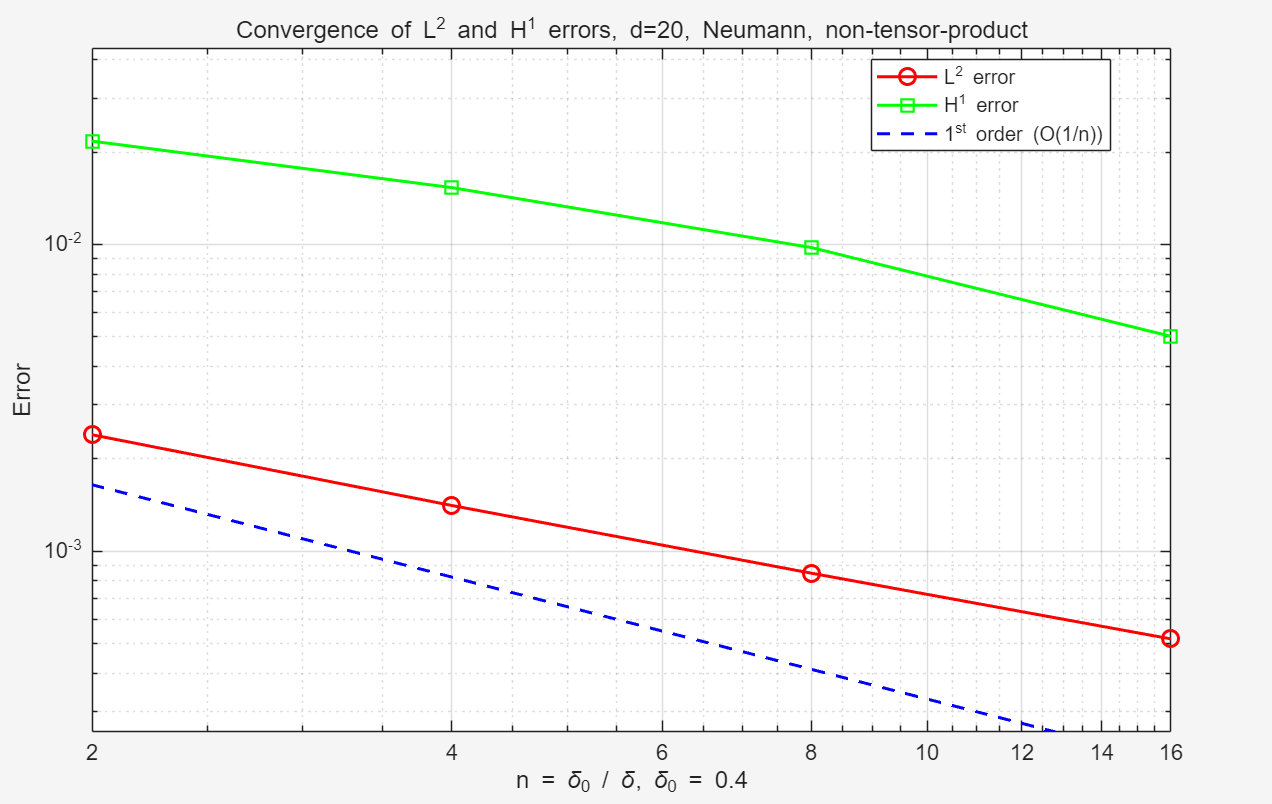}
  \caption{$d=20$}
\end{subfigure}
\caption{$L^2$ and $H^1$ errors between the local solution
         $u_{\mathrm{loc},N}$ and the TNN output $u_{\delta,p,N}$ as
         functions of $\delta$ for the Neumann case with
         non-tensor-product data ($d=3,5,10,20$).}
\label{fig:neumann-errors-nontensor}
\end{figure}

\subsubsection{Dirichlet case}
\label{subsubsec:dirichlet-nontensor}

We also test the same TNN-based method on the Dirichlet problem with the same
exact solution. Here the source term is
\[
f(\xx) = -\Delta u_{\mathrm{loc}}(\xx)
= -|\nabla S|^{2}\,u_{\mathrm{loc}},
\]
where we again used $\Delta S=0$, and the boundary data $g$ is obtained
by restricting $u_{\mathrm{loc}}$ to each face of $\partial\Omega$. The
hyperparameters are taken as in the Neumann case
(Table~\ref{tab:numerical-parameters-nontensor}). The numerical results
are reported in Figures~\ref{fig:dirichlet-residual-nontensor}
and~\ref{fig:dirichlet-errors-nontensor}. Compared with the Neumann case,
the residual exhibits a noticeably larger fluctuation. A possible
explanation is that, in the Dirichlet setting, the theoretical bound in
\eqref{eq:dirichlet-estimate-1} is half an order worse in $\delta$ than
in \eqref{eq:neumann-estimate-1}, owing to the $\varepsilon_g/\delta$
term; moreover, as the dimension increases, the Monte Carlo samples used
to evaluate the residual become more sparsely distributed in $\Omega$,
which leads to higher variance. Despite the fluctuation in early iterations,
the residual and the $L^2$ and $H^1$ errors eventually decrease to
reasonable levels, which again supports the effectiveness of the method.

\begin{figure}[htbp]
\centering
\begin{subfigure}[b]{0.48\textwidth}
  \centering
  \includegraphics[width=\textwidth]{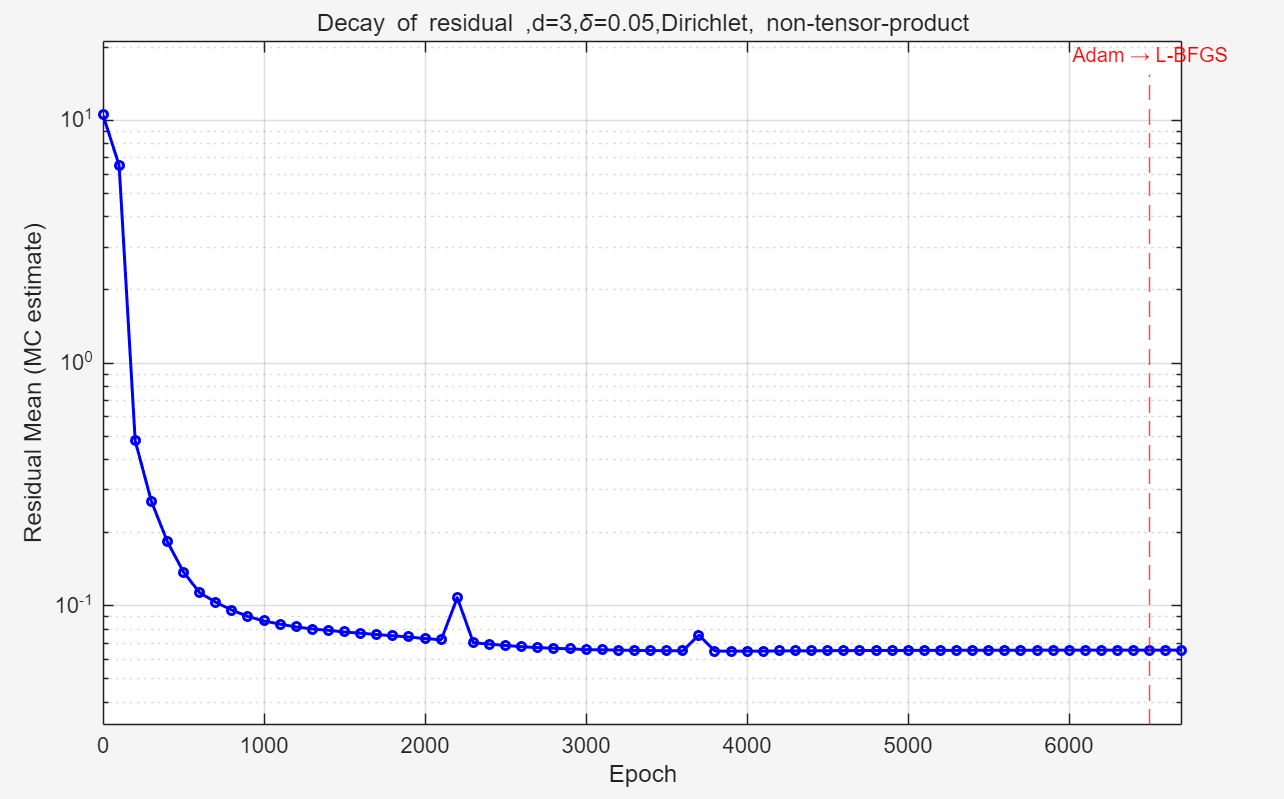}
  \caption{$d=3$}
\end{subfigure}
\hfill
\begin{subfigure}[b]{0.48\textwidth}
  \centering
  \includegraphics[width=\textwidth]{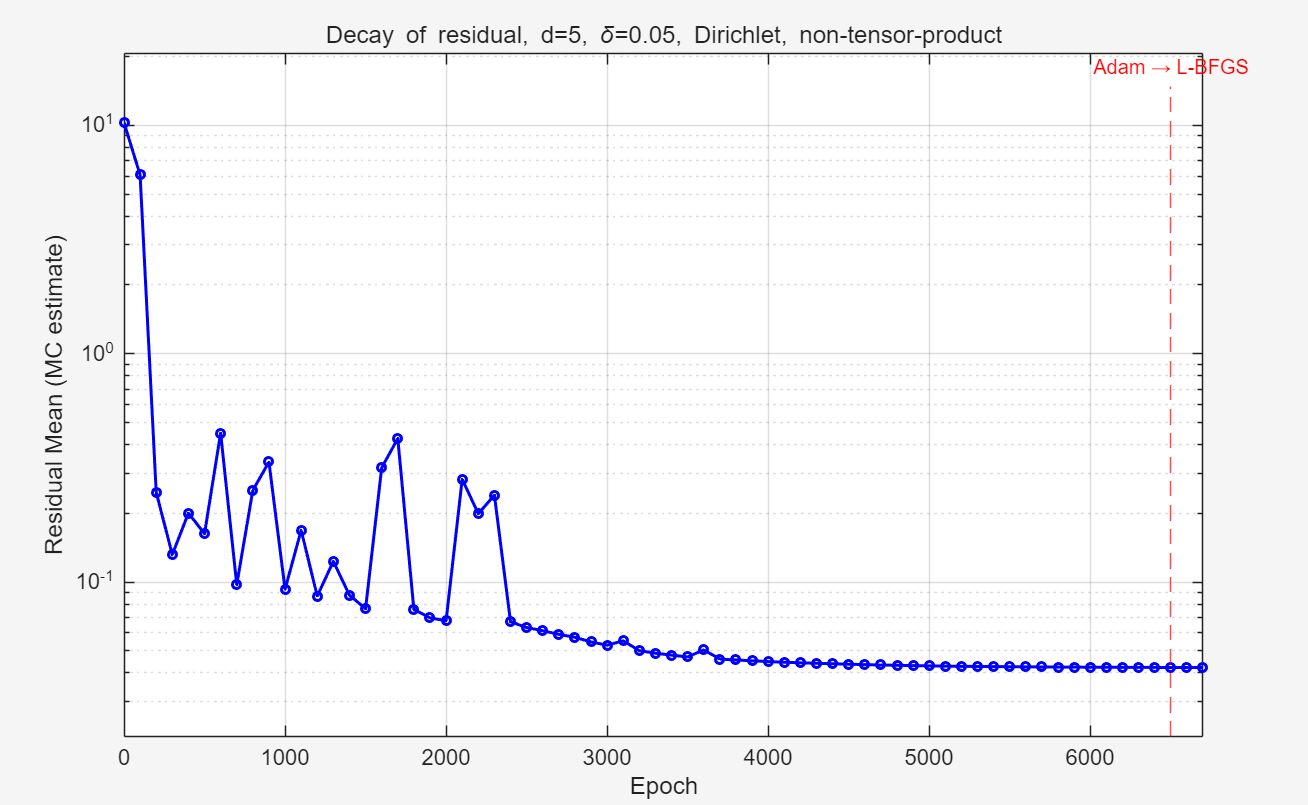}
  \caption{$d=5$}
\end{subfigure}

\vspace{0.5em}

\begin{subfigure}[b]{0.48\textwidth}
  \centering
  \includegraphics[width=\textwidth]{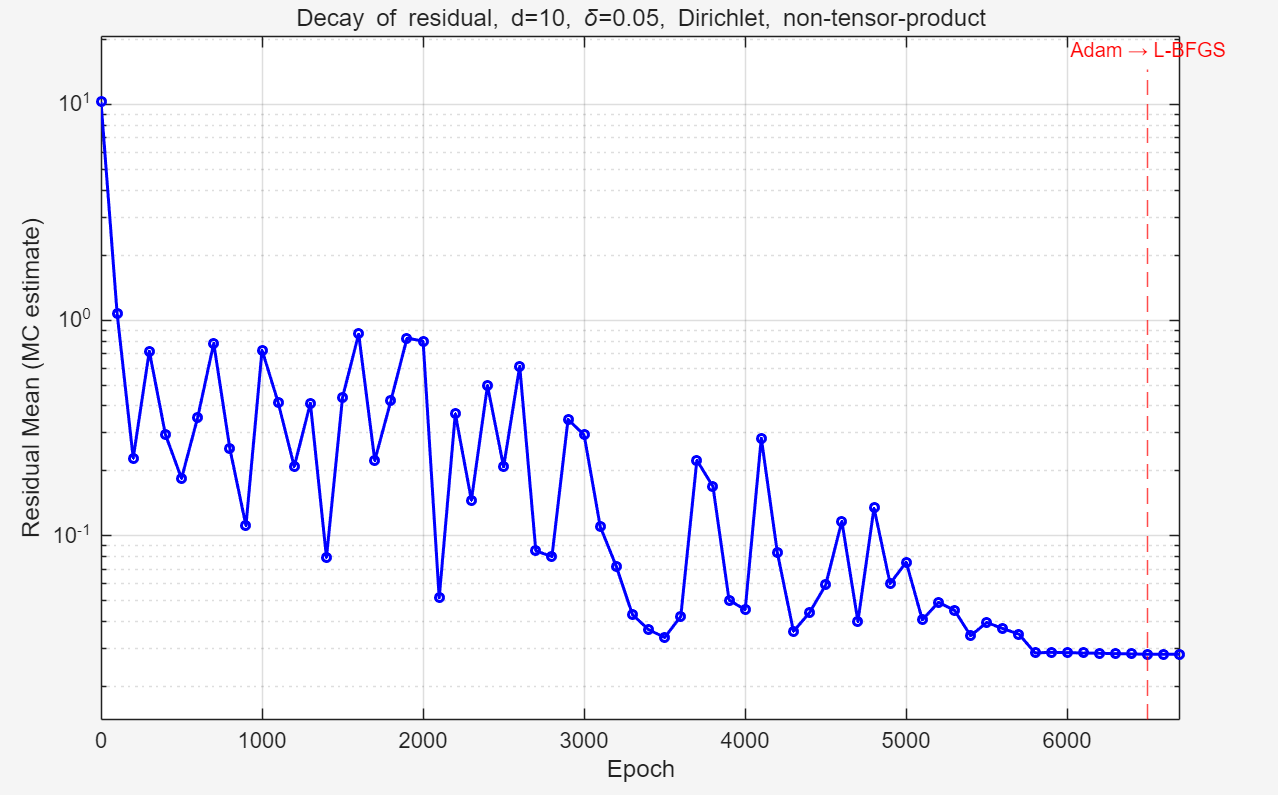}
  \caption{$d=10$}
\end{subfigure}
\hfill
\begin{subfigure}[b]{0.48\textwidth}
  \centering
  \includegraphics[width=\textwidth]{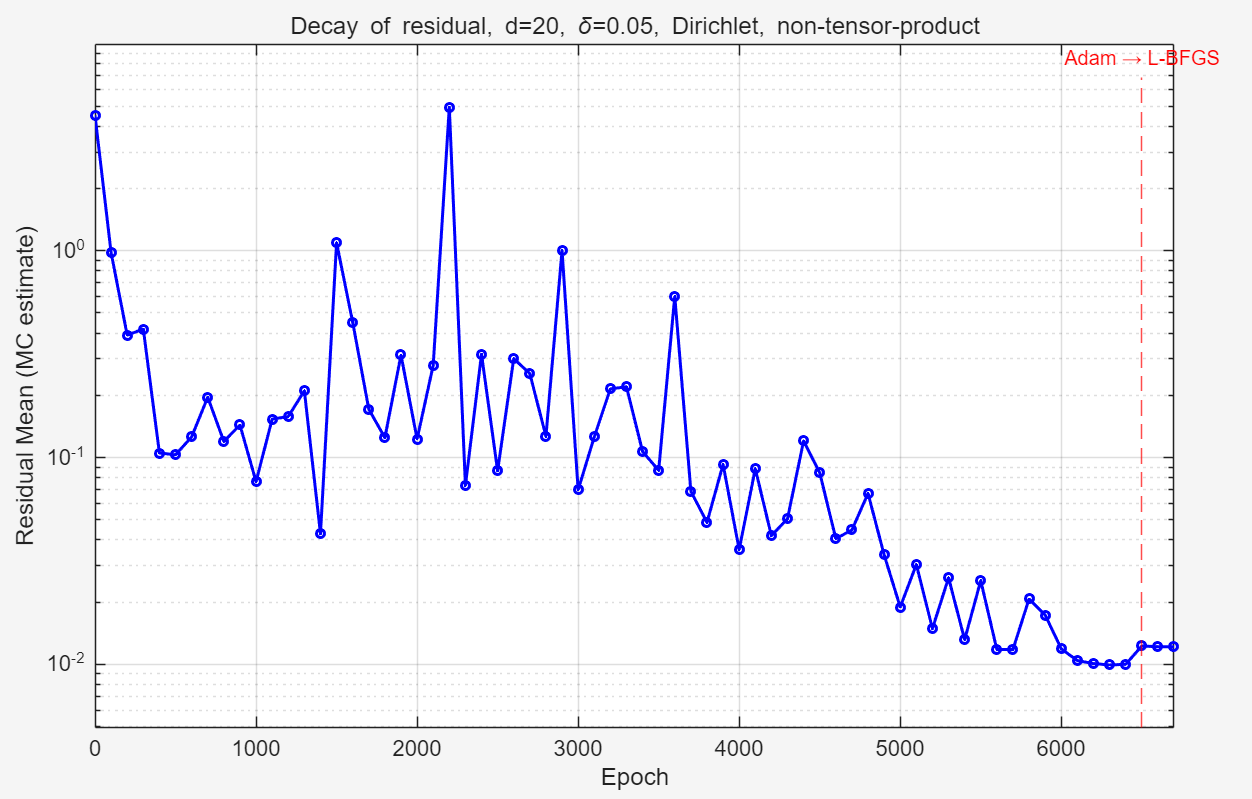}
  \caption{$d=20$}
\end{subfigure}
\caption{Mean pointwise residual $r_{\mathrm{mean}}$ versus training
         iteration for the Dirichlet case with non-tensor-product data
         ($d=3,5,10,20$).}
\label{fig:dirichlet-residual-nontensor}
\end{figure}

\begin{figure}[htbp]
\centering
\begin{subfigure}[b]{0.48\textwidth}
  \centering
  \includegraphics[width=\textwidth]{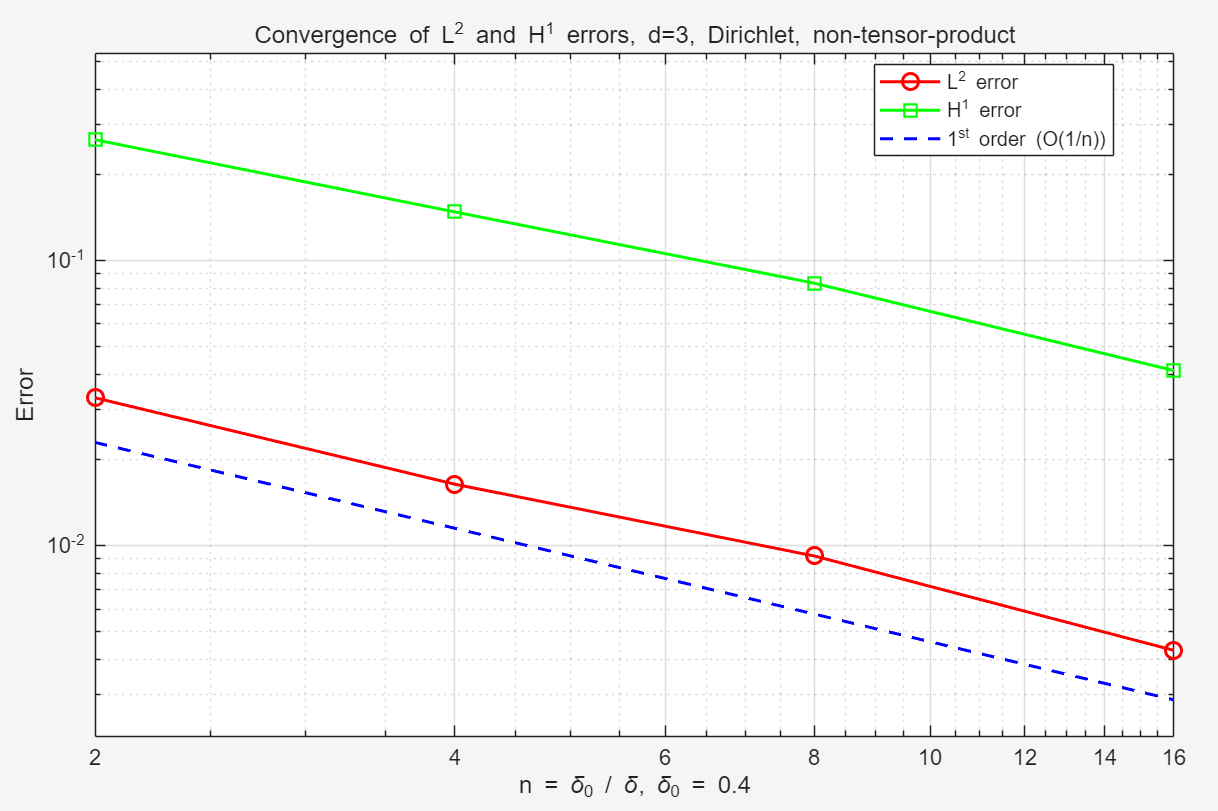}
  \caption{$d=3$}
\end{subfigure}
\hfill
\begin{subfigure}[b]{0.48\textwidth}
  \centering
  \includegraphics[width=\textwidth]{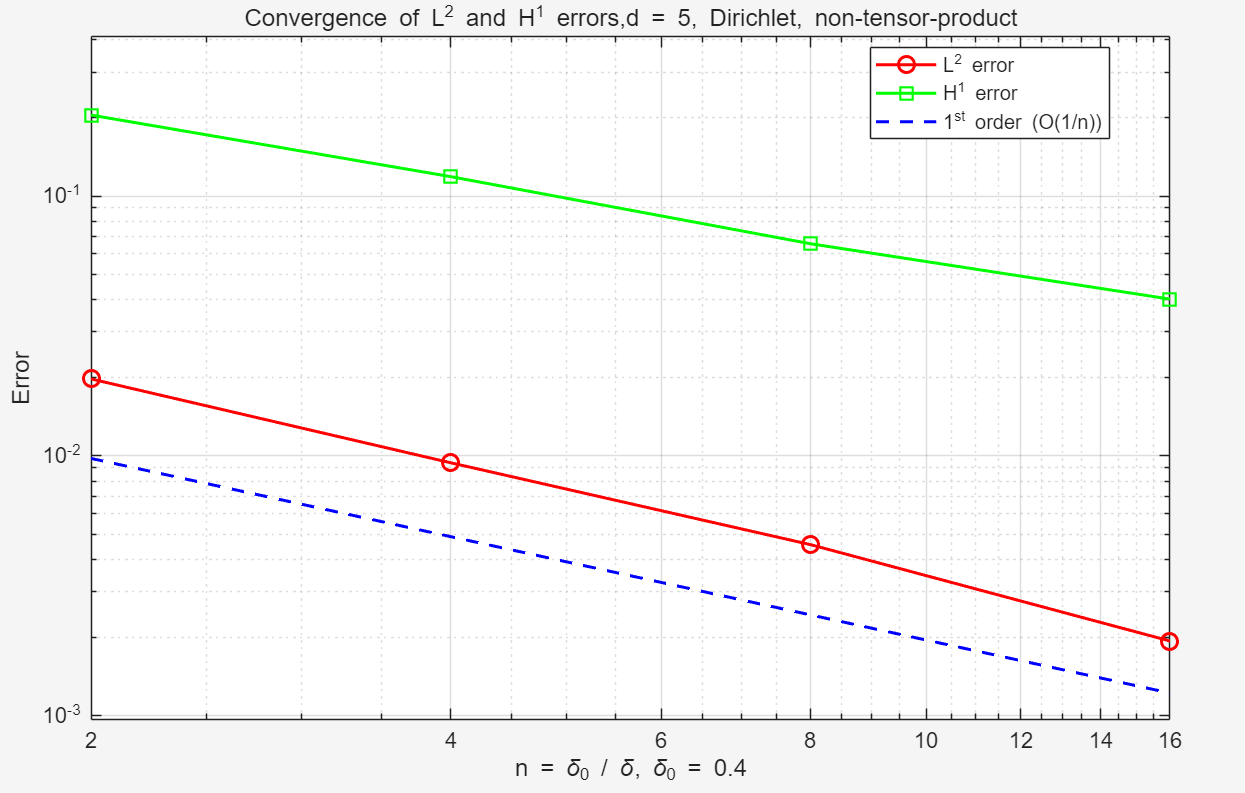}
  \caption{$d=5$}
\end{subfigure}

\vspace{0.5em}

\begin{subfigure}[b]{0.48\textwidth}
  \centering
  \includegraphics[width=\textwidth]{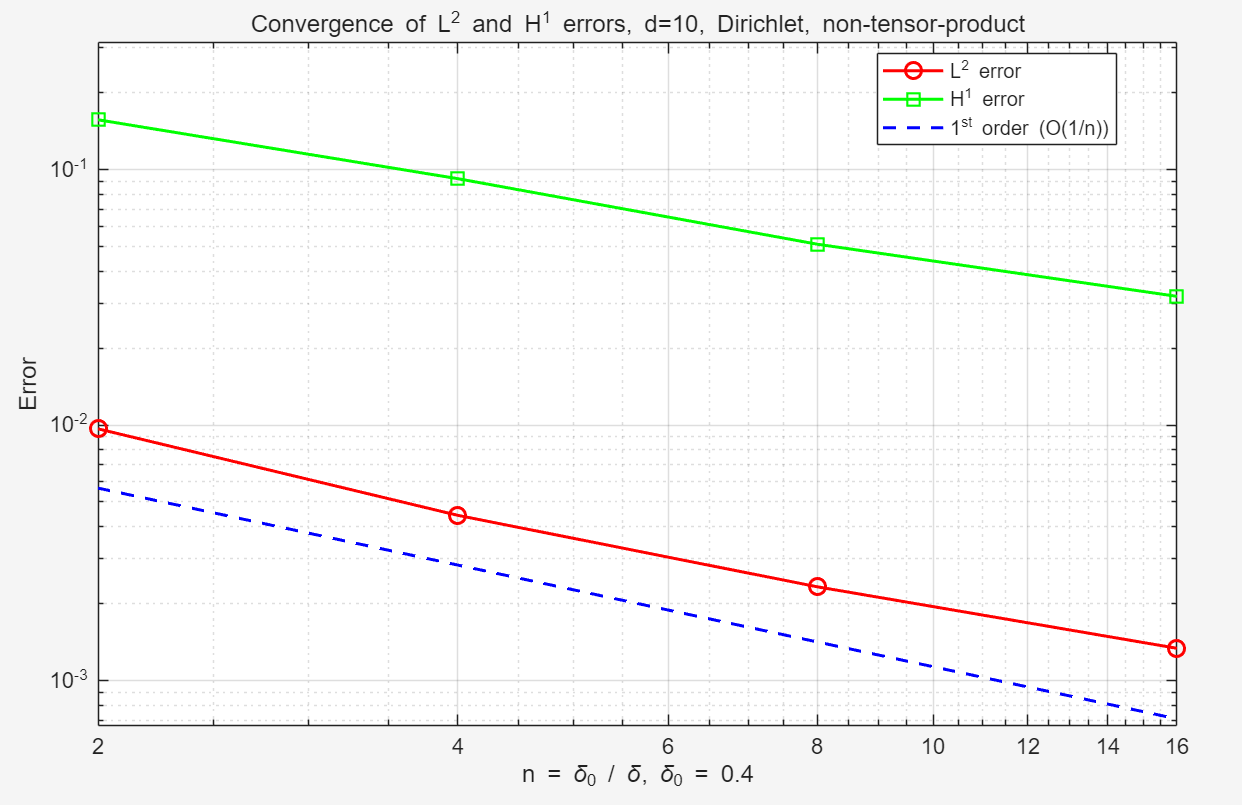}
  \caption{$d=10$}
\end{subfigure}
\hfill
\begin{subfigure}[b]{0.48\textwidth}
  \centering
  \includegraphics[width=\textwidth]{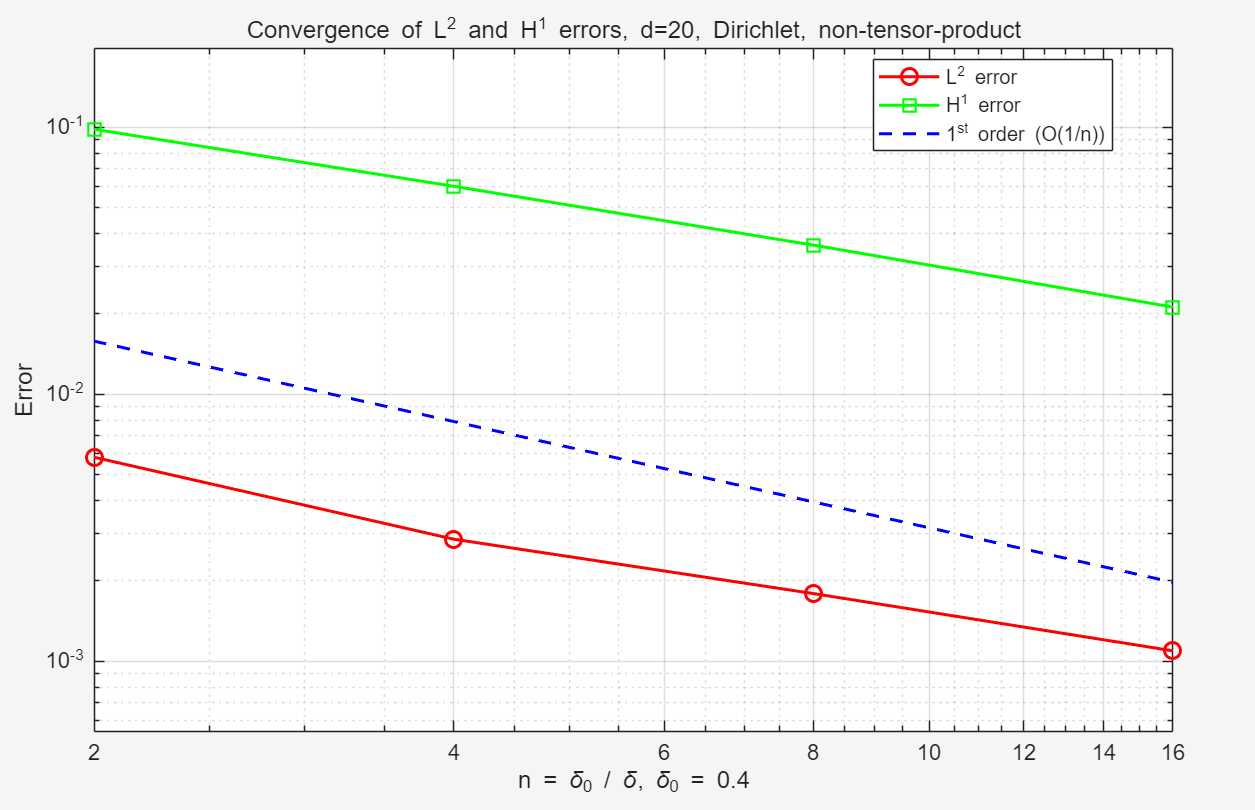}
  \caption{$d=20$}
\end{subfigure}
\caption{$L^2$ and $H^1$ errors between the local solution
         $u_{\mathrm{loc}}$ and the TNN output $u_{\delta,p}$ as functions
         of $\delta$ for the Dirichlet case with non-tensor-product data
         ($d=3,5,10,20$).}
\label{fig:dirichlet-errors-nontensor}
\end{figure}

\subsection{Experiments on \texorpdfstring{$L$}{L}-shaped domains}
\label{subsec:lshape-experiments}

We now examine the TNN-based method on $L$-shaped domains in two and three
dimensions (see Figure~\ref{fig:L-region}) under Neumann boundary
conditions.

\begin{figure}[htbp]
    \centering
    \begin{subfigure}{0.30\textwidth}
        \centering
        \includegraphics[width=\textwidth]{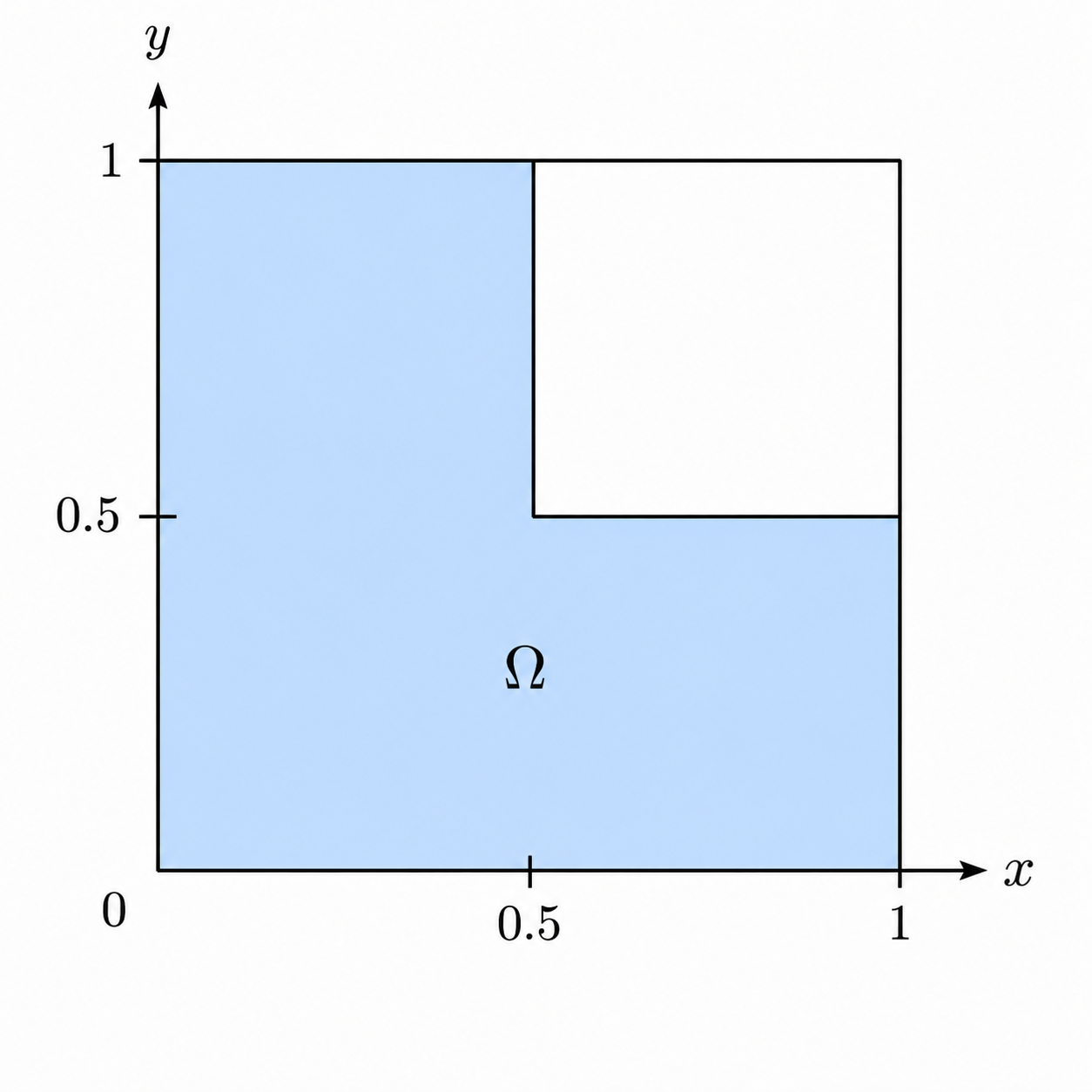}
        \caption{Two-dimensional $L$-shaped region.}
        \label{fig:L-region-2d}
    \end{subfigure}
    \hspace{0.06\textwidth}
    \begin{subfigure}{0.30\textwidth}
        \centering
        \includegraphics[width=\textwidth]{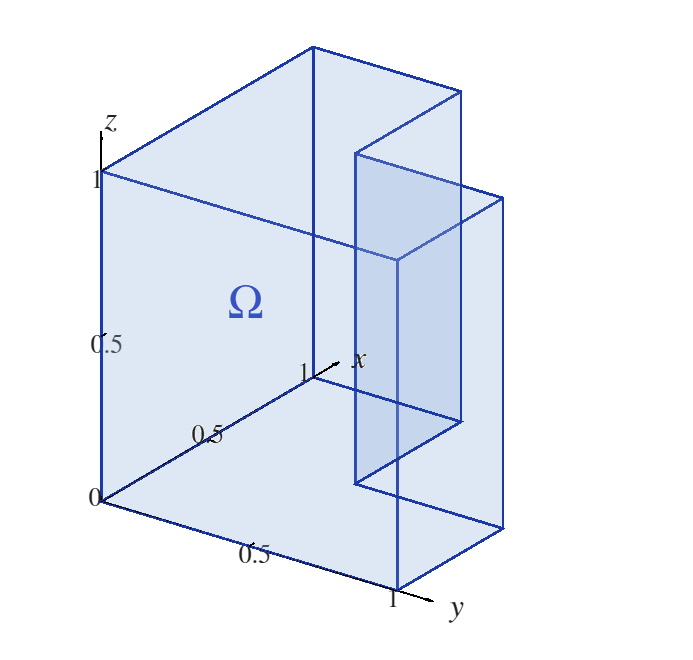}
        \caption{Three-dimensional $L$-shaped region.}
        \label{fig:L-region-3d}
    \end{subfigure}
    \caption{Illustration of the $L$-shaped computational regions.}
    \label{fig:L-region}
\end{figure}

\subsubsection{Two-dimensional case}
\label{subsubsec:lshape-2d}

We first consider the two-dimensional setting. The domain $\Omega$ is
\[
\Omega = \bigl([0,1]\times[0,0.5]\bigr) \cup \bigl([0,0.5]\times[0.5,1]\bigr)
\]
(see Figure~\ref{fig:L-region-2d}), and the exact solution is taken to be
\[
u_{\mathrm{loc}}(x,y) = \exp\!\bigl(0.2x+0.3y+0.2xy\bigr).
\]
The corresponding source term $f$ and boundary data $g$ are computed
through \eqref{eq:neumann-local}. Although $\Omega$ is not of
tensor-product form, the source term $f$ and the nonlocal solution
$u_{\delta,N}$, both viewed as elements of $H^1(\Omega)$, can be extended
to elements of $H^1([0,1]^2)$ via a standard Sobolev extension. Since
the TNN class is dense in $H^1([0,1]^2)$ by
Proposition~\ref{prop:tnn-approximation}, restricting any such TNN
approximation to $\Omega$ yields a TNN approximation in $H^1(\Omega)$. Thus
the approximation mechanism used in
Theorems~\ref{thm:neumann-error-l2}--\ref{thm:neumann-error-h1} remains
applicable. The rigorous boundary-layer estimates in the theorems were stated
for sufficiently smooth domains, so the following $L$-shaped-domain tests
should be viewed as numerical evidence beyond the smooth-domain theory.

We approximate the boundary data $g$ face by face on the six
boundary line segments of $\partial\Omega$:
\begin{align*}
&\{x\in[0,1],\,y=0\},\quad
\{x=1,\,y\in[0,0.5]\},\quad
\{x\in[0.5,1],\,y=0.5\},\\
&\{x=0.5,\,y\in[0.5,1]\},\quad
\{x\in[0,0.5],\,y=1\},\quad
\{x=0,\,y\in[0,1]\}.
\end{align*}
Following this procedure, the method can be applied as in
Section~\ref{subsec:nontensor-experiments}. The numerical results are
displayed in Figures~\ref{fig:lshape-2d-residual}
and~\ref{fig:lshape-2d-error}, supporting the empirical robustness of the method on the
two-dimensional $L$-shaped domain.

\subsubsection{Three-dimensional case}
\label{subsubsec:lshape-3d}

We now consider a three-dimensional $L$-shaped domain
\[
\Omega = \bigl([0,1]\times[0,0.5]\times[0,1]\bigr)
       \cup \bigl([0,0.5]\times[0.5,1]\times[0,1]\bigr)
\]
(see Figure~\ref{fig:L-region-3d}), with exact solution
\[
u_{\mathrm{loc}}(x,y,z)
= \exp\!\bigl(0.2x+0.3y+0.1z+0.2xy+0.15yz+0.1xz\bigr),
\]
and the corresponding $f$ and $g$ defined through
\eqref{eq:neumann-local}. As in the two-dimensional case, $f$ and the
nonlocal solution $u_{\delta,N}$ can be extended from $H^1(\Omega)$ to
$H^1([0,1]^{3})$ and approximated by TNN functions. The boundary data
$g$ is approximated face by face on the ten rectangular faces of
$\partial\Omega$:
\begin{align*}
&\{(x,y,z)\in[0,1]\times[0,0.5]\times\{0\}\},\quad
 \{(x,y,z)\in[0,0.5]\times[0.5,1]\times\{0\}\},\\
&\{(x,y,z)\in[0,1]\times[0,0.5]\times\{1\}\},\quad
 \{(x,y,z)\in[0,0.5]\times[0.5,1]\times\{1\}\},\\
&\{(x,y,z)\in[0,1]\times\{0\}\times[0,1]\},\quad
 \{(x,y,z)\in\{1\}\times[0,0.5]\times[0,1]\},\\
&\{(x,y,z)\in[0.5,1]\times\{0.5\}\times[0,1]\},\quad
 \{(x,y,z)\in\{0.5\}\times[0.5,1]\times[0,1]\},\\
&\{(x,y,z)\in[0,0.5]\times\{1\}\times[0,1]\},\quad
 \{(x,y,z)\in\{0\}\times[0,1]\times[0,1]\}.
\end{align*}
As in the two-dimensional case, this experiment uses the same approximation
mechanism but lies outside the smooth-domain assumptions used in the proof of
Theorems~\ref{thm:neumann-error-l2}--\ref{thm:neumann-error-h1}. The numerical
results are reported in Figures~\ref{fig:lshape-3d-residual}
and~\ref{fig:lshape-3d-error}, confirming the practical robustness of the
method on the three-dimensional $L$-shaped domain.

\begin{figure}[htbp]
\centering
\begin{subfigure}[b]{0.48\textwidth}
  \centering
  \includegraphics[width=\textwidth]{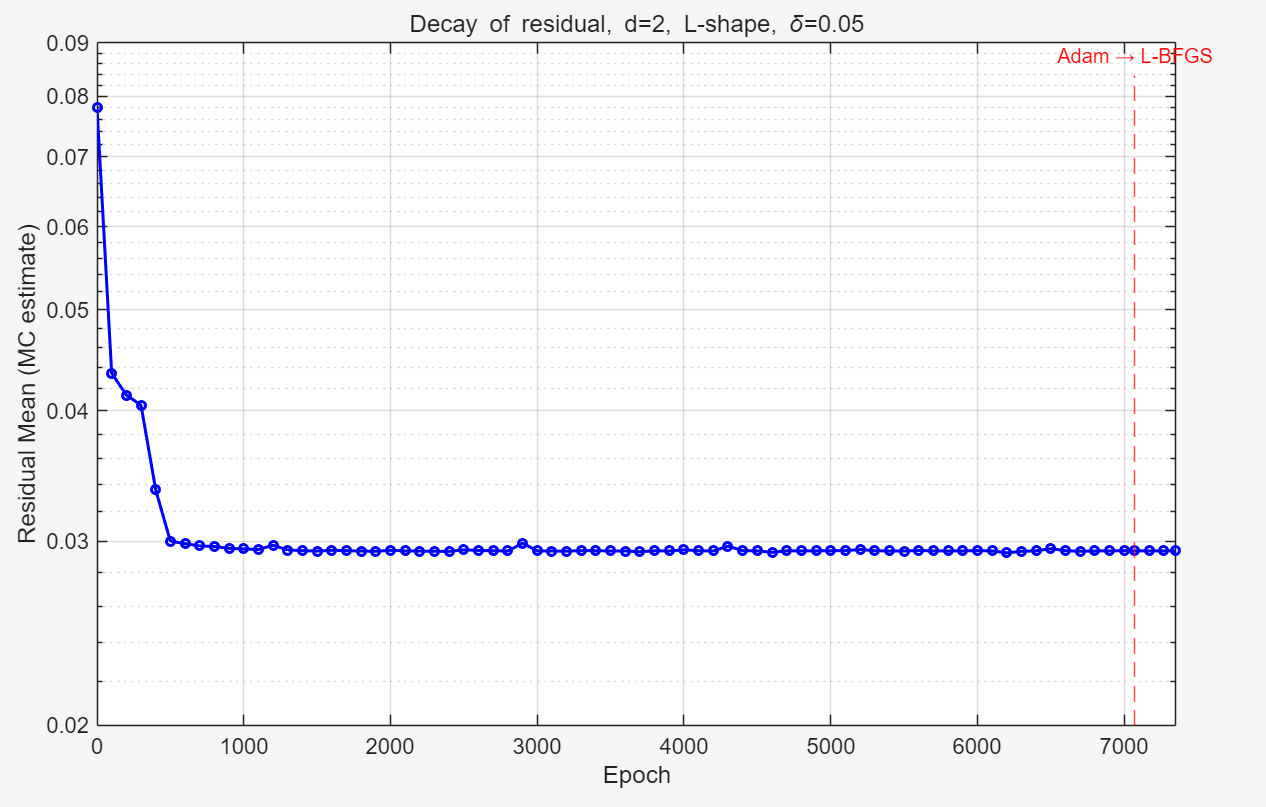}
  \caption{Mean residual ($d=2$).}
  \label{fig:lshape-2d-residual}
\end{subfigure}
\hfill
\begin{subfigure}[b]{0.48\textwidth}
  \centering
  \includegraphics[width=\textwidth]{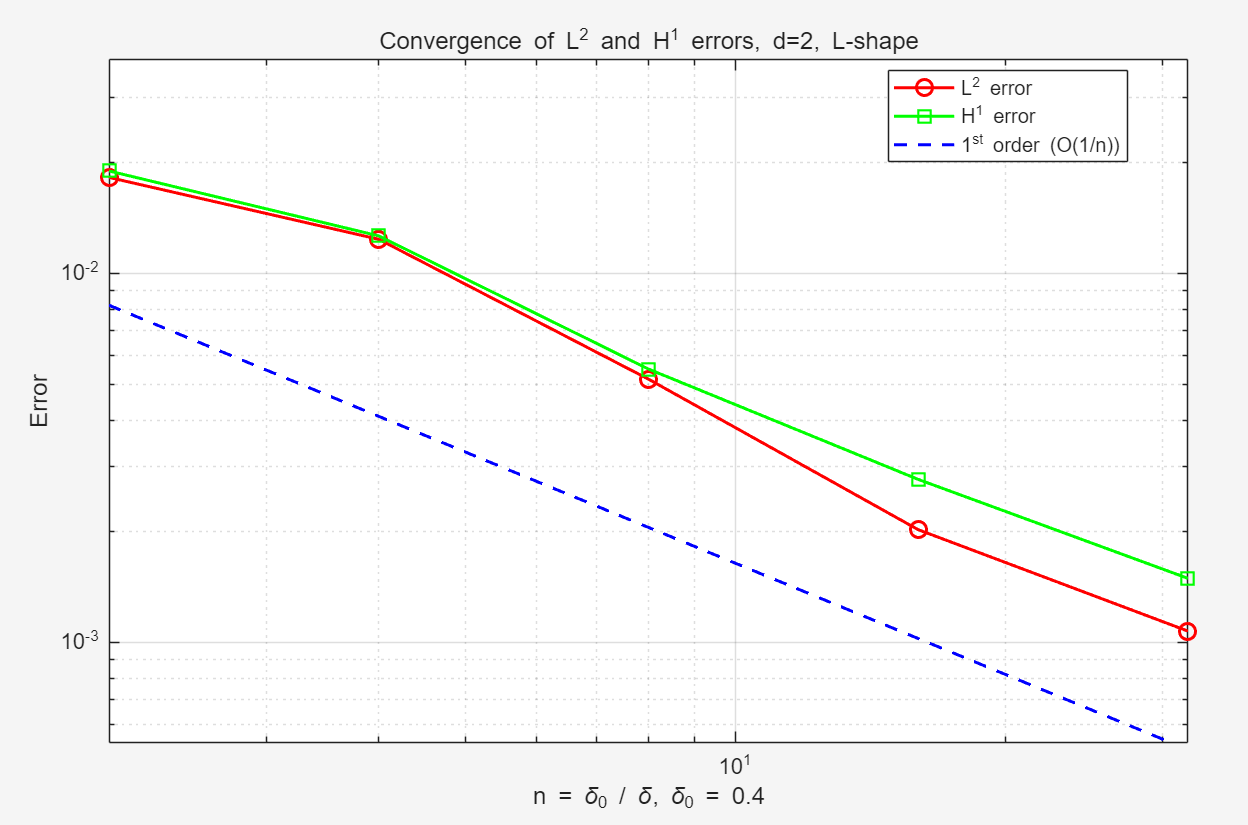}
  \caption{$L^2$ and $H^1$ errors ($d=2$).}
  \label{fig:lshape-2d-error}
\end{subfigure}

\vspace{0.5em}

\begin{subfigure}[b]{0.48\textwidth}
  \centering
  \includegraphics[width=\textwidth]{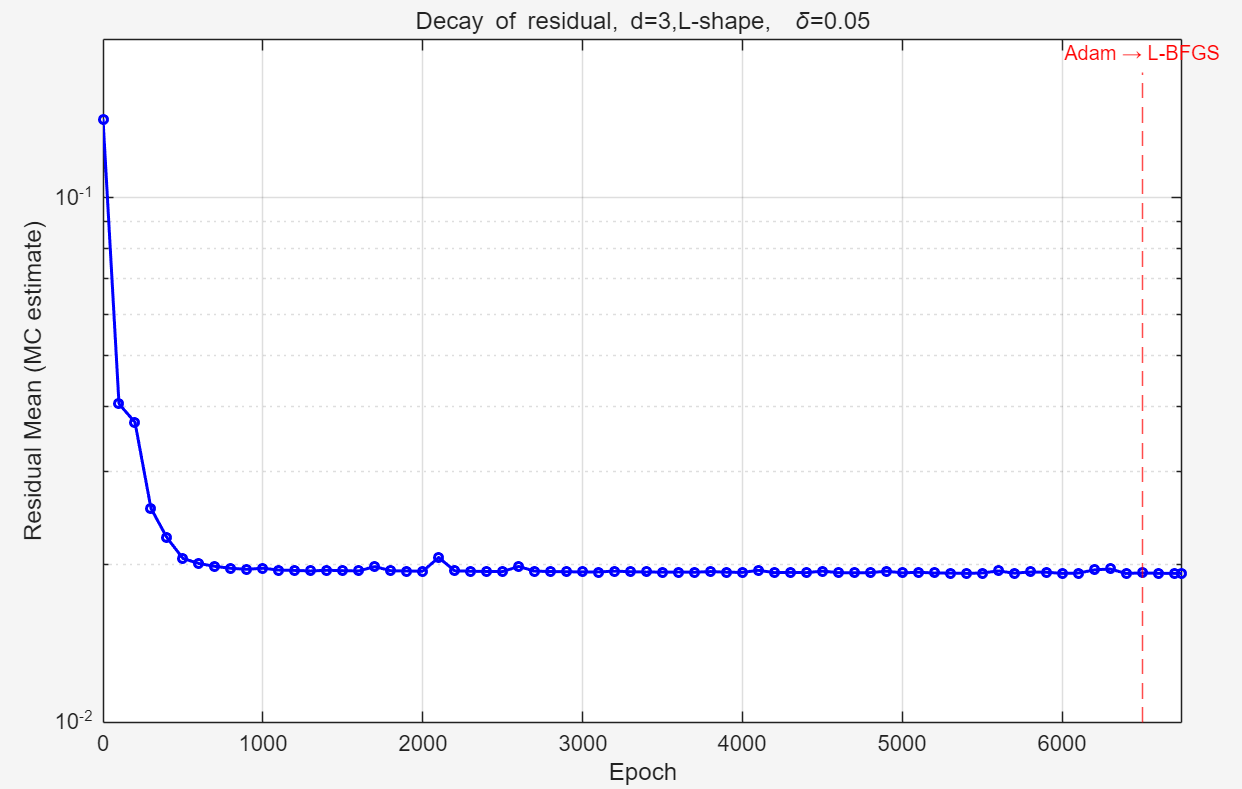}
  \caption{Mean residual ($d=3$).}
  \label{fig:lshape-3d-residual}
\end{subfigure}
\hfill
\begin{subfigure}[b]{0.48\textwidth}
  \centering
  \includegraphics[width=\textwidth]{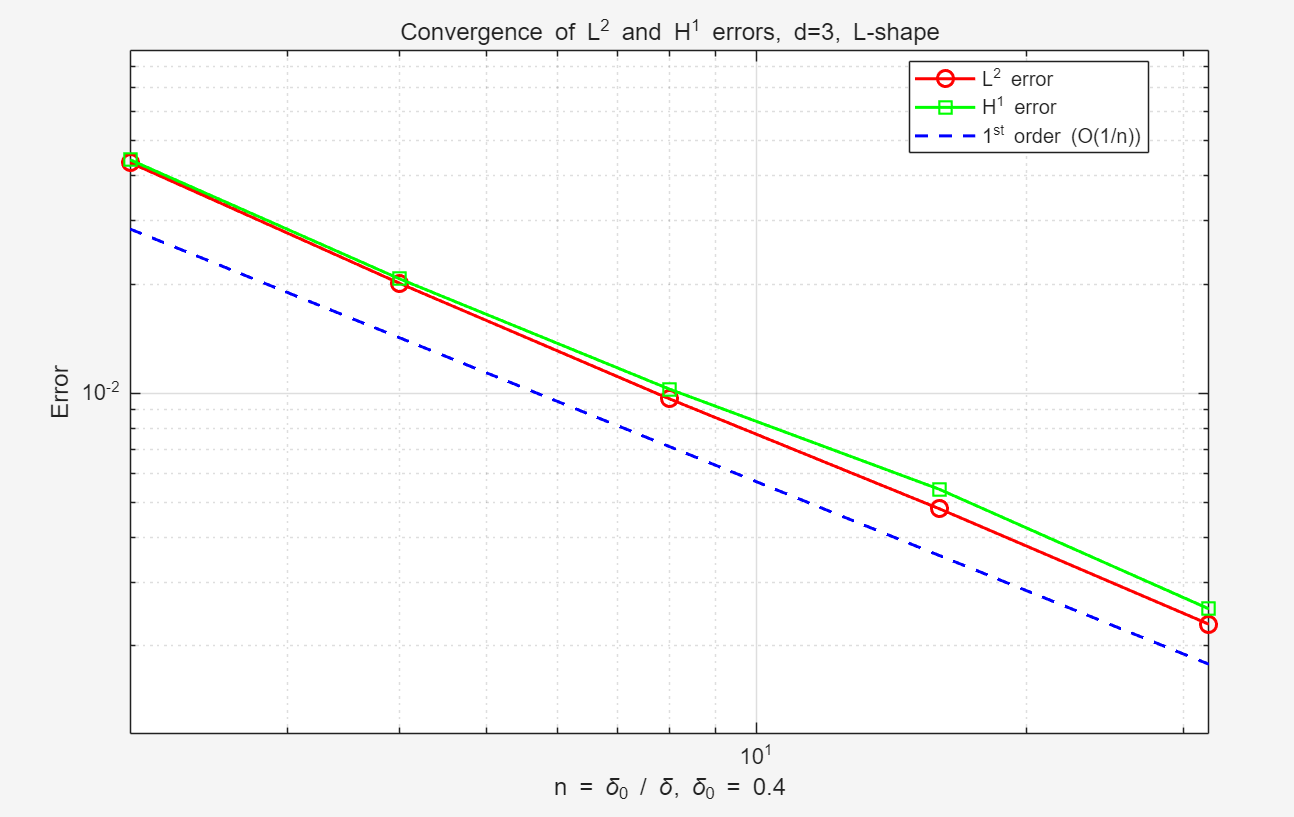}
  \caption{$L^2$ and $H^1$ errors ($d=3$).}
  \label{fig:lshape-3d-error}
\end{subfigure}
\caption{Mean pointwise residual $r_{\mathrm{mean},N}$ and the $L^2$,
         $H^1$ errors between the local solution $u_{\mathrm{loc},N}$ and
         the TNN output $u_{\delta,p,N}$ on $L$-shaped regions under
         Neumann boundary conditions.}
\label{fig:lshape-results}
\end{figure}

\section{Conclusions}
\label{sec:conclusions}

In this paper we have constructed and analyzed a variational solver for
nonlocal diffusion models with Dirichlet and Neumann boundary conditions using
the existing tensor neural network architecture as the trial class. The scheme
combines a TNN ansatz with TNN-based preconditioning of the data; for Gaussian
kernels on
tensor-product or rectangular partitionable domains, every term in the
modified loss reduces to products of low-dimensional integrals evaluated
by composite Gauss--Legendre quadrature. We have established
asymptotically compatible error estimates for both boundary conditions,
decomposing the total error into data-preconditioning, trial-class
approximation, and optimization contributions. In the Neumann case a
gradient estimate is further obtained via a smoothing post-processing.
Numerical experiments on tensor-product domains with both tensor-product and
non-tensor-product data corroborate the theoretical results, while additional
experiments on $L$-shaped domains demonstrate the practical robustness of the
method beyond the smooth-domain setting covered by the analysis.

Several directions for future work suggest themselves: extending the
method to general nonlocal kernels without separable structure; providing
a rigorous derivation of the sharp convergence rate and an $H^1$ estimate
in the Dirichlet case; relaxing the rectangular-domain assumption to more
general geometries; and applying this TNN-based nonlocal solver to time-dependent problems,
multiscale material models, and fracture mechanics.

\appendix

\section{Proof of Proposition~\ref{prop:neumann-wellposedness}}
\label{appendix:neumann-wellposedness}

We prove Proposition~\ref{prop:neumann-wellposedness}. Throughout this
appendix, $E_{\delta,N}(u)$ denotes the energy functional defined in
\eqref{eq:neumann-energy}. The bilinear form associated with
$E_{\delta,N}$ is continuous and coercive on $L^2(\Omega)$ by
\cite[Lemma~3.1]{meng2026asymptotically}, while the right-hand side of
\eqref{eq:neumann-nonlocal} is bounded on $L^2(\Omega)$ by the estimates below.
Thus the Lax--Milgram theorem gives a unique $L^2$ solution; the estimates
below show that this solution belongs to $H^1(\Omega)$. Let $u_{\delta,N}$ be
this solution. Testing \eqref{eq:neumann-nonlocal} with $u_{\delta,N}$ and
integrating over $\Omega$, we obtain
\begin{align}
E_{\delta,N}(u_{\delta,N})
&=
\int_\Omega\!\!\int_\Omega
\bar R_\delta(\xx,\yy)\,
u_{\delta,N}(\xx)f(\yy)
\,\d\xx\,\d\yy
\notag \\
&\quad
+
2\int_\Omega\!\!\int_{\partial\Omega}
\bar R_\delta(\xx,\yy)\,
u_{\delta,N}(\xx)g(\yy)
\,\d\xx\,\d S_\yy .
\label{eq:neumann-energy-identity}
\end{align}

By the Cauchy--Schwarz inequality and the kernel estimate
\eqref{eq:kernel-est-2}, we have
\begin{align}
&\int_\Omega\!\!\int_\Omega
\bar R_\delta(\xx,\yy)\,
u_{\delta,N}(\xx)f(\yy)
\,\d\xx\,\d\yy
\notag \\
&\quad \le
\!\left(
\int_\Omega\!\!\int_\Omega
\bar R_\delta(\xx,\yy)\,
u_{\delta,N}(\xx)^{2}
\,\d\xx\,\d\yy\right)^{\!1/2}
\!\!\left(
\int_\Omega\!\!\int_\Omega
\bar R_\delta(\xx,\yy)\,
f(\yy)^{2}
\,\d\xx\,\d\yy\right)^{\!1/2}
\notag \\
&\quad \le
C\,\|u_{\delta,N}\|_{L^2(\Omega)}\,
\|f\|_{L^2(\Omega)} .
\label{eq:neumann-source-est}
\end{align}
Similarly, by Cauchy--Schwarz and the boundary kernel estimate
\eqref{eq:kernel-est-3},
\begin{align}
&\int_\Omega\!\!\int_{\partial\Omega}
\bar R_\delta(\xx,\yy)\,
u_{\delta,N}(\xx)g(\yy)
\,\d\xx\,\d S_\yy
\notag \\
&\quad \le
\frac{C}{\sqrt{\delta}}\,
\|u_{\delta,N}\|_{L^2(\Omega)}\,
\|g\|_{L^2(\partial\Omega)} .
\label{eq:neumann-boundary-est}
\end{align}
Combining
\eqref{eq:neumann-energy-identity}--\eqref{eq:neumann-boundary-est},
\begin{equation}
\label{eq:neumann-energy-upper}
E_{\delta,N}(u_{\delta,N})
\le
C\,\|u_{\delta,N}\|_{L^2(\Omega)}
\!\left(\|f\|_{L^2(\Omega)}
+\frac{1}{\sqrt{\delta}}\,\|g\|_{L^2(\partial\Omega)}\right).
\end{equation}

By the coercivity estimate of the nonlocal
energy~\cite[Lemma~3.1]{meng2026asymptotically},
\begin{equation}
\label{eq:neumann-coercivity}
\|u_{\delta,N}\|_{L^2(\Omega)}^{2}
\le
C\,E_{\delta,N}(u_{\delta,N}).
\end{equation}
Combining \eqref{eq:neumann-energy-upper} and
\eqref{eq:neumann-coercivity} yields
\begin{equation}
\label{eq:neumann-l2-est}
\|u_{\delta,N}\|_{L^2(\Omega)}
\le
C\!\left(\|f\|_{L^2(\Omega)}
+\frac{1}{\sqrt{\delta}}\,\|g\|_{L^2(\partial\Omega)}\right).
\end{equation}

It remains to estimate the $H^1$-seminorm of $u_{\delta,N}$. From
\eqref{eq:neumann-nonlocal} we may rewrite $u_{\delta,N}$ as
\begin{align}
u_{\delta,N}(\xx)
&=
S_\delta u_{\delta,N}(\xx)
+
\frac{\delta^{2}}{w_\delta(\xx)}
\int_\Omega
\bar R_\delta(\xx,\yy)f(\yy)
\,\d\yy
\notag \\
&\quad
+
\frac{2\delta^{2}}{w_\delta(\xx)}
\int_{\partial\Omega}
\bar R_\delta(\xx,\yy)g(\yy)
\,\d S_\yy
-
\frac{\delta^{2}}{w_\delta(\xx)}
\int_\Omega
\bar R_\delta(\xx,\yy)u_{\delta,N}(\yy)
\,\d\yy
\notag \\
&=: S_\delta u_{\delta,N}(\xx)
+ I_1(\xx) + I_2(\xx) + I_3(\xx) ,
\label{eq:neumann-decomposition}
\end{align}
where $w_\delta(\xx)$ and $S_\delta$ are defined in \eqref{eq:mollifier}. It
therefore suffices to estimate
\[
\|\nabla S_\delta u_{\delta,N}\|_{L^2(\Omega)},\quad
\|\nabla I_1\|_{L^2(\Omega)},\quad
\|\nabla I_2\|_{L^2(\Omega)},\quad
\|\nabla I_3\|_{L^2(\Omega)} .
\]

\medskip
\noindent\emph{Estimate of $\nabla I_1$.} Differentiating $I_1$ gives
\begin{align}
\nabla I_1(\xx)
&=
\frac{\delta^{2}}{w_\delta(\xx)}
\int_\Omega
\nabla_{\!\xx}\bar R_\delta(\xx,\yy)\, f(\yy)
\,\d\yy
\notag \\
&\quad
-
\frac{\delta^{2}\nabla w_\delta(\xx)}{w_\delta(\xx)^{2}}
\int_\Omega
\bar R_\delta(\xx,\yy)\, f(\yy)
\,\d\yy
\notag \\
&=: I_{1,1}(\xx)+I_{1,2}(\xx).
\label{eq:I1-grad-decomp}
\end{align}
Using
$|\nabla_{\!\xx}\bar R_\delta(\xx,\yy)|
=\dfrac{|\xx-\yy|}{2\delta^2}\,R_\delta(\xx,\yy)$ and the scaled
second-moment bound
\[
\int_\Omega \frac{|\xx-\yy|^2}{\delta^4}
R_\delta(\xx,\yy)\,\d\yy \le \frac{C}{\delta^2},
\]
the Cauchy--Schwarz inequality gives
\begin{align}
\|I_{1,1}\|_{L^2(\Omega)}^{2}
&\le
C\delta^{4}
\int_\Omega
\!\left(
\int_\Omega
\frac{|\xx-\yy|}{2\delta^{2}}\,
R_\delta(\xx,\yy)\,|f(\yy)|
\,\d\yy\right)^{\!2}
\d\xx
\notag \\
&\le
C\delta^{2}
\int_\Omega
\int_\Omega R_\delta(\xx,\yy)\,f(\yy)^2\,\d\yy\,\d\xx
\le
C\delta^{2}\,\|f\|_{L^2(\Omega)}^{2} .
\label{eq:I11-est}
\end{align}
Using the bound $|\nabla w_\delta(\xx)|\le C/\delta$ from
\eqref{eq:wdelta-grad}, we similarly
obtain
\begin{align}
\|I_{1,2}\|_{L^2(\Omega)}^{2}
&\le
C\delta^{4}
\int_\Omega
|\nabla w_\delta(\xx)|^{2}
\!\left(
\int_\Omega
\bar R_\delta(\xx,\yy)\,|f(\yy)|
\,\d\yy\right)^{\!2}
\d\xx
\notag \\
&\le
C\delta^{2}
\int_\Omega
\!\left(
\int_\Omega
\bar R_\delta(\xx,\yy)\,|f(\yy)|
\,\d\yy\right)^{\!2}
\d\xx
\notag \\
&\le
C\delta^{2}\,\|f\|_{L^2(\Omega)}^{2} .
\label{eq:I12-est}
\end{align}
Consequently,
\begin{equation}
\label{eq:I1-est}
\|\nabla I_1\|_{L^2(\Omega)}
\le
C\delta\,\|f\|_{L^2(\Omega)} .
\end{equation}

\medskip
\noindent\emph{Estimate of $\nabla I_2$.} Splitting $\nabla I_2$ as in
\eqref{eq:I1-grad-decomp}, the dominant term is
\[
\frac{2\delta^{2}}{w_\delta(\xx)}
\int_{\partial\Omega}
\nabla_{\!\xx}\bar R_\delta(\xx,\yy)\,g(\yy)\,\d S_\yy .
\]
Using
\[
|\nabla_{\!\xx}\bar R_\delta(\xx,\yy)|
\le \frac{|\xx-\yy|}{2\delta^{2}}R_\delta(\xx,\yy),
\]
the Cauchy--Schwarz inequality, and the boundary moment estimate
\[
\int_{\partial\Omega} |\xx-\yy|^2 R_\delta(\xx,\yy)\,\d S_\yy
\le C\delta,
\]
we obtain the required bound. The term containing $\nabla w_\delta$ is handled
similarly by \eqref{eq:wdelta-grad} and \eqref{eq:kernel-est-3}.
Integrating in $\xx$ and using
$\int_\Omega R_\delta(\xx,\yy)\,\d\xx \le C$ for any
$\yy\in\partial\Omega$, we deduce
\begin{equation}
\label{eq:I2-est}
\|\nabla I_2\|_{L^2(\Omega)}
\le C\sqrt{\delta}\,\|g\|_{L^2(\partial\Omega)} .
\end{equation}

\medskip
\noindent\emph{Estimate of $\nabla I_3$.} The same calculation as for
$\nabla I_1$, with $u_{\delta,N}$ in place of $f$, gives
\begin{equation}
\label{eq:I3-est}
\|\nabla I_3\|_{L^2(\Omega)}
\le
C\delta\,\|u_{\delta,N}\|_{L^2(\Omega)}
\le
C\delta\,\|f\|_{L^2(\Omega)}
+
C\sqrt{\delta}\,\|g\|_{L^2(\partial\Omega)} ,
\end{equation}
where in the second inequality we used \eqref{eq:neumann-l2-est}.

\medskip
\noindent\emph{Estimate of $\nabla S_\delta u_{\delta,N}$.} By
\cite[Lemma~3.1]{meng2026asymptotically}, the smoothing operator $S_\delta$
satisfies
\begin{equation}
\label{eq:Sdelta-grad}
\|\nabla S_\delta u_{\delta,N}\|_{L^2(\Omega)}^{2}
\le
C\,E_{\delta,N}(u_{\delta,N}).
\end{equation}
Combining \eqref{eq:Sdelta-grad} with \eqref{eq:neumann-energy-upper} and
\eqref{eq:neumann-l2-est} gives
\begin{equation}
\label{eq:Sdelta-final}
\|\nabla S_\delta u_{\delta,N}\|_{L^2(\Omega)}
\le
C\!\left(\|f\|_{L^2(\Omega)}
+\frac{1}{\sqrt{\delta}}\,\|g\|_{L^2(\partial\Omega)}\right).
\end{equation}

Substituting \eqref{eq:I1-est}, \eqref{eq:I2-est}, \eqref{eq:I3-est} and
\eqref{eq:Sdelta-final} into \eqref{eq:neumann-decomposition}, we obtain
\begin{equation}
\label{eq:neumann-grad-est}
\|\nabla u_{\delta,N}\|_{L^2(\Omega)}
\le
C\!\left(\|f\|_{L^2(\Omega)}
+\frac{1}{\sqrt{\delta}}\,\|g\|_{L^2(\partial\Omega)}\right).
\end{equation}
Finally, combining \eqref{eq:neumann-l2-est} and
\eqref{eq:neumann-grad-est} yields
\[
\|u_{\delta,N}\|_{H^1(\Omega)}
\le
C\!\left(\|f\|_{L^2(\Omega)}
+\frac{1}{\sqrt{\delta}}\,\|g\|_{L^2(\partial\Omega)}\right),
\]
which proves Proposition~\ref{prop:neumann-wellposedness}. \qed

\section{Proof of Proposition~\ref{prop:neumann-locallimit}}
\label{appendix:neumann-locallimit}

We now prove Proposition~\ref{prop:neumann-locallimit}. Set
\[
e_{\delta,N}(\xx):=
u_{\mathrm{loc},N}(\xx)-u_{\delta,N}(\xx).
\]
Since both $u_{\delta,N}$ and $u_{\mathrm{loc},N}$ belong to $H^1(\Omega)$,
$e_{\delta,N}\in H^1(\Omega)$, and a direct calculation using
\eqref{eq:neumann-local}--\eqref{eq:neumann-nonlocal} shows that
$e_{\delta,N}$ satisfies
\begin{equation}
\label{eq:neumann-error-eq}
\frac{1}{\delta^{2}}
\int_\Omega R_\delta(\xx,\yy)
\bigl(e_{\delta,N}(\xx)-e_{\delta,N}(\yy)\bigr)\,\d\yy
+
\int_\Omega
\bar R_\delta(\xx,\yy)e_{\delta,N}(\yy)\,\d\yy
=
r(\xx),
\quad \xx\in\Omega ,
\end{equation}
where the truncation error $r$ is given by
\begin{align*}
r(\xx)
&=
\frac{1}{\delta^{2}}
\int_\Omega R_\delta(\xx,\yy)
\bigl(u_{\mathrm{loc},N}(\xx)-u_{\mathrm{loc},N}(\yy)\bigr)
\,\d\yy
+
\int_\Omega
\bar R_\delta(\xx,\yy)
\Delta u_{\mathrm{loc},N}(\yy)
\,\d\yy
\\
&\quad
-
2\int_{\partial\Omega}
\bar R_\delta(\xx,\yy)\,
\frac{\partial u_{\mathrm{loc},N}}{\partial \mathbf n}(\yy)
\,\d S_\yy .
\end{align*}

The following truncation-error decomposition is standard for nonlocal Neumann
models and relies on the usual moment and boundary estimates of the kernel; see
\cite{shi2017convergence} for the compact-support version. The truncation
error admits a boundary-layer decomposition $r=r_{in}+r_{bd}$, where, for
$u_{\mathrm{loc},N}\in H^3(\Omega)$,
\begin{equation*}
r_{bd}(\xx)
=
\sum_{j=1}^d
\int_{\partial\Omega}
n^{j}(\yy)\,
(\xx-\yy)\!\cdot\!
\nabla\!\bigl(\partial_{j} u_{\mathrm{loc},N}(\yy)\bigr)\,
\bar R_\delta(\xx,\yy)
\,\d S_\yy ,
\end{equation*}
and $r_{in}:=r-r_{bd}$. Here $n^{j}(\yy)$ denotes the $j$-th component of the
unit outward normal $\mathbf n(\yy)$ at $\yy\in\partial\Omega$. The
following estimates hold:
\begin{align}
\|r_{in}\|_{L^2(\Omega)}
&\le
C\delta\,\|u_{\mathrm{loc},N}\|_{H^3(\Omega)},
&
\|r_{bd}\|_{L^2(\Omega)}
&\le
C\delta^{1/2}\,\|u_{\mathrm{loc},N}\|_{H^3(\Omega)},
\label{eq:rin-rbd-l2}
\\
\|\nabla r_{in}\|_{L^2(\Omega)}
&\le
C\,\|u_{\mathrm{loc},N}\|_{H^3(\Omega)},
&
\|\nabla r_{bd}\|_{L^2(\Omega)}
&\le
C\delta^{-1/2}\,\|u_{\mathrm{loc},N}\|_{H^3(\Omega)} .
\label{eq:rin-rbd-grad}
\end{align}
Moreover, for any $h\in H^1(\Omega)$,
\begin{equation}
\label{eq:rbd-dual}
\!\left|\int_\Omega r_{bd}(\xx)\,h(\xx)\,\d\xx\right|
\le
C\delta\,\|u_{\mathrm{loc},N}\|_{H^3(\Omega)}\,
\|h\|_{H^1(\Omega)} .
\end{equation}

From \eqref{eq:neumann-error-eq} and the same algebra as in
\eqref{eq:neumann-decomposition}, we may rewrite $e_{\delta,N}$ as
\begin{equation}
\label{eq:neumann-error-rep}
e_{\delta,N}(\xx)
=
S_\delta e_{\delta,N}(\xx)
+
\frac{\delta^{2}}{w_\delta(\xx)}
\bigl(r(\xx)+\psi(\xx)\bigr),
\quad
\psi(\xx):=-\!\int_\Omega\!\bar R_\delta(\xx,\yy)e_{\delta,N}(\yy)\,\d\yy .
\end{equation}

Differentiating \eqref{eq:neumann-error-rep} and using
$|\nabla(\delta^{2}/w_\delta)|\le C\delta$, which follows from
\eqref{eq:wdelta-grad}, together with $\delta^{2}/w_\delta\le C\delta^{2}$,
we obtain
\begin{align}
\|\nabla e_{\delta,N}\|_{L^2(\Omega)}^{2}
&\le
2\|\nabla S_\delta e_{\delta,N}\|_{L^2(\Omega)}^{2}
+
C\delta^{2}\,\|r+\psi\|_{L^2(\Omega)}^{2}
+
C\delta^{4}\,\|\nabla(r+\psi)\|_{L^2(\Omega)}^{2} .
\label{eq:grad-error-base}
\end{align}

\medskip
\noindent\emph{Estimates for $\psi$.} By \eqref{eq:kernel-est-2} and the
chain-rule estimate
\[
|\nabla_{\!\xx}\bar R_\delta(\xx,\yy)|
=\frac{|\xx-\yy|}{2\delta^2}R_\delta(\xx,\yy),
\]
together with the scaled second-moment bound of $R_\delta$,
\begin{equation}
\label{eq:psi-est}
\|\psi\|_{L^2(\Omega)}
\le
C\,\|e_{\delta,N}\|_{L^2(\Omega)},
\qquad
\|\nabla\psi\|_{L^2(\Omega)}
\le
\frac{C}{\delta}\,\|e_{\delta,N}\|_{L^2(\Omega)} .
\end{equation}
Combining \eqref{eq:rin-rbd-l2} and \eqref{eq:psi-est}, and using
$0<\delta\le 1$ to bound $\delta^{2}$ by $\delta$, we get
\begin{align}
\|r+\psi\|_{L^2(\Omega)}^{2}
&\le
C\!\left(\|r_{in}\|_{L^2(\Omega)}^{2}
+\|r_{bd}\|_{L^2(\Omega)}^{2}
+\|\psi\|_{L^2(\Omega)}^{2}\right)
\notag \\
&\le
C\!\left(
\delta\,\|u_{\mathrm{loc},N}\|_{H^3(\Omega)}^{2}
+\|e_{\delta,N}\|_{L^2(\Omega)}^{2}\right).
\label{eq:r-psi-l2}
\end{align}
Similarly, by \eqref{eq:rin-rbd-grad} and \eqref{eq:psi-est},
\begin{align}
\|\nabla(r+\psi)\|_{L^2(\Omega)}^{2}
&\le
C\!\left(\|\nabla r_{in}\|_{L^2(\Omega)}^{2}
+\|\nabla r_{bd}\|_{L^2(\Omega)}^{2}
+\|\nabla\psi\|_{L^2(\Omega)}^{2}\right)
\notag \\
&\le
C\!\left(
\frac{1}{\delta}\,\|u_{\mathrm{loc},N}\|_{H^3(\Omega)}^{2}
+\frac{1}{\delta^{2}}\,\|e_{\delta,N}\|_{L^2(\Omega)}^{2}\right).
\label{eq:r-psi-grad}
\end{align}

\medskip
\noindent\emph{Estimate for $\nabla S_\delta e_{\delta,N}$.} Testing
\eqref{eq:neumann-error-eq} with $e_{\delta,N}$ gives
$E_{\delta,N}(e_{\delta,N})=(r,e_{\delta,N})_{L^2(\Omega)}$. Splitting $r$ as
above, applying Cauchy--Schwarz to $(r_{in},e_{\delta,N})$ and using the
dual estimate \eqref{eq:rbd-dual} for $(r_{bd},e_{\delta,N})$,
\begin{align}
\|\nabla S_\delta e_{\delta,N}\|_{L^2(\Omega)}^{2}
&\le
C\,E_{\delta,N}(e_{\delta,N})
=
C\,(r,e_{\delta,N})_{L^2(\Omega)}
\notag \\
&\le
C\,\|r_{in}\|_{L^2(\Omega)}\|e_{\delta,N}\|_{L^2(\Omega)}
+
C\delta\,\|u_{\mathrm{loc},N}\|_{H^3(\Omega)}\,
\|e_{\delta,N}\|_{H^1(\Omega)}
\notag \\
&\le
C\delta\,\|u_{\mathrm{loc},N}\|_{H^3(\Omega)}\,
\|e_{\delta,N}\|_{H^1(\Omega)} ,
\label{eq:Sdelta-error-est}
\end{align}
where in the last step we used \eqref{eq:rin-rbd-l2} and
$\|e_{\delta,N}\|_{L^2(\Omega)}\le \|e_{\delta,N}\|_{H^1(\Omega)}$.

Combining \eqref{eq:grad-error-base},
\eqref{eq:r-psi-l2}, \eqref{eq:r-psi-grad} and
\eqref{eq:Sdelta-error-est}, we obtain
\begin{align}
\|\nabla e_{\delta,N}\|_{L^2(\Omega)}^{2}
&\le
C\delta\,\|u_{\mathrm{loc},N}\|_{H^3(\Omega)}\,
\|e_{\delta,N}\|_{H^1(\Omega)}
\notag \\
&\quad
+
C\delta^{3}\,\|u_{\mathrm{loc},N}\|_{H^3(\Omega)}^{2}
+
C\delta^{2}\,\|e_{\delta,N}\|_{L^2(\Omega)}^{2} .
\label{eq:grad-error-combined}
\end{align}

By the same argument as in \eqref{eq:Sdelta-error-est} together with the
coercivity \eqref{eq:neumann-coercivity},
\begin{equation}
\label{eq:l2-error-est}
\|e_{\delta,N}\|_{L^2(\Omega)}^{2}
\le
C\,E_{\delta,N}(e_{\delta,N})
\le
C\delta\,\|u_{\mathrm{loc},N}\|_{H^3(\Omega)}\,
\|e_{\delta,N}\|_{H^1(\Omega)} .
\end{equation}

Adding \eqref{eq:grad-error-combined} and \eqref{eq:l2-error-est}, and using
$\|e_{\delta,N}\|_{L^2(\Omega)}^{2}\le\|e_{\delta,N}\|_{H^1(\Omega)}^{2}$
to absorb the last term in \eqref{eq:grad-error-combined} for $\delta$
sufficiently small, we obtain
\begin{equation}
\label{eq:h1-error-pre-young}
\|e_{\delta,N}\|_{H^1(\Omega)}^{2}
\le
C\delta\,\|u_{\mathrm{loc},N}\|_{H^3(\Omega)}\,
\|e_{\delta,N}\|_{H^1(\Omega)}
+
C\delta^{3}\,\|u_{\mathrm{loc},N}\|_{H^3(\Omega)}^{2} .
\end{equation}
Applying Young's inequality to the first term on the right of
\eqref{eq:h1-error-pre-young} gives
\begin{equation}
\label{eq:h1-error-final}
\|u_{\mathrm{loc},N}-u_{\delta,N}\|_{H^1(\Omega)}
=
\|e_{\delta,N}\|_{H^1(\Omega)}
\le
C\delta\,\|u_{\mathrm{loc},N}\|_{H^3(\Omega)} .
\end{equation}

Finally, by the standard elliptic regularity estimate for the local Neumann
problem \eqref{eq:neumann-local},
\[
\|u_{\mathrm{loc},N}\|_{H^3(\Omega)}
\le
C\!\left(\|f\|_{H^1(\Omega)}
+\|g\|_{H^{3/2}(\partial\Omega)}\right),
\]
hence
\[
\|u_{\mathrm{loc},N}-u_{\delta,N}\|_{H^1(\Omega)}
\le
C\delta\!\left(\|f\|_{H^1(\Omega)}
+\|g\|_{H^{3/2}(\partial\Omega)}\right),
\]
which completes the proof of Proposition~\ref{prop:neumann-locallimit}.
\qed

\bibliographystyle{abbrv}

\bibliography{ref}

\end{document}